\documentclass[10pt]{article}
\usepackage{verbatim}
\usepackage[multiple]{footmisc}
\usepackage{amsfonts,color}
\usepackage{longtable}
\usepackage{makecell, multirow, tabularx}
\usepackage{amsmath,amssymb}
\usepackage{adjustbox}
\usepackage{epsfig}
\usepackage{lscape}
\usepackage{amssymb}
\usepackage{graphicx}
\usepackage[utf8]{inputenc}	
\usepackage{amsthm,amssymb}
\usepackage{aligned-overset}
\usepackage{amsfonts}
\usepackage[english]{babel}
\usepackage{dsfont}
\usepackage{bbm}
\usepackage[T1]{fontenc}
\usepackage{colonequals}
\usepackage{geometry}
\usepackage{mathtools}
\usepackage{csquotes}
\usepackage{url}
\usepackage{lipsum}
\usepackage{array}
\usepackage{tablefootnote}
\usepackage{wrapfig}
\usepackage[utf8]{inputenc}
\usepackage{cancel}
\usepackage[dvipsnames]{xcolor}
\usepackage{accents}
\usepackage{marvosym}
\usepackage{centernot}
\usepackage{tikz}
\usepackage{nicematrix}
\usepackage[new]{old-arrows}
\usepackage{titling}
\usepackage{multicol}

\usepackage{amsthm}
\newtheoremstyle{mystyle}
{}
{}
{}
{}
{\bfseries}
{.}
{ }
{}
\theoremstyle{plain}
\theoremstyle{mystyle}
\newtheorem{thm}{Theorem}[section]
\theoremstyle{plain}
\theoremstyle{mystyle}
\newtheorem{prop}[thm]{Proposition}
\theoremstyle{plain}
\theoremstyle{mystyle}
\newtheorem{cor}[thm]{Corollary}
\theoremstyle{plain}
\theoremstyle{mystyle}
\newtheorem{lem}[thm]{Lemma}
\theoremstyle{plain}
\theoremstyle{mystyle}

\theoremstyle{plain}
\theoremstyle{mystyle}

\theoremstyle{plain}
\theoremstyle{mystyle}
\newtheorem{rem}[thm]{Remark}

\theoremstyle{plain}
\theoremstyle{mystyle}
\newtheorem{definition}[thm]{Definition}

\newcommand{\id}{{\boldsymbol{\mathbbm{1}}}}

\allowdisplaybreaks[1]
\makeindex
\newcommand{\R}{\mathbb{R}}

\newcommand{\norm}[1]{\lVert #1 \rVert}

\newcommand{\SO}{{\rm SO}}
\newcommand{\OO}{{\rm O}}
\newcommand{\DD}{\mathrm{D}}
\newcommand{\WW}{\mathrm{W}}
\newcommand{\innerproduct}[1]{\langle #1 \rangle}
\newcommand{\iprod}{\innerproduct}

\newcommand*\dif{\mathop{}\!\mathrm{d}}

\newlength{\dhatheight}

\DeclareMathOperator{\diag}{diag}
\DeclareMathOperator{\Sym}{Sym}
\DeclareMathOperator{\Cof}{Cof}
\DeclareMathOperator{\dev}{dev}

\DeclareMathOperator{\sym}{sym}

\DeclareMathOperator{\tr}{tr}

\DeclareMathOperator{\ZJ}{ZJ}

\def\barr{\begin{array}}
\def\tr{\textnormal{tr}}

\def\sk{\textnormal{skew}}

\def\dd{\displaystyle}

\def\barr{\begin{array}}
\def\earr{\end{array}}
\def\becn{\begin{equation*}}
\def\endec{\end{equation}}
\def\endecn{\end{equation*}}
\def\C{\mathbb{C}}

\makeatletter
\let\@fnsymbol\@arabic
\makeatother
\numberwithin{equation}{section}

\global\long\def\pD{\textrm D}

\newcommand{\normb}[1]{\bigl\lVert #1 \bigr\rVert}

\newcommand{\scal}[2]{\left\langle#1,#2\right\rangle}

\newcommand{\scalb}[2]{\big\langle#1,#2\big\rangle}
\newcommand{\scalB}[2]{\Big\langle#1,#2\Big\rangle}

\newcommand{\parton}[1]{\left(#1\right)}

\newcommand{\partonb}[1]{\big(#1\big)}
\newcommand{\partonB}[1]{\Big(#1\Big)}
\newcommand{\partonbb}[1]{\bigg(#1\bigg)}
\newcommand{\partonBB}[1]{\Bigg(#1\Bigg)}

\newcommand{\parqB}[1]{\Big[#1\Big]}
\newcommand{\parqbb}[1]{\bigg[#1\bigg]}

\newcommand{\grafb}[1]{\big\{#1\big\}}
\newcommand{\grafB}[1]{\Big\{#1\Big\}}

\global\long\def\bR{\mathbb{R}}

\global\long\def\hyp{\textrm{hyp}}
\global\long\def\ela{\textrm{ela}}
\title{Hypo-elasticity and logarithmic strain}
\makeatletter
\renewcommand*{\@fnsymbol}[1]{\ifcase#1\or1\or2\or3\or4\or5\or6\or7\or$\ast$\else\fi}
\makeatother
\begin{document}
\title{A constitutive condition for idealized isotropic Cauchy elasticity involving the logarithmic strain}
 \date{\it Dedicated to Professor J\"org Schr\"oder on the occasion of his 60th anniversary  }
\author{
Marco Valerio d'Agostino\thanks{
Marco Valerio d'Agostino, GEOMAS, INSA-Lyon, Université de Lyon, 20 avenue Albert Einstein, 69621, Villeurbanne cedex, France, email: marco-valerio.dagostino@insa-lyon.fr
}, \qquad
Sebastian Holthausen\thanks{
Sebastian Holthausen, University of Duisburg-Essen, Chair for Nonlinear Analysis and Modelling,  Faculty of Mathematics, Thea-Leymann-Stra{\ss}e 9,
D-45127 Essen, Germany, email: sebastian.holthausen@uni-due.de
}, \qquad
Davide Bernardini\thanks{
Davide Bernardini, Department of Structural and Geotechnical Engineering, Sapienza University of Rome, Rome, Italy, e-mail: davide.bernardini@uniroma1.it
}, \\[0.8em]
Adam Sky\thanks{
Adam Sky, Institute of Computational Engineering and Sciences, Department of Engineering, Faculty of Science, Technology and Medicine, University of Luxembourg, 6 Avenue de la Fonte, L-4362 Esch-sur-Alzette, Luxembourg, e-mail: adam.sky@uni.lu
}, \qquad
Ionel-Dumitrel Ghiba\thanks{
Ionel-Dumitrel Ghiba, Alexandru Ioan Cuza University of Ia\c si, Department of Mathematics,  Blvd. Carol I, no. 11, 700506 Ia\c si, Romania;  Octav Mayer Institute of Mathematics
of the Romanian Academy, Ia\c si Branch, 700505 Ia\c si, email: dumitrel.ghiba@uaic.ro
}, \qquad
Robert J. Martin\thanks{
Robert J. Martin,  Lehrstuhl für Nichtlineare Analysis und Modellierung, Fakultät für Mathematik, Universität Duisburg-Essen, Thea-Leymann Str. 9, 45127 Essen, Germany, email: robert.martin@uni-due.de
} \\[0.6em]
and \\[0.6em]
Patrizio Neff\thanks{
Patrizio Neff, University of Duisburg-Essen, Head of Chair for Nonlinear Analysis and Modelling, Faculty of Mathematics, Thea-Leymann-Stra{\ss}e 9,
D-45127 Essen, Germany, email: patrizio.neff@uni-due.de}
}
\maketitle
\vspace{-0,6cm}
\begin{abstract}
\noindent Following Hill and Leblond, the aim of our work is to show, for isotropic nonlinear elasticity, a relation between the corotational Zaremba-Jaumann objective derivative of the Cauchy stress $\sigma$, i.e.
	\begin{align*}
	\frac{\DD^{\ZJ}}{\DD t}[\sigma] = \frac{\dif}{\dif t}[\sigma] - W \, \sigma + \sigma \, W, \qquad W = \sk(\dot F \, F^{-1})
	\end{align*}
and a constitutive requirement involving the logarithmic strain tensor. Given the deformation tensor \break $F = \DD \varphi$, the left Cauchy-Green tensor $B = F \, F^T$, and the strain-rate tensor $D = \sym(\dot F \, F^{-1})$, we show that
	\begin{equation}
	\label{eqCPSdef}
	\begin{alignedat}{2}
	   \forall \,D\in\Sym(3) \! \setminus \! \{0\}:
	   ~
	   \scal{\frac{\DD^{\ZJ}}{\DD t}[\sigma]}{D} > 0
	   \quad
	   &\iff
	   \quad
	   \log B
	   \longmapsto
	   \widehat\sigma(\log B)
	   \;\textrm{is strongly Hilbert-monotone} \\
		&\iff \quad \sym \DD_{\log B} \widehat \sigma(\log B) \in \Sym^{++}_4(6) \quad \text{(TSTS-M$^{++}$)},
	\end{alignedat} \tag{1}
	\end{equation}
where $\Sym^{++}_4(6)$ denotes the set of positive definite, (minor and major) symmetric fourth order tensors. We call the first inequality of \eqref{eqCPSdef} ``corotational stability postulate'' (CSP), a novel concept, which implies the \textbf{T}rue-\textbf{S}tress \textbf{T}rue-\textbf{S}train strict Hilbert-\textbf{M}onotonicity (TSTS-M$^+$) for $B \mapsto \sigma(B) = \widehat \sigma(\log B)$, i.e.
\begin{equation*}
	\scal{\widehat\sigma(\log B_1)-\widehat\sigma(\log B_2)}{\log B_1-\log B_2} > 0
	\qquad
	\forall \, B_1\neq B_2\in\Sym^{++}(3) \, .
\end{equation*}
A similar result, but for the Kirchhoff stress $\tau = J \, \sigma$ has been shown by Hill as early as 1968. Leblond translated this idea to the Cauchy stress $\sigma$ but only for the hyperelastic case. In this paper we expand on the ideas of Hill and Leblond, extending Leblonds calculus to the Cauchy elastic case. \\
\\
\textbf{Keywords:} nonlinear elasticity, hyperelasticity, rate-formulation, Cauchy-elasticity, material stability, Hill's method of Lagrangian axis, corotational derivatives, constitutive inequalities, logarithmic strain, stress increases with strain \\
\\
\textbf{Mathscinet} classification
15A24, 73G05, 73G99, 74B20
\end{abstract}
\clearpage
\tableofcontents
\section{Introduction}
This contribution builds upon the work of Hill \cite{hill1968constitutivea,hill1968constitutiveb, hill1978}. In search for reasonable constitutive assumptions to be placed on the nonlinear elasticity law that maps the left Cauchy-Green tensor $B = F \, F^T$ to the Cauchy stress tensor $B \mapsto \sigma(B)$ such that \textbf{physically reasonable response} is ensured, Hill considered
	\begin{align}\label{eq2}
	\langle \frac{\DD^{\sharp}}{\DD t}[\tau] ,D\rangle>0 \qquad \text{for all}\quad D\neq 0
	\end{align}
for the spatial Kirchhoff stress tensor $\tau = J \, \sigma$ and a specific subclass of objective derivatives, amongst which was the Zaremba-Jaumann rate \cite{jaumann1905, jaumann1911geschlossenes, zaremba1903forme}. He continued to examine the consequences of imposing \eqref{eq2} for the constitutive law $B\mapsto \tau(B)$ and obtained with his complicated method of Lagrangian axes for the Zaremba-Jaumann rate the result that $\tau(B)=\widehat{\tau}(\log B)$ satisfies the condition
	\begin{align} \label{eq2.1}
	\DD_{\log B} \widehat \tau(\log B) \in \Sym^{++}_4(6) \qquad \text{if and only if \eqref{eq2} is satisfied,}
	\end{align}
which itself implies the strict monotonicity of $\widehat \tau(\log B)$ in $\log B$, i.e.
	\begin{align}\label{eq3}
	\langle \widehat{\tau}(\log B_1)-\widehat{\tau}(\log B_2), \log B_1-\log B_2\rangle>0 \quad \forall \, B_1,B_2\in {\rm Sym}^{++}(3), \quad B_1\neq B_2.
	\end{align}
Inequality \eqref{eq3} is known as {\bf Hill's inequality} \cite{sidoroff1974restrictions}. A prominent constitutive law which satisfies \eqref{eq3} is the Hencky energy \cite{Neff_Osterbrink_Martin_Hencky13, NeffGhibaLankeit, NeffGhibaPoly} where
	\begin{equation}
	\label{eq4}
	\begin{alignedat}{2}
	\WW_{\text{Hencky}}(F) = \widehat \WW(\log V) &= \mu \, \norm{\log V}^2 + \frac{\lambda}{2} \, \tr^2(\log V), \\
	\tau = \DD_{\log V} \widehat \WW(\log V) = 2\,\mu\, \log V+\lambda\, \tr(\log V)\cdot \id &= \mu \, \log B + \frac{\lambda}{2} \, \tr(\log B) \cdot \id, \qquad \mu > 0, \quad 2 \, \mu + 3 \, \lambda > 0.
	\end{alignedat}
	\end{equation}
Unfortunately, while the Hencky model is arguably the best nonlinear elasticity model \cite{anand1979} with only two material parameters for moderate strains, \eqref{eq4} gets physically inadequate for large strains as e.g., the volumetric response is non-monotone and the model is not LH-elliptic (first shown in \cite{neffphd}, see also \cite{martin2019} and \cite{xiao2001}).

Leblond \cite{leblond1992constitutive} took up Hill's development but confined himself to the hyperelastic case together with the Zaremba-Jaumann rate, imposing \eqref{eq2} for the Cauchy stress $\sigma$. He obtained similarly to Hill for \newline $\sigma(B)=\widehat{\sigma}(\log B)$
	\begin{align} \label{eq4.1}
	\langle \frac{\DD^{\ZJ}}{\DD t}[\sigma] ,D\rangle>0 \qquad \text{for all}\quad D\neq 0 \qquad \iff \qquad \DD_{\log B} \widehat \sigma(\log B) \in \Sym^{++}_4(6),
	\end{align}
which again implies strict monotonicity of $\widehat \sigma(\log B)$ in $\log B$, i.e.
	\begin{align}\label{eq5}
	\langle \widehat{\sigma}(\log B_1)-\widehat{\sigma}(\log B_2), \log B_1-\log B_2\rangle>0, \qquad \forall \, B_1,B_2 \in {\rm Sym}^{++}(3), \; B_1\neq B_2.
	\end{align}
Since $\frac{\DD^{\rm ZJ}}{\DD t}[\sigma]$ is a corotational rate, we will refer to the requirement
	\begin{align}
	\langle \frac{\DD^{\rm ZJ}}{\DD t}[\sigma],D\rangle>0 \quad \forall \, D \in \Sym(3) \! \setminus \! \{0\} \qquad \text{as {\bf corotational stability postulate (CSP)}.}
	\end{align}
The right member of \eqref{eq4.1} has already been introduced as the true-stress true-strain monotonicity condition TSTS-M$^{++}$ \cite{jog2013conditions, NeffGhibaLankeit}. The goal of this paper is to revisit Leblond's result characterizing CSP and to extend it to the Cauchy-elastic case, being careful to give full proofs. As a side product we will also prove rigorously Hill's equivalence \eqref{eq2.1} in the Appendix \ref{appendixhill}.
\section{The corotational stability postulate in the hyperelastic case} \label{sec:leblond}
Before we expand the argument from Leblond \cite{leblond1992constitutive} for the hyperelastic case, we recall the definition of the \textbf{Zaremba-Jaumann} derivative using the spin given by the vorticity tensor $W = \sk \, L, L = \dot F \, F^{-1}$,
	\begin{align}
	\label{ZJrate01}
	\boxed{\frac{\DD^{\ZJ}}{\DD t}[\sigma] := \frac{\dif}{\dif t}[\sigma] + \sigma \, W - W \, \sigma = Q \, \frac{\dif}{\dif t}[Q^T \, \sigma \, Q] \, Q^T  \quad \text{for $Q(t) \in \OO(3)$ so that} \;  W =\dot{Q} \, Q^T\in \mathfrak{so}(3).}
	\end{align}
Here, $\frac{\dif}{\dif t}[\sigma]$ denotes the material or substantial time-derivative. For more information on the notation, we refer to the Appendix \ref{appendixnotation}. \\
\\
With formula \eqref{ZJrate01} we prove next that 
	\begin{align}
	\label{eqmarco01}
	\langle \frac{\DD^{\ZJ}}{\DD t}[\sigma], D \rangle > 0 \quad \forall \, D \neq 0 \qquad \iff \qquad  \sym\pD_{\log\lambda}\widehat\sigma(\log \lambda)\in\Sym^{++}(3),
	\end{align}
if $\sigma$ derives from hyperelasticity. Here, the expression $\widehat \sigma(\log \lambda)$ denotes the vector of principal stresses of $\widehat \sigma(\log V)$ (i.e.~the vector consisting of the eigenvalues of $\widehat \sigma(\log V)$). Due to isotropy we have the equivalence (cf.~Remark \ref{remA5} in the Appendix)
	\begin{align}
	 \sym\pD_{\log\lambda}\widehat\sigma(\log \lambda)\in\Sym^{++}(3) \qquad \iff \qquad \sym \DD_{\log B} \widehat \sigma(\log B) \in \Sym^{++}_4(6),
	\end{align}
which can be combined with \eqref{eqmarco01} to obtain the sought-after relation
	\begin{align}
	\langle \frac{\DD^{\ZJ}}{\DD t}[\sigma], D \rangle > 0 \quad \forall \, D \neq 0 \qquad \iff \qquad \sym \DD_{\log B} \widehat \sigma(\log B) \in \Sym^{++}_4(6).
	\end{align}
The proof is partitioned into several technical steps.
\subsection{Corotational stability postulate in terms of the quadratic form $Q_{\hyp}$} \label{sec3.01}
In this section we demonstrate how the corotational stability postulate (CSP) in terms of the Zaremba-Jaumann rate $\langle\frac{\DD^{\ZJ}}{\DD t}[\sigma],D\rangle$ can be expressed as a quadratic form $Q_{\hyp}(\dot E)$, where $E$ is the Green-Lagrange strain tensor $E:= \frac12 \, (C - \id)$. \\
\\
In a first step we use the defining relations
	\begin{align}\label{eq.sigma}
	    \frac{\DD^{\ZJ}}{\DD t}[\sigma] 
	    = 
	    \frac{\dif}{\dif t}[\sigma] 
	    + 
	    \sigma \, W - W \, \sigma 
	    \qquad 
	    \text{and} 
	    \qquad 
	    \sigma 
	    = 
	    \frac{1}{J} \, F \, S_2 \, F^T
	    =
	    \frac{1}{J} \, S_1 \, F^T
	    ,
	\end{align}
where $S_1= \DD_F \WW(E)$ and $S_2 = \DD_E \WW(E)$ denote the first and second Piola-Kirchhoff stress tensor respectively.\footnote
{
Observe the notational difference between the elastic energy $\WW(F)$ and the vorticity tensor $W = \sk L$.
}
Upon additionally setting $L = D + W$, with $D \in \Sym(3)$ and $W \in \mathfrak{so}(3)$ we then obtain
	\begin{equation}
	\label{eqfirsteq01}
	\begin{alignedat}{2}
	\frac{\DD^{\ZJ}}{\DD t}[\sigma] &= \frac{\dif}{\dif t}[\sigma] + [\sigma \, W - W \, \sigma] = \frac{\dif}{\dif t}[J^{-1} \, F \, S_2 \, F^T] + [\sigma \, W - W \, \sigma] \\
	&= [-J^{-2} \, \dot{J} \, F \, S_2 \, F^T] + J^{-1}(\dot{F} \, S_2 \, F^T + F \, \dot{S_2} \, F^T + F \, S_2 \, \dot{F}^T) + [\sigma \, W - W \, \sigma] \\
	&= [-J^{-2} \, \dot{J} \, F \, S_2 \, F^T + J^{-1} \, F \, \dot{S_2} \, F^T] + J^{-1} (\dot{F} \, S_2 \, F^T + F \, S_2 \, \dot{F}^T) + [\sigma \, W - W \, \sigma] \\
	&= [-J^{-2} \, \dot{J} \, F \, S_2 \, F^T + J^{-1} \, F \, \dot{S_2} \, F^T] + J^{-1}(\dot{F} \, F^{-1} \, F \, S_2 \, F^T + F \, S_2 \, F^T \, F^{-T} \, \dot{F}^T) + [\sigma \, W - W \, \sigma] \\
	&= [-J^{-2} \, \dot{J} \, F \, S_2 \, F^T + J^{-1} \, F \, \dot{S_2} \, F^T] + L \, \sigma + \sigma \, L^T + [\sigma \, W - W \, \sigma] \\
	&= [-J^{-2} \, \dot{J} \, F \, S_2 \, F^T + J^{-1} \, F \, \dot{S_2} \, F^T] + (L - W) \, \sigma + \sigma \, (L^T + W) \\
	&= [-J^{-2} \, \dot{J} \, F \, S_2 \, F^T + J^{-1} \, F \, \dot{S_2} \, F^T] + D \, \sigma + \sigma \, D \\
	&= \frac{1}{J} [F \, \dot{S_2} \, F^T + D \, F \, S_2 \, F^T + F \, S_2 \, F^T \, D - \frac{\dot{J}}{J} \, F \, S_2 \, F^T].
	\end{alignedat}
	\end{equation}
This expression may now be used to reformulate the corotational stability postulate $\langle \frac{\DD^{\ZJ}}{\DD t}[\sigma],D \rangle > 0$. \break Observing the additional identities
	\begin{equation}
	\begin{alignedat}{2}
	\frac{\dot{J}}{J} &= \tr (D),
	\qquad
	    \dot{E} 
	    = 
	    \frac{\dif}{\dif t}\left[\frac12 \, (F^T \, F - \id)\right] = \frac12 \, (\dot{F}^T \, F + F^T \, \dot{F}),  \\
	    D
	    &= 
	    \frac12 \, (\dot{F} \, F^{-1} + F^{-T} \, \dot{F}^T) 
	    \quad
	    \iff 
	    \quad 
	    F^{-T} \, \dot{E} \, F^{-1} = D,
	\end{alignedat}
	\end{equation}
we obtain
	\begin{align}
	J \, \left\langle \frac{\DD^{\ZJ}}{\DD t}[\sigma] , D \right\rangle &= \langle F \, \dot{S_2} \, F^T + D \, F \, S_2 \, F^T + F \, S_2 \, F^T \, D - \frac{\dot{J}}{J} \, F \, S_2 \, F^T , D \rangle \notag \\
	&= \langle F \, \dot{S_2} \, F^T + F^{-T} \, \dot{E} \, \underbrace{F^{-1} \, F}_{\id} \, S_2 \, F^T + F \, S_2 \, \underbrace{F^T \, F^{-T}}_{\id} \, \dot{E} \, F^{-1} - \tr(D) \, F \, S_2 \, F^T , F^{-T} \, \dot{E} \, F^{-1} \rangle \notag \\
	&= \langle F \, \dot{S_2} \, F^T , F^{-T} \, \dot{E} \, F^{-1} \rangle - \langle \underbrace{F^{-T} \, \dot{E} \, F^{-1} , \id \rangle}_{\tr(D)} \, \langle F \, S_2 \, F^T , F^{-T} \, \dot{E} \, F^{-1} \rangle \notag \\
	&\quad \, + \langle F^{-T} \, \dot{E} \, S_2 \, F^T , F^{-T} \, \dot{E} \, F^{-1} \rangle + \langle F \, S_2 \, \dot{E} \, F^{-1} , F^{-T} \, \dot{E} \, F^{-1} \rangle \notag \\
	&= \langle \dot{S_2} , \dot{E} \rangle - \langle \dot{E} , \underbrace{F^{-1} \, F^{-T}}_{= \; (F^T F)^{-1}} \rangle \, \langle S_2 , \dot{E} \rangle \\
	&\quad \, + \langle \underbrace{F^{-1} \, F^{-T}}_{= \; (F^T F)^{-1}} \, \dot{E} \, S_2 , \dot{E} \, \underbrace{F^{-1} \, F}_{\id} \rangle + \langle \underbrace{F^{-1} \, F}_{\id} \, S_2 \, \dot{E} , \dot{E} \, \underbrace{F^{-1} \, F^{-T}}_{= \; (F^T F)^{-1}} \rangle \notag \\
	&= \langle \dot{S_2} , \dot{E} \rangle - \langle \dot{E} , C^{-1} \rangle \, \langle S_2 , \dot{E} \rangle + \langle C^{-1} \, \dot{E} \, S_2 \, \dot{E}, \id \rangle + \langle C^{-1} \, \dot{E} \, S_2 \, \dot{E} , \id \rangle \notag \\
	&= \langle \dot{S_2} , \dot{E} \rangle + 2 \, \langle C^{-1} \, \dot{E} \, S_2 \, \dot{E} , \id \rangle - \langle C^{-1}, \dot{E} \rangle \, \langle S_2 , \dot{E} \rangle \notag \\
	&= \langle \dot{S_2} , \dot{E} \rangle + 2 \, \langle C^{-1} \, \dot{E} , \dot{E} \, S_2 \rangle - \langle C^{-1}, \dot{E} \rangle \, \langle S_2 , \dot{E} \rangle \notag \\
	&= 
	\langle \pD_E S_2(E) . \dot{E}, \dot{E} \rangle 
	+ 
	2 \, \langle C^{-1} \, \dot{E} , \dot{E} \, S_2 \rangle - \langle C^{-1} , \dot{E} \rangle \, \langle S_2 , \dot{E} \rangle > 0, \notag
	\end{align}
for all non-zero rates $\dot{E} \in \Sym(3)$. Since 
	\begin{align}
	S_2 = \pD_E \WW(E) = 2 \, \pD_C \widetilde \WW(C) \qquad \text{and} \qquad \dot S_2(x,t) = \frac{\dif}{\dif t} \, \pD_E \WW(E(x,t)) = \pD_E^2 \WW(E(x,t)) . \dot E(x,t),
	\end{align}
this is equivalent to requiring that the quadratic form
	\begin{equation}
	\label{eqleblond01}
	\begin{alignedat}{2}
	Q_{\hyp}(\dot{E}) 
	&= 
	\pD_E^2\WW(E).(\dot{E},\dot{E}) + 2 \, \langle C^{-1} \, \dot{E} , \dot{E} \cdot \pD_E \WW(E) \rangle - \langle C^{-1} , \dot{E} \rangle \, \langle \pD_E \WW(E) , \dot{E} \rangle 
	\\
	&= 
	\langle \pD_E^2\WW(E) . \dot{E},\dot{E}\rangle + 2 \, \langle C^{-1} \, \dot{E} , \dot{E} \, \pD_E \WW(E) \rangle - \langle C^{-1} , \dot{E} \rangle \, \langle \pD_E \WW(E) , \dot{E} \rangle
	\\
	&= 
	\langle \pD_E S_2(E) . \dot{E} , \dot{E} \rangle + 2 \, \langle C^{-1} \, \dot{E} , \dot{E} \, S_2(E) \rangle - \langle C^{-1} , \dot{E} \rangle \, \langle S_2(E) , \dot{E} \rangle
	\\
	&=
	\langle \dot S_2(E) , \dot{E} \rangle + 2 \, \,\textrm{tr}\,\big(C^{-1}\,\dot E\,S_2\,\dot E\big) - \langle C^{-1} , \dot{E} \rangle \, \langle S_2(E) , \dot{E} \rangle
	\end{alignedat}
	\end{equation}
must be positive-definite over the space of second rank symmetric tensors.


\subsection{Positive definiteness of the quadratic form $Q_\hyp$ in isotropic hyperelasticity}
In this section, we repropose in all detail the procedure developed by Leblond in \cite{leblond1992constitutive}, to show that the positive definiteness of the quadratic form $Q_\hyp$ is equivalent to the positive definiteness of $\sym \, \pD_{\log\lambda}\widehat\sigma(\log \lambda)$. It is assumed that the considered elastic medium is isotropic, i.e.~the elastic energy function $\WW(F)$ fulfills
\begin{align}
   \WW(F)
    =
   \WW(F \, Q)
    \qquad
    \qquad
    \forall \,
    Q\in\SO(3).
\end{align}
\begin{rem}[Notation]
We use the following convention. Whenever vector-valued functions are considered, we label the elements of a collection of functions with upper indices and the components of these vector-valued functions with lower indices. For example, considering the family $\{U^i\}_i$, we mean functions  $U^i:\Omega\subset\bR^q\to\bR^d$ and $U^i_j$  denotes the j-component of $U^i$.
\end{rem}
\noindent A key role in the upcoming analysis is the representation of occurring tensors in a basis of eigenvectors of the conjugate Green-Lagrange strain tensor $E(x,t)=\frac{1}{2}\partonb{F^T(x,t) \, F(x,t)-\id}$. Therefore, let $\{U^{i}(x,t)\}_{i=1}^3$ denote the \textbf{eigenvectors} of $E(x,t)$ and $\{e_i(x,t)\}_{i=1}^3$ the corresponding \textbf{eigenvalues}. Then we can establish this first preliminary result.

\begin{prop} \label{prop1}
	Consider a hyperelastic isotropic material with energy density $\WW(E)$. Then there exists an auxiliary function $\breve{\WW}(e)$, where $e(x,t)=\partonb{e_1(x,t),e_2(x,t),e_3(x,t)}$ is the vector field of the eigenvalues of $E(x,t)$, such that
	\begin{align}
	     \WW(E(x,t))
	     =
	     \breve{\WW}(e(x,t))
	     \qquad
	     \qquad
	     \forall \,
	     (x,t)\in\Omega\times[0,T).
	\end{align}
\end{prop}
\begin{proof}
Considering the stored energy density $\WW(E)$ as a function $\WW(E) = \overline \WW(C)$ of the Cauchy strain tensor $C$, the isotropy requirement reads as \cite{munch2018}
\begin{align}
    \overline \WW(C)
    =
    \overline \WW(Q\,C\,Q^T)
    \qquad
    \qquad
    \forall \,
    Q\in\SO(3).
\end{align}
Since $E=\frac{1}{2} (C-\id) \; \iff \; C = 2 \, E + \id$ and $\WW(E)=\overline \WW(C)$, we obtain
\begin{align}
    \WW(E)
    =
    \overline \WW(2 \, E + \id)
    =
    \overline \WW\partonb{Q \, ( 2 \, E + \id ) Q^T }
    \qquad
    \qquad
    \forall \,
    Q\in\SO(3).
\end{align}
Additionally, since $C = F^T \, F = 2 \, E + \id\in\Sym^{++}(3)$ by the spectral theorem, there exists a couple \break $(R,\widehat D)\in\textrm{SO}(3)\times \textrm{Diag}(3)$, such that $2 \, E + \id = R^T\, \widehat D\,R$ and the diagonal elements of $\widehat D$ are the eigenvalues of  $2 \, E + \id$. Therefore
\begin{align}
    \WW(E)
    =
    \overline \WW(2 \, E + \id)
    =
    \overline \WW\partonb{Q^T \, ( 2 \, E + \id ) Q }
    =
    \overline \WW\partonb{Q^T  R^T  \widehat D \, R \, Q }
    =
    \overline \WW\partonb{(R\,Q)^T  \widehat D \, R \, Q }
    \qquad
    \forall \,
    Q\in\SO(3),
\end{align}
with an orthogonal matrix $\widetilde{Q} = R \, Q$, so that by the isotropy
\begin{align}
    \WW(E)
    =
    \overline \WW(2 \, E + \id)
    =
    \overline \WW\partonb{\widetilde Q^T  \widehat D \, \widetilde Q }
    =
    \overline \WW (D) = \overline \WW(\diag(2 \, e_1 + 1, 2 \, e_2 + 1, 2 \, e_3 + 1)).
\end{align}
Inasmuch as the spectrum $\textrm{spt}(C) = \textrm{spt}(2 \, E + \id) = 2 \, \textrm{spt}(E) + 1$, we can finally introduce the auxiliary function $\breve \WW:\bR^3\to\bR$ of the eigenvalues of $E$, such that
\begin{align}
   \WW(E(x,t))
   &= \overline \WW (\widehat D) =
   \breve \WW(e(x,t))
   \qquad
   \qquad
   \forall (x,t)\in\Omega\times[0,T). \qedhere
\end{align}
\end{proof}
\begin{cor} \label{cor2}
The second Piola-Kirchhoff stress tensor can be expressed relative to the basis of the (unit) eigenvectors $\{U^i\}_{i=1}^3$ of $E$ as
\begin{equation}\label{eq:stress_Kirchh_dec}
	S_2(x,t)
	=
	\sum_{i=1}^3\frac{\partial\breve{\WW}}{\partial e_i}\partonb{e(x,t)} 
	\;
	U^i(x,t)\otimes U^i(x,t)
	\in\Sym(3)
\end{equation}
and hence 
\begin{align}\label{eq:der stress_Kirchh_dec}
	\dot S_2(x,t)
	=
	\sum_{i=1}^3
	\parqbb{
		\partonbb{
			\sum_{j=1}^3
			\frac{\partial^2\breve{\WW}}{\partial e_i\partial e_j} \, \dot{e}_j
		}
		\;
		U^i\otimes U^i
		+\frac{\partial\breve{\WW}}{\partial e_i}
		\partonbb{
			\dot{U}^i\otimes U^i+ U^i\otimes\dot{U}^i
		}
	}.
\end{align}
\end{cor}
\begin{proof}
This is a consequence of the fact that the auxiliary energy density $\breve \WW(e)$ can be thought as a function of the diagonal matrix $\widehat D(E) = \sum_i e_i \, U^i \otimes U^i$ given by the spectral theorem, i.e.~$\breve \WW(e) = \breve \WW(\widehat D(e))$. Therefore, setting $e_{ij}=0$ for $i\neq j$, we can introduce the function 
\begin{align}
\breve \WW(e) = \breve \WW\bigg(\sum_{i,j}e_{ij} \, U^i\otimes U^j\bigg) \qquad \text{and} \qquad \DD_e \breve \WW\bigg(\sum_{i,j}e_{ij} \, U^i\otimes U^j\bigg) &= \grafB{\frac{\partial}{\partial e_{hk}} \breve \WW\bigg(\sum_{i,j}e_{ij} \, U^i\otimes U^j\bigg)}_{h,k}. \qedhere
\end{align}
\end{proof}
\noindent In order to show that the positive definiteness of $Q_\hyp$ is equivalent to the positive definiteness of $\sym\pD_{\log\lambda}\widehat\sigma$, we first demonstrate that $Q_\hyp$ can be split into two quadratic forms, $Q_{\hyp(1)}$ and $Q_{\hyp(2)}$ respectively, such that 
\begin{align}
   Q_\hyp(\dot E)
   =
   Q_{\hyp(1)}(\dot E_{11},\dot E_{22}, \dot E_{33})
   +
   \,
   Q_{\hyp(2)}(\dot E_{12},\dot E_{23}, \dot E_{31}).
\end{align}
In other words, the quadratic form $Q_\hyp$ can be expressed as a bilinear form over $\bR^6$ by setting 
	\begin{equation}
		\widehat{Q}_\hyp(\dot E,\dot E)
		=
		\left\langle\begin{pmatrix}
				\dot E_{11} \\ \dot E_{22} \\ \dot E_{33} \\ \dot E_{23} \\ \dot E_{13} \\ \dot E_{12}
			\end{pmatrix} \!,\!
			\left(\begin{array}{@{}c@{}c|c@{}c@{}}
				&\phantom{0} &&
				\\
				& \widehat Q_{\hyp (1)} &0 &
				\\
				&\phantom{0} &&
				\\
				\hline
				&\phantom{0} &&
				\\
				&0 & \widehat Q_{\hyp (2)}&
				\\
				&\phantom{0} &&
			\end{array}\right)
			\!.\!
			\begin{pmatrix}
				\dot E_{11} \\ \dot E_{22} \\ \dot E_{33} \\ \dot E_{23} \\ \dot E_{13} \\ \dot E_{12}
			\end{pmatrix}\right\rangle,
	\end{equation}
with matrices $\widehat Q_{\hyp(1)}$ and $\widehat Q_{\hyp(2)}$ that will be defined later, so that
	\begin{equation}
		\widehat Q_\hyp(\dot E,\dot E)
		=
		\scal{\begin{pmatrix}
				\dot E_{11} \\ \dot E_{22} \\ \dot E_{33} 
		\end{pmatrix} \!\!}{ \widehat Q_{\hyp(1)} .\!\! \begin{pmatrix}
				\dot E_{11} \\ \dot E_{22} \\ \dot E_{33} 
		\end{pmatrix}}_{\mathclap{\bR^{3}}}
		+
		\scal{\begin{pmatrix}
				\dot E_{23} \\ \dot E_{13} \\ \dot E_{12}
		\end{pmatrix} \!\!}{ \widehat Q_{\hyp(2)} .\!\! \begin{pmatrix}
				\dot E_{23} \\ \dot E_{13} \\ \dot E_{12}
		\end{pmatrix}}_{\mathclap{\bR^{3}}}.
	\end{equation}
This separation allows us to establish the positive definiteness of  $Q_\hyp$ by separately studying the positive definiteness of $Q_{\hyp(1)}$ and $Q_{\hyp(2)}$. In the next step, we need to develop separately the three terms
$ 
\langle \dot S_2(E) , \dot{E} \rangle 
$, 
$
2 \, \,\textrm{tr}\,\big(C^{-1}\,\dot E\,S_2\,\dot E\big)
$ and
$
\langle C^{-1} , \dot{E} \rangle \, \langle S_2(E) , \dot{E} \rangle
$
comprising $Q_\hyp(\dot E)$. At this point it is important to remark that, in general, the vectors $\{U^i\}_{i=1}^3$ are not the eigenvectors of the eigenvalues of $\dot E$ and hence $\dot E$ is not diagonal in this basis. Thus, we begin our considerations with the following
\begin{lem}\label{lem3}
Let $\{U^i\}_{i=1}^3$ be the eigenvectors of $E$ with eigenvalues $\{e_i(x,t)\}_{i=1}^3$. Then the tensor $\dot E$ admits the representation
\begin{align}
		\dot E
		=
		\sum_{i=1}^3\partonbb{
			\dot e_i\,U^i\otimes U^i
			+
			e_i\,\dot U^i\otimes U^i
			+
			e_i\,U^i\otimes \dot U^i
		}
	    =
	    \sum_{j,k=1}^3 \dot E_{jk} \, U^{j}\otimes U^{k}, 
	\end{align}
with
\begin{align}
\label{eq.3}
\dot E_{jk} = \left\{
	\begin{array}{lr}
		\dot e_j &\qquad \text{if} \quad j = k,  \\
		(e_j-e_k) \, \scalb{\dot U^j}{U^k} & \qquad \text{if} \quad j \neq k.
	\end{array} \right.
\end{align}
\end{lem}
\begin{proof}
We begin the calculation of the components $\{\dot E_{jk}\}_{j,k=1}^3$ of $\dot E$ in the basis $\{U^i\}_{i=1}^3$ by differentiating
	\begin{align}
		E 
		= 
		\sum_i e_i \, U^i \otimes U^i \quad \implies \quad \dot{E} 
		= 
		\sum_i(\dot{e}_i \, U^i \otimes U^i + e_i \, \dot{U}^i \otimes U^i + e_i \, U^i \otimes \dot{U}^i).
	\end{align}
The component $\dot{E}_{jk} = \langle \dot E.U^j,U^k\rangle$ of $\dot E$ can be consequently expressed as 
	\begin{align}
		\dot  E_{jk}&=\scal{\dot  E.U^j}{U^k}
		=
		\scal{\parqbb{\sum_{i=1}^3\partonbb{
					\dot e_i\,U^i\otimes U^i
					+
					e_i\,\dot U^i\otimes U^i
					+
					e_i\,U^i\otimes \dot U^i
			}}\,U^j}{U^k}
		\nonumber
		\\
		&
		=
		\scalB{\sum_{i=1}^3\partonB{\dot e_i\,U^i\otimes U^i}U^j}{U^k}
		+
		\scalB{\sum_{i=1}^3\partonB{e_i\,\dot U^i\otimes U^i}U^j}{U^k}
		+
		\scalB{\sum_{i=1}^3\partonB{e_i\,U^i\otimes \dot U^i}U^j}{U^k}
		\nonumber
		\\
		&
		=
		\scalB{\sum_{i=1}^3\underbrace{\scalb{U^i}{U^j}}_{\delta_{ij}}\dot e_i\,U^i}{U^k}
		+
		\scalB{\sum_{i=1}^3\underbrace{\scalb{U^i}{U^j}}_{\delta_{ij}}e_i\,\dot U^i}{U^k}
		+
		\scalB{\sum_{i=1}^3\scal{\dot U_i}{U_j} e_i\, U^i}{U^k}
		\\
		&=
		\scalb{\dot e_j\,U^j}{U^k}
		+
		\scalb{e_j\,\dot U^j}{U^k}
		+
		\scalB{\sum_{i=1}^3\scalb{\dot U^i}{U^j} \, e_i \, U^i}{U^k}
		\nonumber
		\\
		&
		=
		\dot e_j\,\underbrace{\scalb{U^j}{U^k}}_{=\,\delta_{jk}}
		+
		e_j\scalb{\dot U^j}{U^k}
		+
		\sum_{i=1}^3\scalb{\dot U^i}{U^j} e_i \, \delta_{ik}
		=
		\dot e_j\,\delta_{jk}
		+
		e_j\scalb{\dot U^j}{U^k}
		+
		e_k\scalb{\dot U^k}{U^j}.
		\nonumber
	\end{align}
From the orthonomality $\scal{U^k}{U^j} = \delta_{kj}$ given by the spectral theorem we see
	\begin{align}
	0
	=
	\frac{\dif}{\dif t}
	\delta_{kj}
	=
	\frac{\dif}{\dif t}\scalb{U^k}{U^j}
	=
	\scalb{\dot U^k}{U^j}
	+
	\scalb{ U^k}{\dot U^j}
	\qquad
	\Longrightarrow
	\qquad
	\scalb{\dot U^k}{U^j}
	=
	-
	\scalb{ U^k}{\dot U^j}
	\end{align}
	and obtain
	\begin{equation}\label{eq:eigen_diff}
		\dot E_{jk}
		=
		\dot e_j\,\delta_{jk}
		+
		e_j\scalb{\dot U^j}{U^k}
		+
		e_k\scalb{\dot U^k}{U^j}
		=
		\dot e_j\,\delta_{jk}
		+
		(e_j-e_k)\scalb{\dot U^j}{U^k}
		\qquad
		\forall \, 
		i,j\in\{1,2,3\}.
	\end{equation}
	In particular 
	\begin{align}
		\dot E_{jj}
		=
		\dot e_j
		\qquad
		\forall \,
		j\in\{1,2,3\}
		\qquad \text{and} \qquad
		\dot E_{jk}
		&=
		(e_j-e_k) \, \scalb{\dot U^j}{U^k}
		\qquad
		\forall \,
		j\neq k. \qedhere
	\end{align}
\end{proof}
\begin{cor}
The orthogonal decomposition of $\dot U^j$ in the orthonormal basis $\{U^k\}_{k=1}^3$ is given by
	\begin{equation}\label{eq.2}
		\dot U^j
		=
		\sum_{k\neq j}\frac{\dot E_{jk}}{e_j-e_k}U^k, \qquad \text{for} \qquad e_j \neq e_k,
	\end{equation}
where $\dot E_{jk} = 0$, if $e_j = e_k$, as can be seen from \eqref{eq:eigen_diff}.
\end{cor}
\begin{proof}
From the representation
	\begin{equation}\label{eq:1}
		\dot E_{jk}
		=
		(e_j-e_k) \, \scalb{\dot U^j}{U^k}
		\qquad
		\Longleftrightarrow
		\qquad
		\scalb{\dot U^j}{U^k}
		=
		\frac{\dot E_{jk}}{e_j-e_k}
		\qquad
		\forall 
		j\neq k
	\end{equation}
we easily infer 
	\begin{equation}\label{eq.2}
		\dot U^j
		=
		\sum_{k=1}^3
		\scalb{\dot U^j}{U^k} \, U^k
		=
		\sum_{k\neq j}\scalb{\dot U^j}{U^k}U^k
		=
		\sum_{k\neq j}\frac{\dot E_{jk}}{e_j-e_k}U^k,
	\end{equation}
where $\langle \dot U^j, U^j \rangle = 0$ since $\norm{U^j}^2 = 1$ and thus $\frac{\dif}{\dif t} \norm{U^j}^2 = 2 \, \langle \dot U^j, U^j \rangle = 0$.
\end{proof}
\noindent Next we develop the quadratic form $Q_{\hyp}$.
\begin{prop} \label{prop5}
We have
	\begin{equation}
	\label{eq.all1}
	\begin{alignedat}{2}
	\scalb{\dot S_2}{\dot E}
	&=
	\sum_{i,j=1}^3
	\frac{\partial^2\breve{\WW}}{\partial e_i\partial e_j} \, \dot{E}_{jj}\, \dot{E}_{ii}
	+
	\sum_{i\neq j}
	\frac{\frac{\partial\breve{\WW}}{\partial e_i}-\frac{\partial\breve{\WW}}{\partial e_j}}{e_i-e_j}\dot{E}_{ij}^2,
	\\
	     \scalb{C^{-1}}{\dot E}\scalb{S_2}{\dot E}
	     &=
	     \sum_{i=1}^3\frac{1}{\lambda_i^2}\frac{\partial \breve{\WW}}{\partial e_i}\dot E_{ii}^2
	     +
	     \frac{1}{2}
	     \sum_{i\neq j}
	     \partonbb{
	     	\frac{1}{\lambda_i^2}
	     	\frac{\partial \breve{\WW}}{\partial e_j}
	     	+\frac{1}{\lambda_j^2}
	     	\frac{\partial \breve{\WW}}{\partial e_i}
	     }
	     \dot E_{ii}
	     \dot{E}_{jj},
	\\
	    \tr\partonb{C^{-1}\,\dot E\,S_2\,\dot E}
	    &=
	    \sum_{i=1}^3\frac{1}{\lambda_i^2}\dot E_{ii}^2\frac{\partial \breve{\WW}}{\partial e_i}
	    +
	    \frac{1}{2}
	    \sum_{i\neq j}\partonbb{\frac{1}{\lambda_i^2}\frac{\partial \breve{\WW}}{\partial e_j}+\frac{1}{\lambda_j^2}\frac{\partial \breve{\WW}}{\partial e_i}} \dot E_{ij}^2.
	\end{alignedat}
	\end{equation}
\end{prop}
\begin{proof}
Starting with $\eqref{eq.all1}_1$,, we can use \eqref{eq.2} and \eqref{eq.3} in the expression of $\dot S_2$ from Corollary \ref{cor2}, to obtain
	\begin{align}
		\dot{S}_2
		&
		=
		\sum_{i=1}^3
		\parqbb{
			\partonbb{
				\sum_{j=1}^3
				\frac{\partial^2\breve{\WW}}{\partial e_i\partial e_j} \, \dot{e}_j
			}
			\;
			U^i\otimes U^i
			+\frac{\partial\breve{\WW}}{\partial e_i}
			\partonbb{
				\dot{U}^i\otimes U^i+ U^i\otimes\dot{U}^i
			}
		}
		\nonumber
		\\
		&
		=
		\sum_{i,j=1}^3
		\frac{\partial^2\breve{\WW}}{\partial e_i\partial e_j} \, \dot{E}_{jj}
		\;
		U^i\otimes U^i
		+
		\sum_{i\neq j}
		\frac{\frac{\partial\breve{\WW}}{\partial e_i}-\frac{\partial\breve{\WW}}{\partial e_j}}{e_i-e_j}
		\, \dot{E}_{ij}U^i\otimes U^j.
		\label{eq:Spunto}
	\end{align}
Calculating the scalar product on both sides of equation \eqref{eq:Spunto} with $\dot{E} = \sum_{h,k}\dot{E}_{hk} \, U^h \otimes U^k$ then yields
	\begin{align}
		\scalb{\dot S_2}{\dot E}
		&
		=
		\scal{\sum_{i,j=1}^3
			\frac{\partial^2\breve{\WW}}{\partial e_i\partial e_j} \, \dot{E}_{jj}
			\;
			U^i\otimes U^i
			+
			\sum_{i\neq j}
			\frac{\frac{\partial\breve{\WW}}{\partial e_i}-\frac{\partial\breve{\WW}}{\partial e_j}}{e_i-e_j}\dot{E}_{ij}U^i\otimes U^j
		}{\sum_{h,k}\dot E_{hk}U^h\otimes U^k}
		\nonumber
		\\
		&
		=
		\scal{\sum_{i,j=1}^3
			\frac{\partial^2\breve{\WW}}{\partial e_i\partial e_j} \, \dot{E}_{jj}
			\;
			U^i\otimes U^i
		}{\sum_{h,k}\dot E_{hk}U^h\otimes U^k}
		+
		\scal{
			\sum_{i\neq j}
			\frac{\frac{\partial\breve{\WW}}{\partial e_i}-\frac{\partial\breve{\WW}}{\partial e_j}}{e_i-e_j}\dot{E}_{ij}U^i\otimes U^j
		}{\sum_{h\neq k}\dot E_{hk}U^h\otimes U^k}
		\nonumber
		\\
		&
		=
		\sum_{i,j=1}^3
		\frac{\partial^2\breve{\WW}}{\partial e_i\partial e_j} \, \dot{E}_{jj}\, \dot{E}_{ii}
		+
		\sum_{i\neq j}
		\frac{\frac{\partial\breve{\WW}}{\partial e_i}-\frac{\partial\breve{\WW}}{\partial e_j}}{e_i-e_j}\dot{E}_{ij}^2.
	\end{align}
For $\eqref{eq.all1}_2$, we obtain by the representation of $S_2$ from Corollary \ref{cor2} and that of $\dot E$ from Lemma \ref{lem3},
	\begin{align}
		\scalb{S_2}{\dot E}
		&
		=
		\scal{\sum_{j=1}^3\frac{\partial \breve{\WW}}{\partial e_j}U^j\otimes U^j}{
			\sum_{j=1}^3\dot e_j U^j\otimes U^j
			+
			\sum_{j=1}^3 e_j \dot U^j\otimes U^j
			+
			\sum_{j=1}^3 e_j  U^j\otimes \dot U^j
		}
		\\
		&
		=
		\sum_{j=1}^3\frac{\partial \breve{\WW}}{\partial e_j}\underbrace{\dot e_j}_{=\,\dot E_{jj}}
		+
		\sum_{j=1}^3\frac{\partial \breve{\WW}}{\partial e_j}\underbrace{\normb{U^j}^2}_{=\,1}
		\underbrace{\scalb{\dot U^j}{U^j}}_{=\,0}
		+
		\sum_{j=1}^3\frac{\partial \breve{\WW}}{\partial e_j}\underbrace{\normb{U^j}^2}_{=\,1}
		\underbrace{\scalb{\dot U^j}{U^j}}_{=\,0}
		=
		\sum_{j=1}^3\frac{\partial \breve{\WW}}{\partial e_j}\dot{E}_{jj}.
		\nonumber
	\end{align}
To rewrite the term $\scalb{C^{-1}}{\dot E}$, let us first remark that
the eigenvalues of $C$ are given by $\{\lambda_j^2=1+2e_j\}_{j=1}^3$ because if $v\in\mathbb{R}^3$ satisfies $E\,v= e_j\,v$, then it also satisfies 
	\begin{equation}
		E\,v
		=
		\frac{1}{2} (C-\id) \, v
		= e_j \, v
		\qquad
		\Longleftrightarrow
		\qquad
		C\,v
		=
		(1+2\,e_j) \, v
		=
		\lambda_j^2 \, v.
	\end{equation}	
	Hence, from 
	\begin{align}
	C
	=
	\sum_{j=1}^3\lambda_j^2\,U^j\otimes U^j
	\qquad
	\Longleftrightarrow
	\qquad
	C^{-1}
	=
	\sum_{j=1}^3\frac{1}{\lambda_j^2}\,U^j\otimes U^j,
	\end{align}
	we obtain
	\begin{align}
		\scalb{C^{-1}}{\dot E}
		&
		=
		\scal{\sum_{j=1}^3\frac{1}{\lambda_j^2}\,U^j\otimes U^j}{
			\sum_{j=1}^3\dot e_j U^j\otimes U^j
			+
			\sum_{j=1}^3 e_j \dot U^j\otimes U^j
			+
			\sum_{j=1}^3 e_j  U^j\otimes \dot U^j
		}
		\nonumber
		\\
		&
		=
		\sum_{j=1}^3\frac{1}{\lambda_j^2}\dot e_j
		\overset{\textrm{by} \, \eqref{eq.3}}{=}
		\sum_{j=1}^3\frac{1}{\lambda_j^2}\dot E_{jj}.
	\end{align}
	Therefore
	\begin{align}
		\scalb{C^{-1}}{\dot E}\scalb{S_2}{\dot E}
		&
		=
		\parton{\sum_{i=1}^3\frac{1}{\lambda_i^2}\dot E_{ii}}
		\parton{\sum_{j=1}^3\frac{\partial \breve{\WW}}{\partial e_j}\dot{E}_{jj}}
		=
		\sum_{i=1}^3\frac{1}{\lambda_i^2}\frac{\partial \breve{\WW}}{\partial e_i}\dot E_{ii}^2
		+
		\sum_{i\neq j}
		\frac{1}{\lambda_i^2}
		\frac{\partial \breve{\WW}}{\partial e_j}
		\dot E_{ii}
		\dot{E}_{jj}
		\nonumber
		\\
		&
		=
		\sum_{i=1}^3\frac{1}{\lambda_i^2}\frac{\partial \breve{\WW}}{\partial e_i}\dot E_{ii}^2
		+
		\frac{1}{2}
		\sum_{i\neq j}
		\partonbb{
			\frac{1}{\lambda_i^2}
			\frac{\partial \breve{\WW}}{\partial e_j}
			+\frac{1}{\lambda_j^2}
			\frac{\partial \breve{\WW}}{\partial e_i}
		}
		\dot E_{ii}
		\dot{E}_{jj}.
	\end{align}
For the last term of \eqref{eq.all1}, we first observe that
   \begin{align}
        	S_2\,\dot E
        	&
        	=
        	\partonbb{\sum_{k=1}^3\frac{\partial \breve{\WW}}{\partial e_k}U^k\otimes U^k}
        	\,
        	\partonbb{
        		\sum_{i,j}\dot E_{ij}U^i\otimes U^j}
        	=
        	\sum_{i,j}\frac{\partial \breve{\WW}}{\partial e_i}\dot E_{ij}U^i\otimes U^j.
        \end{align}
This leads to
   \begin{align}
	    	\dot E \, S_2\,\dot E
	    	&
	    	=
	    	\dot E \, \partonb{ S_2\,\dot E }
	    	=
	    	\partonbb{
	    		\sum_{h,k}\dot E_{hk}U^h\otimes U^k}
	    	\partonbb{
	    		\sum_{i,j}\frac{\partial \breve{\WW}}{\partial e_i}\dot E_{ij}U^i\otimes U^j
	    	}
	    	\nonumber
	    	\\
	    	&
	    	=
	    	\sum_{h,k}\sum_{i,j}\dot E_{hk}\frac{\partial \breve{\WW}}{\partial e_i}\dot E_{ij}\delta_{ki}U^h\otimes U^j
	    	=
	    	\sum_{h,j}\partonbb{\sum_{i}\dot E_{hi}\frac{\partial \breve{\WW}}{\partial e_i}\dot E_{ij}}U^h\otimes U^j,
	    \end{align}
which finally results in
	    \begin{align}
	    	C^{-1}\dot E \, S_2\,\dot E
	    	&
	    	=
	    	C^{-1} \, \partonb{\dot E \, S_2 \, \dot E }
	    	=
	    	\partonbb{
	    		\sum_{k}\frac{1}{1+2e_k}U^k\otimes U^k}
	    	\partonbb{
	    		\sum_{h,j}\partonbb{\sum_{i}\dot E_{hi}\frac{\partial \breve{\WW}}{\partial e_i}\dot E_{ij}}U^h\otimes U^j
	    	}
	    	\nonumber
	    	\\
	    	&
	    	=
	    	\sum_{k,j}\frac{1}{1+2e_k}\partonbb{\sum_{i}\dot E_{ki}\frac{\partial \breve{\WW}}{\partial e_i}\dot E_{ij}}U^k\otimes U^j.
	    \end{align}
Thus, calculating the trace, we obtain
	    \begin{align}
	    	\tr\partonb{C^{-1}\dot E \, S_2\,\dot E}
	    	&
	    	=
	    	\tr\partonbb{
	    		\sum_{k,j}\frac{1}{1+2e_k}\partonbb{\sum_{i}\dot E_{ki}\frac{\partial \breve{\WW}}{\partial e_i}\dot E_{ij}}U^k\otimes U^j}
	    	=
	    	\sum_{i,j}\frac{1}{\lambda_j^2}\dot E_{ji}\frac{\partial \breve{\WW}}{\partial e_i}\dot E_{ij}
	    	\\
	    	&
	    	=
	    	\sum_{i}\frac{1}{\lambda_i^2}\dot E_{ii}^2\frac{\partial \breve{\WW}}{\partial e_i}
	    	+
	    	\sum_{i\neq j}\frac{1}{\lambda_i^2}\frac{\partial \breve{\WW}}{\partial e_j} \dot E_{ij}^2
	    	=
	    	\sum_{i}\frac{1}{\lambda_i^2}\dot E_{ii}^2\frac{\partial \breve{\WW}}{\partial e_i}
	    	+
	    	\frac{1}{2}
	    	\sum_{i\neq j}\partonbb{\frac{1}{\lambda_i^2}\frac{\partial \breve{\WW}}{\partial e_j}+\frac{1}{\lambda_j^2}\frac{\partial \breve{\WW}}{\partial e_i}} \dot E_{ij}^2.
	    	\qedhere
	    \end{align}
\end{proof}
Now Proposition \ref{prop5} allows us to derive explicitly the quadratic form $Q_{\hyp}(\dot E)$ by reordering the terms and regrouping them according to the groups of variables $(\dot E_{11},\dot E_{22}, \dot E_{33})$ and  $(\dot E_{12},\dot E_{23}, \dot E_{31})$. The quadratic form defined by \eqref{eqleblond01} hence becomes the sum of the two \textbf{independent} quadratic forms
\begin{equation}
\begin{alignedat}{2}
    Q_\hyp(\dot E)
    &=
	\langle \dot S_2(E) , \dot{E} \rangle + 2 \, \,\textrm{tr}\,\big(C^{-1}\,\dot E\,S_2\,\dot E\big) - \langle C^{-1} , \dot{E} \rangle \, \langle S_2(E) , \dot{E} \rangle \\
    &=
    \underbrace{Q_{\hyp(1)}(\dot E_{11},\dot E_{22}, \dot E_{33})}_{
    	\mathclap{\substack{
    	= \, Q_{\hyp(1)}(\dot e_{1},\dot e_{2}, \dot e_{3}) \\ \textrm{by} \; \eqref{eq.3}}
     }}
    +
    \;
    Q_{\hyp(2)}(\dot E_{12},\dot E_{23}, \dot E_{31}),
\end{alignedat}
\end{equation}
where
\begin{align}
	Q_{\hyp(1)}(\dot E_{11},\dot E_{22}, \dot E_{33})
	&
	=
	\sum_{i,j=1}^3 \frac{\partial^2 \breve{\WW}}{\partial e_i\partial e_j} \dot E_{ii} \dot E_{jj}
	+
	2
	\sum_{i=1}^3\frac{1}{\lambda_i^2}\frac{\partial \breve{\WW}}{\partial e_i}\dot E_{ii}^2
	-
	\sum_{i=1}^3\frac{1}{\lambda_i^2}\frac{\partial \breve{\WW}}{\partial e_i}\dot E_{ii}^2
	-
	\frac{1}{2}
	\sum_{i\neq j}
	\partonbb{
		\frac{1}{\lambda_i^2}
		\frac{\partial \breve{\WW}}{\partial e_j}
		+\frac{1}{\lambda_j^2}
		\frac{\partial \breve{\WW}}{\partial e_i}
	}
	\dot E_{ii}
	\dot{E}_{jj} \nonumber
	\\
	\label{eqhypcompare1}
	&=
	\sum_{i=1}^3\partonbb{
		\frac{\partial^2 \breve{\WW}}{\partial e_i^2}+\frac{1}{\lambda_i^2}\frac{\partial \breve{\WW}}{\partial e_i}
	}\dot E_{ii}^2
	+
	\sum_{i\neq j}
	\partonbb{
		\frac{\partial^2 \breve{\WW}}{\partial e_i\partial e_j}
		-
		\frac{1}{2}
		\partonbb{
			\frac{1}{\lambda_i^2}
			\frac{\partial \breve{\WW}}{\partial e_j}
			+\frac{1}{\lambda_j^2}
			\frac{\partial \breve{\WW}}{\partial e_i}
		}
	}
	\dot E_{ii}
	\dot{E}_{jj}
\end{align}
and 
\begin{align}
\label{eqhypcompare2}
	Q_{\hyp(2)}(\dot E_{12},\dot E_{23}, \dot E_{31})
	= 
	\sum_{i\neq j}
	\partonbb{
		\frac{\frac{\partial\breve{\WW}}{\partial e_i}-\frac{\partial\breve{\WW}}{\partial e_j}}{e_i-e_j}
		+
		\frac{1}{\lambda_i^2}\frac{\partial \breve{\WW}}{\partial e_j}+\frac{1}{\lambda_j^2}\frac{\partial \breve{\WW}}{\partial e_i}
	}
	\dot{E}_{ij}^2.
\end{align}

\begin{rem}
As indicated above, the quadratic form $Q_\hyp$ can be expressed as a bilinear form over $\bR^6$ by setting
\begin{equation}
		\widehat{Q}_\hyp(\dot E,\dot E)
		=
		\left\langle \begin{pmatrix}
				\dot E_{11} \\ \dot E_{22} \\ \dot E_{33} \\ \dot E_{23} \\ \dot E_{13} \\ \dot E_{12}
			\end{pmatrix} \!,\!
			\left(\begin{array}{@{}c@{}c|c@{}c@{}}
				&\phantom{0} &&
				\\
				& \widehat Q_{\hyp (1)} &0 &
				\\
				&\phantom{0} &&
				\\
				\hline
				&\phantom{0} &&
				\\
				&0 & \widehat Q_{\hyp (2)}&
				\\
				&\phantom{0} &&
			\end{array}\right)
			\!.\!
			\begin{pmatrix}
				\dot E_{11} \\ \dot E_{22} \\ \dot E_{33} \\ \dot E_{23} \\ \dot E_{13} \\ \dot E_{12}
			\end{pmatrix}\right\rangle,
	\end{equation}
	i.e.
	\begin{equation}
		\widehat Q_\hyp(\dot E,\dot E)
		=
		\scal{\begin{pmatrix}
				\dot E_{11} \\ \dot E_{22} \\ \dot E_{33} 
		\end{pmatrix}\!\!}{ \widehat Q_{\hyp(1)} .\!\! \begin{pmatrix}
				\dot E_{11} \\ \dot E_{22} \\ \dot E_{33} 
		\end{pmatrix}}_{\mathclap{\bR^{3}}}
		\;\;
		+
		\;\;
		\scal{\begin{pmatrix}
				\dot E_{23} \\ \dot E_{13} \\ \dot E_{12}
		\end{pmatrix}\!\!}{ \widehat Q_{\hyp(2)} .\!\! \begin{pmatrix}
				\dot E_{23} \\ \dot E_{13} \\ \dot E_{12}
		\end{pmatrix}}_{\mathclap{\bR^{3}}},
	\end{equation}
	where
	\begin{equation*}
		\widehat Q_{\hyp (1)}
		=
		\begin{pmatrix}
			\frac{\partial^2 \breve{\WW}}{\partial e_1^2}+\frac{1}{\lambda_1^2}\frac{\partial \breve{\WW}}{\partial e_1}
			&
			\frac{\partial^2 \breve{\WW}}{\partial e_1\partial e_2}
			-
			\frac{1}{2}
			\partonB{
				\frac{1}{\lambda_1^2}
				\frac{\partial \breve{\WW}}{\partial e_2}
				+\frac{1}{\lambda_2^2}
				\frac{\partial \breve{\WW}}{\partial e_1}
			}
			&
			\frac{\partial^2 \breve{\WW}}{\partial e_1\partial e_3}
			-
			\frac{1}{2}
			\partonB{
				\frac{1}{\lambda_1^2}
				\frac{\partial \breve{\WW}}{\partial e_3}
				+\frac{1}{\lambda_3^2}
				\frac{\partial \breve{\WW}}{\partial e_1}
			}
			\\[3mm]
			\frac{\partial^2 \breve{\WW}}{\partial e_2\partial e_1}
			-
			\frac{1}{2}
			\partonB{
				\frac{1}{\lambda_1^2}
				\frac{\partial \breve{\WW}}{\partial e_2}
				+\frac{1}{\lambda_2^2}
				\frac{\partial \breve{\WW}}{\partial e_1}
			}
			&
			\frac{\partial^2 \breve{\WW}}{\partial e_2^2}
			+
			\frac{1}{\lambda_2^2}
			\frac{\partial \breve{\WW}}{\partial e_2}
			&
			\frac{\partial^2 \breve{\WW}}{\partial e_2\partial e_3}
			-
			\frac{1}{2}
			\partonB{
				\frac{1}{\lambda_2^2}
				\frac{\partial \breve{\WW}}{\partial e_3}
				+\frac{1}{\lambda_3^2}
				\frac{\partial \breve{\WW}}{\partial e_2}
			}
			\\[3mm]
			\frac{\partial^2 \breve{\WW}}{\partial e_3\partial e_1}
			-
			\frac{1}{2}
			\partonB{
				\frac{1}{\lambda_1^2}
				\frac{\partial \breve{\WW}}{\partial e_3}
				+\frac{1}{\lambda_3^2}
				\frac{\partial \breve{\WW}}{\partial e_1}
			}
			&
			\frac{\partial^2 \breve{\WW}}{\partial e_3\partial e_2}
			-
			\frac{1}{2}
			\partonB{
				\frac{1}{\lambda_2^2}
				\frac{\partial \breve{\WW}}{\partial e_3}
				+\frac{1}{\lambda_3^2}
				\frac{\partial \breve{\WW}}{\partial e_2}
			}
			&
			\frac{\partial^2 \breve{\WW}}{\partial e_3^2}+\frac{1}{\lambda_3^2}\frac{\partial \breve{\WW}}{\partial e_3}
		\end{pmatrix}
	\end{equation*}
	and
	\begin{equation}
		\widehat Q_{\hyp (2)}
		=
		\begin{pmatrix}
			\frac{\frac{\partial\breve{\WW}}{\partial e_2}-\frac{\partial\breve{\WW}}{\partial e_3}}{e_2-e_3}
			+
			\frac{1}{\lambda_2^2}\frac{\partial \breve{\WW}}{\partial e_3}+\frac{1}{\lambda_3^2}\frac{\partial \breve{\WW}}{\partial e_2}
			&
			0
			&
			0
			\\[2mm]
			0
			&
			\frac{\frac{\partial\breve{\WW}}{\partial e_1}-\frac{\partial\breve{\WW}}{\partial e_3}}{e_1-e_3}
			+
			\frac{1}{\lambda_1^2}\frac{\partial \breve{\WW}}{\partial e_3}+\frac{1}{\lambda_3^2}\frac{\partial \breve{\WW}}{\partial e_1}
			&
			0
			\\[2mm]
			0
			&
			0
			&
			\frac{\frac{\partial\breve{\WW}}{\partial e_1}-\frac{\partial\breve{\WW}}{\partial e_2}}{e_1-e_2}
			+
			\frac{1}{\lambda_1^2}\frac{\partial \breve{\WW}}{\partial e_2}+\frac{1}{\lambda_2^2}\frac{\partial \breve{\WW}}{\partial e_1}
		\end{pmatrix}.
	\end{equation}
\end{rem}
\subsection{Logarithmic energy and positive definiteness of the bilinear forms}
With the help of the foregoing development we now prove the equivalence
\begin{align}
     Q_{\hyp(1)} > 0
     \qquad \Longleftrightarrow \qquad
     \sym\pD_{\log\lambda}\widehat\sigma(\log \lambda)\in\Sym^{++}(3) \, .
\end{align}
In fact, as we will discuss in Section \ref{par:QQ}, the positive definiteness of $Q_{\hyp(2)}$ directly follows from the positive definiteness of $Q_{\hyp(1)}$, making the positive definiteness of $Q_{\hyp(1)}$ a sufficient condition for the positive definiteness of $Q_{\hyp}$, hence implying
\begin{align}
     Q_{\hyp} > 0
     \qquad \Longleftrightarrow \qquad
     \sym\pD_{\log\lambda}\widehat\sigma(\log \lambda)\in\Sym^{++}(3).
\end{align}
For further simplification of the quadratic forms $Q_{\hyp(1)}$ and $Q_{\hyp(2)}$ we note that by \eqref{eq.sigma}
\begin{align}
	\sigma 
	= 
	\frac{1}{J} \, F \, S_2 \, F^T 
	= 
	\frac{1}{J} \, F \, \left( \sum_i \frac{\partial \breve{\WW}(e)}{\partial e_i} \, U^i \otimes U^i \right) \, F^T. 
\end{align}
Computing the eigenvalues of $\sigma$ (i.e.~the principal Cauchy stresses), we see that the vectors $u^j\coloneqq F.U^j$ are eigenvectors for $\sigma$. Indeed
\begin{align}
	\sigma.u^j
	&
	=
	\frac{1}{J} \partonBB{ F \, \partonbb{ \sum_i \frac{\partial \breve{\WW}(e)}{\partial e_i} \, U^i \otimes U^i } \, F^T }.u^j
	=
	\frac{1}{J} \; F \, \partonbb{ \sum_i \frac{\partial \breve{\WW}(e)}{\partial e_i} \, U^i \otimes U^i } \, (F^T \, F).U^j
	\nonumber
	\\
	&
	\label{eqeigenvaluessigma}
	=
	\frac{1}{J} \; F . \partonbb{  \sum_i \frac{\partial \breve{\WW}(e)}{\partial e_i} \, U^i \scal{ U^i }{ (F^T \, F).U^j}}
	=
	\frac{1}{J} \; F . \partonbb{  \sum_i \frac{\partial \breve{\WW}(e)}{\partial e_i} \, U^i \, C_{ij}}
	\\
	&
	=
	\frac{1}{J} \; F . \partonbb{  \sum_i \frac{\partial \breve{\WW}(e)}{\partial e_i} \, U^i \lambda_{j}^2\delta_{ij}}
	=
	\frac{1}{J} \; F . \partonbb{  \frac{\partial \breve{\WW}(e)}{\partial e_j} \, U^j \lambda_{j}^2}
	=
	\frac{1}{J} \, \lambda_{j}^2 \,  \frac{\partial \breve{\WW}(e)}{\partial e_j} \, u^j, \nonumber
\end{align}
because 
$
	C
	=
	2E+\id
	=
	\sum_i\lambda_i^2 \, U^i\otimes U^i
$
and hence $\sigma$ decomposes as
$
	\sigma
	=
	\sum_i
	\frac{1}{J} \, \lambda_{i}^2 \,  \frac{\partial \breve \WW(e)}{\partial e_i} \, u^i\otimes u^i.
$
Thus, finally, the eigenvalues of $\sigma$ are given by 
\begin{equation}\label{eq:princ_stresses_sigmaa}
	\sigma_i
	=
	\frac{1}{J} \, \lambda_{i}^2 
	\,  
	\frac{\partial \breve \WW(e)}{\partial e_i} 
	\qquad
	\textrm{for}
	\qquad 
	i\in\{1,2,3\}. 
\end{equation}
Note that the eigenvectors $\{u^i\}_i$ are distinct since $F$ is invertible and the eigenvectors $\{U^i\}_i$ are distinct. Let us explicitly remark that they are functions of the eigenvalues of $C$, i.e.~$\sigma_i=\sigma_i(\lambda)$. Considering $e_i=\frac{1}{2}\partonb{\lambda_i^2-1}$, since $\lambda_i\mapsto\log\lambda_i$ is monotone increasing for every $i\in\{1,2,3\}$, we can introduce an auxiliary energy density in the positive principle stretches $\{\lambda_i\}_{i=1}^3$:
\begin{equation}\label{eq:lambda}
	\widehat \WW(\log \lambda)
	\coloneqq
	\breve{\WW}(e(\lambda))
	=
	\breve{\WW}\partonbb{\frac{1}{2}\partonb{\lambda^2-1}},
\end{equation}
where $\lambda$ and $\log \lambda$ stand for the vector-fields $(\lambda_1,\lambda_2,\lambda_3)$ and $\partonb{\log \lambda_1,\log \lambda_2,\log \lambda_3}$ respectively. Then the Cauchy stress $\sigma$ can be expressed as a function of $\log \lambda$, introducing the auxiliary stress $\widehat \sigma (\log \lambda)$ and setting by abuse of notation
\begin{align}
     \widehat\sigma(\log \lambda)
     \coloneqq
     \sigma(e(\lambda)).
\end{align}
\subsubsection{Analysis of $Q_{\hyp (1)}$}\label{par_equ}

In the following, we show that the positive definiteness of $Q_{\hyp (1)}$  is equivalent to the positive definiteness of $\sym \pD_{\log \lambda}\widehat\sigma(\log \lambda)$, i.e.
	\begin{align}
	Q_{\hyp (1)} > 0 \qquad \iff \qquad \sym \pD_{\log \lambda}\widehat\sigma(\log \lambda)\in\Sym^{++}(3).
	\end{align}
In a first step we remark that the entries of $\sym \, \DD_{\log \lambda} \widehat \sigma(\log \lambda)$ are given by
	\begin{align}
	\left( \sym \, \DD_{\log \lambda} \widehat \sigma(\log \lambda) \right)_{ij} = \left\{
		\begin{array}{ll}
		\frac{\partial \widehat \sigma_i}{\partial(\log \lambda_i)} & \qquad i = j , \\
		\frac12 \left(\frac{\partial \widehat \sigma_i}{\partial (\log \lambda_j)} + \frac{\partial \widehat \sigma_j}{\partial (\log \lambda_i)} \right) & \qquad i \neq j.
		\end{array}
	\right.
	\end{align}
Now we need to establish a relation between the entries of $\sym \, \DD_{\log \lambda} \widehat \sigma(\log \lambda)$ and the components of $Q_{\hyp(1)}$. To achieve this, we need to obtain explicit relations between $\widehat\sigma$, its derivatives and the derivatives of $\widehat \WW$. \\
\\
We begin by computing the derivative on both sides of equation \eqref{eq:lambda} w.r.t.\ $\lambda$
\begin{equation}
	\begin{cases}
		\displaystyle 
		\frac{\partial}{\partial \lambda_i}\breve{\WW}\partonb{e(\lambda)}
		=
		\scal{\pD_e \, \breve{\WW}\partonb{e(\lambda)}}{ \frac{\partial e}{\partial \lambda_i}}
		=
		\sum_j \frac{\partial\breve{\WW}}{\partial e_j}\partonb{e(\lambda)}\frac{\partial e_j}{\partial \lambda_i}
		=
		\sum_j \frac{\partial\breve{\WW}}{\partial e_j}\partonb{e(\lambda)} \, \lambda_i \, \delta_{ij}
		=
		\frac{ \partial \breve{\WW}}{\partial e_i}\partonb{e(\lambda)} \; \lambda_i,
		\\[5mm]
		\displaystyle 	\frac{\partial}{\partial \lambda_i}\widehat{\WW}(\log \lambda)
		=
		\scal{\pD_{\log\lambda} \, \widehat{\WW}(\log \lambda)}{ \frac{\partial\log \lambda}{\partial \lambda_i}}
		=
		\sum_j \frac{\partial \widehat{\WW}(\log \lambda)}{\partial (\log \lambda_j)} \, \frac{1}{\lambda_i} \, \delta_{ij}
		=
		\frac{\partial \widehat{\WW}(\log \lambda)}{\partial (\log \lambda_i)} \, \frac{1}{\lambda_i}.
	\end{cases}
\end{equation}
Hence from the identity $\frac{\partial}{\partial \lambda_i}\breve{\WW}\partonb{e(\lambda)}=\frac{\partial}{\partial \lambda_i}\widehat{\WW}(\log \lambda)$ we obtain 
\begin{equation}
\label{eqlambdaquadr01}
	 \frac{\partial \widehat{\WW}(\log \lambda)}{\partial (\log \lambda_i)}
	 =
	 \frac{ \partial \breve{\WW}}{\partial e_i}\partonb{e(\lambda)} \; \lambda_i^2
	 =
	 J\,\sigma_i(e(\lambda)) = J \, \widehat \sigma_i (\lambda)
	 \quad
	 \Longleftrightarrow
	 \quad
	 \widehat \sigma_i
	 =
	 \frac{1}{J} \frac{\partial \widehat{\WW}(\log \lambda)}{\partial (\log \lambda_i)},
	 \qquad
	 \forall \, i\in\{1,2,3\}.
\end{equation}
Recalling that
\begin{align}
    J
    =
    \det F
    =
    \lambda_1 \, \lambda_2 \, \lambda_3
    =
    e^{\log \lambda_1 + \log \lambda_2 + \log \lambda_3},
\end{align}
we proceed with the relation
\begin{align}
	\frac{\partial \widehat \sigma_i}{\partial \log\lambda_i}
	&
	=
	\frac{\partial }{\partial \log\lambda_i}\partonbb{\frac{1}{J}\frac{\partial \widehat{\WW}(\log \lambda)}{\partial (\log \lambda_i)}}
	=
	\frac{\partial }{\partial \log\lambda_i}
	\partonbb{e^{-\sum_k\log\lambda_k}
	\;
	\frac{\partial \widehat{\WW}(\log \lambda)}{\partial (\log \lambda_i)
    }}
    \nonumber
    \\
    &
    =
    -
    e^{-\sum_k\log\lambda_k}
    \;
    \frac{\partial \widehat{\WW}(\log \lambda)}{\partial (\log \lambda_i)
    }
    +
    e^{-\sum_k\log\lambda_k}
    \;
    \frac{\partial^2 \widehat{\WW}(\log \lambda)}{\partial (\log \lambda_i)^2
    }
    \\
    &
    =
    \frac{1}{J}
    \partonbb{
    -
    \frac{\partial \widehat{\WW}(\log \lambda)}{\partial (\log \lambda_i)
    }
    +
    \frac{\partial^2 \widehat{\WW}(\log \lambda)}{\partial (\log \lambda_i)^2
    }},\nonumber
\end{align}
i.e.
\begin{equation}
\label{eqreffromearlier01}
	J \, \frac{\partial \widehat \sigma_i}{\partial \log\lambda_i}
	=
	\frac{\partial^2 \widehat{\WW}(\log \lambda)}{\partial (\log \lambda_i)^2}
	-
	\frac{\partial \widehat{\WW}(\log \lambda)}{\partial (\log \lambda_i)}.
\end{equation}
Moreover, from taking the second derivatives
\begin{equation}
	\begin{cases}
		\displaystyle 
		\frac{\partial^2}{\partial \lambda_i^2}\breve{\WW}\partonb{e(\lambda)}
		=
		\frac{\partial}{\partial \lambda_i}
		\partonbb{
		\frac{\partial}{\partial \lambda_i}\breve{\WW}\partonb{e(\lambda)}
	    }
		=
		\frac{\partial}{\partial \lambda_i}
		\partonbb{
			\frac{ \partial \breve{\WW}}{\partial e_i}\partonb{e(\lambda)} \; \lambda_i
		}
		=
		\lambda_i^2
		\;
		\frac{ \partial^2 \breve{\WW}}{\partial e_i^2}\partonb{e(\lambda)} 
		+
		\frac{ \partial \breve{\WW}}{\partial e_i}\partonb{e(\lambda)}, 
		\\[5mm]
		\displaystyle 	\frac{\partial^2}{\partial \lambda_i^2}\widehat{\WW}(\log \lambda)
		=
		\frac{\partial}{\partial \lambda_i}
		\partonbb{
		\frac{\partial}{\partial \lambda_i}\widehat{\WW}(\log \lambda)
	    }
		=
		\frac{\partial}{\partial \lambda_i}
		\partonbb{
			\frac{\partial \widehat{\WW}(\log \lambda)}{\partial (\log \lambda_i)} \, \frac{1}{\lambda_i}
		}
		=
		\frac{\partial^2 \widehat{\WW}(\log \lambda)}{\partial (\log \lambda_i)^2} \, \frac{1}{\lambda_i^2}
		-
		\frac{\partial \widehat{\WW}(\log \lambda)}{\partial (\log \lambda_i)}\frac{1}{\lambda_i^2}
		,
		\\[5mm]
		\displaystyle 
		\hphantom{	\frac{\partial^2}{\partial \lambda_i^2}\widehat{\WW}(\log \lambda) }
		=
		\frac{1}{\lambda_i^2}
		\partonbb{
		\frac{\partial^2 \widehat{\WW}(\log \lambda)}{\partial (\log \lambda_i)^2}
		-
		\frac{\partial \widehat{\WW}(\log \lambda)}{\partial (\log \lambda_i)}
	    },
	\end{cases}
\end{equation}
we obtain the identity
\begin{equation}
	  \frac{\partial^2 \widehat{\WW}(\log \lambda)}{\partial (\log \lambda_i)^2}
	  -
	  \frac{\partial \widehat{\WW}(\log \lambda)}{\partial (\log \lambda_i)}
	  =
	  \lambda_i^4
	  \;
	  \frac{ \partial^2 \breve{\WW}}{\partial e_i^2}\partonb{e(\lambda)} 
	  +
	  \lambda_i^2
	  \;
	  \frac{ \partial \breve{\WW}}{\partial e_i}\partonb{e(\lambda)} \, ,
\end{equation}
allowing us to establish with the use of \eqref{eqreffromearlier01}
\begin{equation}\label{eq:Q_1}
	\frac{J}{\lambda_i^4} \, \frac{\partial \widehat \sigma_i}{\partial \log\lambda_i}
	=
	\frac{ \partial^2 \breve{\WW}}{\partial e_i^2}\partonb{e(\lambda)} 
	+
	\frac{1}{\lambda_i^2}
	\;
	\frac{ \partial \breve{\WW}}{\partial e_i}\partonb{e(\lambda)},
\end{equation}
yielding the required relation for the diagonal entries of $\sym \, \DD_{\log \lambda} \widehat \sigma(\log \lambda)$. \\
\\
Proceeding with the non-diagonal entries, we observe that
\begin{align}
\label{eq3.57}
	\frac{J}{2 \, \lambda_i^2 \, \lambda_j^2} \, \left(\frac{\partial \widehat \sigma_i}{\partial (\log \lambda_j)} + \frac{\partial \widehat \sigma_j}{\partial (\log \lambda_i)} \right)
	=
	\frac{\partial^2 \breve{\WW}(e)}{\partial e_i \, \partial e_j} 
	- 
	\frac12 \, \left( \frac{1}{\lambda_i^2} \, \frac{\partial \breve{\WW}(e)}{\partial e_j} + \frac{1}{\lambda_j^2} \, \frac{\partial \breve{\WW}(E)}{\partial e_i} \right) \quad \text{for} \quad i \neq j \, ,
\end{align}
holds. To see \eqref{eq3.57}, consider that 
\begin{equation}
	\begin{cases}
		\displaystyle
		\frac{\partial \widehat \sigma_i}{\partial(\log\lambda_j)}
		=
		\frac{\partial }{\partial(\log\lambda_j)}
		\partonbb{\frac{1}{J}\frac{\partial \widehat{\WW}(\log \lambda)}{\partial (\log \lambda_i)}}
		=
		-
		\frac{1}{J}\frac{\partial \widehat{\WW}(\log \lambda)}{\partial (\log \lambda_i)}
		+
		\frac{1}{J}\frac{\partial^2 \widehat{\WW}(\log \lambda)}{\partial (\log \lambda_j)\partial (\log \lambda_i)},
		\\[5mm]
		\displaystyle
		\frac{\partial \widehat \sigma_j}{\partial(\log\lambda_i)}
		=
		\frac{\partial }{\partial(\log\lambda_i)}
		\partonbb{\frac{1}{J}\frac{\partial \widehat{\WW}(\log \lambda)}{\partial (\log \lambda_j)}}
		=
		-
		\frac{1}{J}\frac{\partial \widehat{\WW}(\log \lambda)}{\partial (\log \lambda_j)}
		+
		\frac{1}{J}\frac{\partial^2 \widehat{\WW}(\log \lambda)}{\partial (\log \lambda_i)\partial (\log \lambda_j)},
	\end{cases}
\end{equation}
and
\begin{equation}
	\begin{cases}
		\displaystyle
		\frac{\partial \breve{\WW}}{\partial\lambda_i\partial\lambda_j}
		=
		\frac{\partial}{\partial\lambda_i}
		\frac{\partial\breve{\WW}}{\partial\lambda_j}
		=
		\frac{\partial}{\partial\lambda_i}
		\partonbb{\frac{\partial\breve{\WW}}{\partial e_j}\lambda_j}
		=
		\frac{\partial^2 \breve{\WW}}{\partial e_i\partial e_j}\lambda_i\lambda_j,
		\\[5mm]
		\displaystyle
		\frac{\partial \widehat{\WW}}{\partial\lambda_i\partial\lambda_j}
		=
		\frac{\partial}{\partial\lambda_i}
		\frac{\partial\widehat{\WW}}{\partial\lambda_j}
		=
		\frac{\partial}{\partial\lambda_i}
		\partonbb{\frac{\partial\widehat{\WW}}{\partial(\log \lambda_j)} \; \frac{1}{\lambda_j}}
		=
		\frac{\partial^2\widehat{\WW}}{\partial(\log \lambda_i)\partial(\log \lambda_j)} \; \frac{1}{\lambda_i\lambda_j}.
	\end{cases}
\end{equation} 
Therefore, using \eqref{eqlambdaquadr01} with $i \neq j$ we have
\begin{align}\label{eq:Q_2}
		\frac{J}{2 \, \lambda_i^2 \, \lambda_j^2} \, \left(\frac{\partial \widehat \sigma_i}{\partial (\log \lambda_j)} + 
		\frac{\partial \widehat \sigma_j}{\partial (\log \lambda_i)} \right) 
		&
		=
		\frac{J}{2 \, \lambda_i^2 \, \lambda_j^2}
		\partonbb{
		-
		\frac{1}{J}\frac{\partial \widehat{\WW}(\log \lambda)}{\partial (\log \lambda_i)}
		-
		\frac{1}{J}\frac{\partial \widehat{\WW}(\log \lambda)}{\partial (\log \lambda_j)}
		+
		2
		\frac{1}{J}\frac{\partial^2 \widehat{\WW}(\log \lambda)}{\partial (\log \lambda_j)\partial (\log \lambda_i)}
		}
		\nonumber
		\\
		&
		=
		\frac{1}{2 \, \lambda_i^2 \, \lambda_j^2}
		\partonbb{
			-
			\frac{ \partial \breve{\WW}}{\partial e_i}\partonb{e(\lambda)} \; \lambda_i^2
			-
			\frac{ \partial \breve{\WW}}{\partial e_j}\partonb{e(\lambda)} \; \lambda_j^2
			+
			2
			\frac{\partial^2 \breve{\WW}}{\partial e_i\partial e_j}\lambda_i^2\lambda_j^2
		}
		\nonumber
		\\
		&
		=
		\frac{\partial^2 \breve{\WW}(e)}{\partial e_i\partial e_j}
		-
		\frac{1}{2}
		\partonbb{
		\frac{ \partial \breve{\WW}(e)}{\partial e_i} \; \frac{1}{\lambda_j^2}
		+
		\frac{ \partial \breve{\WW}(e)}{\partial e_j} \; \frac{1}{\lambda_i^2}
	    }.
\end{align}
Finally, by combining \eqref{eq:Q_1} with \eqref{eq:Q_2} and recalling
\begin{align}
	Q_{\hyp(1)}(\dot E_{11},\dot E_{22}, \dot E_{33})
	=
	\sum_{i=1}^3\partonbb{
		\frac{\partial^2 \breve{\WW}}{\partial e_i^2}+\frac{1}{\lambda_i^2}\frac{\partial \breve{\WW}}{\partial e_i}
	}\dot E_{ii}^2
	+
	\sum_{i\neq j}
	\partonbb{
		\frac{\partial^2 \breve{\WW}}{\partial e_i\partial e_j}
		-
		\frac{1}{2}
		\partonbb{
			\frac{1}{\lambda_i^2}
			\frac{\partial \breve{\WW}}{\partial e_j}
			+\frac{1}{\lambda_j^2}
			\frac{\partial \breve{\WW}}{\partial e_i}
		}
	}
	\dot E_{ii}
	\dot{E}_{jj},
\end{align}
it follows that
	\begin{align}
	Q_{\hyp(1)}(\dot{E}_{11},\dot{E}_{22},\dot{E}_{33}) 
	&= 
	\frac{J}{2} \, \sum_{i,j} \left(\frac{\partial \widehat \sigma_i}{\partial (\log \lambda_j)} 
	+ 
	\frac{\partial \widehat \sigma_j}{\partial (\log \lambda_i)}\right) \, \frac{\dot{E}_{ii}}{\lambda_i^2} \, \frac{\dot{E}_{jj}}{\lambda_j^2} \\
	&\hspace*{-100pt}=
	J \, \left \langle \begin{pmatrix} \dot E_{11} \\ \dot E_{22} \\ \dot E_{33} \end{pmatrix} \!,\!
	\begin{pmatrix}
	\frac{1}{\lambda_1^4} \, \frac{\partial \widehat \sigma_1}{\partial (\log \lambda_1)} 
	& \frac{1}{2 \, \lambda_1^2 \, \lambda_2^2} \, \left( \frac{\partial \widehat \sigma_1}{\partial (\log \lambda_2)} + \frac{\partial \widehat \sigma_2}{\partial (\log \lambda_1)} \right) 
	& \frac{1}{2 \, \lambda_1^2 \, \lambda_3^2} \, \left( \frac{\partial \widehat \sigma_1}{\partial (\log \lambda_3)} + \frac{\partial \widehat \sigma_3}{\partial (\log \lambda_1)} \right) \\
	\frac{1}{2 \, \lambda_1^2 \, \lambda_2^2} \, \left( \frac{\partial \widehat \sigma_1}{\partial (\log \lambda_2)} + \frac{\partial \widehat \sigma_2}{\partial (\log \lambda_1)} \right) 
	& \frac{1}{\lambda_2^4} \, \frac{\partial \widehat \sigma_2}{\partial (\log \lambda_2)} 
	& \frac{1}{2 \, \lambda_2^2 \, \lambda_3^2} \, \left( \frac{\partial \widehat \sigma_2}{\partial (\log \lambda_3)} + \frac{\partial \widehat \sigma_3}{\partial (\log \lambda_2)} \right) \\
	\frac{1}{2 \, \lambda_1^2 \, \lambda_3^2} \, \left( \frac{\partial \widehat \sigma_1}{\partial (\log \lambda_3)} + \frac{\partial \widehat \sigma_3}{\partial (\log \lambda_1)} \right) 
	& \frac{1}{2 \, \lambda_2^2 \, \lambda_3^2} \, \left( \frac{\partial \widehat \sigma_2}{\partial (\log \lambda_3)} + \frac{\partial \widehat \sigma_3}{\partial (\log \lambda_2)} \right)
	& \frac{1}{\lambda_3^4} \, \frac{\partial \widehat \sigma_3}{\partial (\log \lambda_3)} 
	\end{pmatrix} \!.\!
	\begin{pmatrix} \dot E_{11} \\ \dot E_{22} \\ \dot E_{33} \end{pmatrix} \right\rangle \, ,\notag
	\end{align}
and hence that $Q_{\hyp(1)}$ is positive-definite if and only if the same is true for the tensor $\Lambda$ of components
	\begin{align}
	\label{eqleblond03}
	\Lambda_{ij} = \frac12 \, \left( \frac{\partial \widehat \sigma_i}{\partial (\log \lambda_j)} + \frac{\partial \widehat \sigma_j}{\partial (\log \lambda_i)}\right) = \sym \frac{\partial \widehat \sigma_i(\log \lambda)}{\partial (\log \lambda_j)} = \sym \, \DD_{\log \lambda} \widehat \sigma(\log \lambda)
	\in\Sym^{++}(3)
	.
	\end{align}
\subsubsection{Analysis of $Q_{\hyp (2)}$}\label{par:QQ}
Let us similarly express $Q_{\hyp(2)}$ as a function of the components $\{\sigma_i\}_{i=1}^3$ (the principal Cauchy stresses). From
\begin{align}
	Q_{\hyp(2)}(\dot E_{12},\dot E_{23}, \dot E_{31})
	= 
	\sum_{i\neq j}
	\partonbb{
		\frac{\frac{\partial\breve{\WW}}{\partial e_i}-\frac{\partial\breve{\WW}}{\partial e_j}}{e_i-e_j}
		+
		\frac{1}{\lambda_i^2}\frac{\partial \breve{\WW}}{\partial e_j}+\frac{1}{\lambda_j^2}\frac{\partial \breve{\WW}}{\partial e_i}
	}
	\dot{E}_{ij}^2,
\end{align}
using the already established identity (see \eqref{eq:princ_stresses_sigmaa})
	\begin{align}
	\sigma_i=\frac{\lambda_i^2}{J}\frac{\partial\breve{\WW}}{\partial e_i}\; \iff \;\frac{\partial\breve{\WW}}{\partial e_i}=\frac{J}{\lambda_i^2}\sigma_i,
	\end{align}
we obtain
\begin{align}
	Q_{\hyp(2)}(\dot E_{12},\dot E_{23}, \dot E_{31})
	&
	= 
	\sum_{i\neq j}
	\partonBB{\frac{\frac{\partial\breve{\WW}}{\partial e_i}-\frac{\partial\breve{\WW}}{\partial e_j}}{e_i-e_j}
		+
		\frac{1}{\lambda_i^2}\frac{\partial\breve{\WW}}{\partial e_j}
		+
		\frac{1}{\lambda_j^2}\frac{\partial\breve{\WW}}{\partial e_i}
		}\dot{E}_{ij}^2
	=
	\sum_{i\neq j}
	\partonBB{
		\frac{
			\frac{J}{\lambda_i^2}\widehat \sigma_i-\frac{J}{\lambda_j^2}\widehat \sigma_j}{e_i-e_j}
		+
		\frac{J}{\lambda_i^2\,\lambda_j^2}\widehat \sigma_j
		+
		\frac{J}{\lambda_j^2\,\lambda_i^2}\widehat \sigma_i
	}\dot{E}_{ij}^2\nonumber
	\\
	&
	=
	J \,
	\sum_{i\neq j}
	\partonBB{
		\frac{1}{\underbrace{e_i-e_j}_{\mathclap{=\,\frac{1}{2}\lambda_i^2-\frac{1}{2}-\frac{1}{2}\lambda_j^2+\frac{1}{2}}}}
		\,
		\frac{\lambda_j^2\widehat \sigma_i-\lambda_i^2\widehat \sigma_j}{\lambda_i^2\lambda_j^2}
		+
		\frac{\widehat \sigma_i+\widehat \sigma_j}{\lambda_i^2\,\lambda_j^2}
	}\dot{E}_{ij}^2
	=
	J \,
	\sum_{i\neq j}
	\partonBB{
		\frac{2}{\lambda_i^2-\lambda_j^2}
		\,
		\frac{\lambda_j^2\widehat \sigma_i-\lambda_i^2\widehat \sigma_j}{\lambda_i^2\lambda_j^2}
		+
		\frac{\widehat \sigma_i+\widehat \sigma_j}{\lambda_i^2\,\lambda_j^2}
	}\dot{E}_{ij}^2\nonumber
	\\[-5mm]
	&
	\label{eqQ2forsigma}
	=
	J \,
	\sum_{i\neq j}
	\partonbb{
		\frac{\cancel{2} \, \lambda_j^2 \, \widehat \sigma_i - \cancel{2} \, \lambda_i^2 \, \widehat \sigma_j 
			+ \overbrace{(\widehat \sigma_i+\widehat \sigma_j)(\lambda_i^2-\lambda_j^2)}^{\mathclap{\lambda_i^2 \, \widehat \sigma_i-\cancel{\lambda_j^2 \, \widehat \sigma_i}+\cancel{\lambda_i^2 \, \widehat \sigma_j}-\lambda_j^2 \, \widehat \sigma_j}}}{(\lambda_i^2-\lambda_j^2) \, \lambda_i^2 \, \lambda_j^2}
	}\dot{E}_{ij}^2
	\\
	&
	=
	J \,
	\sum_{i\neq j}
	\partonbb{
		\frac{ \lambda_j^2 \, \widehat \sigma_i -  \lambda_i^2 \, \widehat \sigma_j 
			+ 
			\lambda_i^2 \, \widehat \sigma_i-\lambda_j^2 \, \widehat \sigma_j}{(\lambda_i^2-\lambda_j^2) \, \lambda_i^2 \, \lambda_j^2}
	}\dot{E}_{ij}^2
	=
	J \,
	\sum_{i\neq j}
		\frac{  \widehat \sigma_i \, (\lambda_i^2 + \lambda_j^2 ) 
			-   
			\widehat \sigma_j \, (\lambda_i^2 + \lambda_j^2 )}{(\lambda_i^2-\lambda_j^2) \, \lambda_i^2 \, \lambda_j^2}
	\, \dot{E}_{ij}^2
	\nonumber
	\\
	&
	=
	J \,
	\sum_{i\neq j}
	\frac{  (\widehat \sigma_i - \widehat \sigma_j ) \, (\lambda_i^2 + \lambda_j^2 ) 
		}{(\lambda_i^2-\lambda_j^2) \, \lambda_i^2 \, \lambda_j^2}
	\, \dot{E}_{ij}^2
	=
	J \,
	\sum_{i\neq j}
	\frac{  (\widehat \sigma_i - \widehat \sigma_j ) 
	}{(\lambda_i^2-\lambda_j^2)}
	\,
	\frac{\lambda_i^2 + \lambda_j^2  
	}{ \lambda_i^2 \, \lambda_j^2}
	\, \dot{E}_{ij}^2
	\nonumber
	\\
	&
	=
	J \,
	\sum_{i\neq j}
	\frac{  \widehat \sigma_i - \widehat \sigma_j  
	}{(\lambda_i^2-\lambda_j^2)}
	\underbrace{\partonbb{
	\frac{1  
	}{ \lambda_i^2 }
	+
	\frac{1  
	}{ \lambda_j^2}
    }
	\, \dot{E}_{ij}^2}_{>0}.
	\nonumber
\end{align}
Thus $Q_{\hyp(2)}$ is seen to be positive-definite if and only if the ordering of the principal Cauchy stresses $\widehat \sigma_i$ is the same as that of the principal stretches $\lambda_i$. This is verified if the Baker-Ericksen (BE$^+$) \cite{BakerEri54} inequality is satisfied, i.e. 
	\begin{align}
	(\sigma_i - \sigma_j) \, (\lambda_i - \lambda_j) > 0,
	\end{align}
Incidentally, this is already \underline{true} if $Q_{\hyp(1)}(\dot E)>0$, as previously stated by Hill \cite{hill1968constitutivea}, since\footnote
{Here and throughout, we refer to a differentiable mapping as \emph{strongly monotone} if the symmetric part of its derivative is positive definite. Note that strong monotonicity immediately implies strict (Hilbert) monotonicity.
}
(see also Theorem \ref{theorem:mainResult} in the Appendix)
\begin{equation} 
	\label{eqproof001}
	Q_{\hyp(1)}(\dot E)>0
	\iff
	\Lambda\in\Sym^{++}(3)
	\iff
	\log V
	\to
	\widehat \sigma(\log V)
	\;\textrm{is strongly Hilbert-monotone}
	\quad
	\Longrightarrow
	\quad
	\text{BE}^+,
\end{equation}
where the last implication will be proven in the following
\begin{prop}[TSTS-M$^+$ implies BE$^+$] \label{prop3.7}
Strong Hilbert-monotonicity (always) implies strict Hilbert-monotonicity of $\log V \mapsto \widehat \sigma(\log V)$, which implies the Baker-Ericksen BE$^+$-inequalities \cite{BakerEri54}
	\begin{align}
	(\sigma_i - \sigma_j) \, (\lambda_i - \lambda_j) > 0.
	\end{align}
\end{prop}
\begin{proof}
Assume that the mapping $\log V \mapsto \widehat \sigma(\log V)$ is strictly monotone, i.e.,
	\begin{align}
	\label{eqmonotonlog}
	\text{TSTS-M$^+$:} \quad \langle \widehat \sigma(\log V_2) - \widehat \sigma(\log V_1), \log V_2 - \log V_1 \rangle > 0 \qquad \forall \, V_1, V_2 \in \Sym^{++}(3), \quad V_1 \neq V_2.
	\end{align}
We begin by considering the two-dimensional case, i.e.~$V \in \Sym^{++}(2)$. Since $\sigma(V) = \widehat \sigma(\log V)$ is an isotropic tensor function, the principal Cauchy stresses $\widehat \sigma_i(\log \lambda)$ are given by the diagonal entries of $\widehat \sigma(\log \diag(\lambda_1, \lambda_2))$, which itself is a diagonal matrix, i.e.
	\begin{align}
	\sigma(\log \diag(\lambda_1, \lambda_2))
	=
	\begin{pmatrix} \sigma_1(\log \lambda_1, \log \lambda_2) & 0 \\ 0 & \sigma_2(\log \lambda_1, \log \lambda_2) \end{pmatrix}.
	\end{align}
Additionally, also due to isotropy, we have by reversing $\lambda_1, \lambda_2$
	\begin{align}
	\sigma(\log \diag(\lambda_2, \lambda_1))
	=
	\begin{pmatrix} \sigma_2(\log \lambda_1, \log \lambda_2) & 0 \\ 0 & \sigma_1(\log \lambda_1, \log \lambda_2) \end{pmatrix}.
	\end{align}
Now, the monotonicity \eqref{eqmonotonlog} implies with $V_1 = \diag(\lambda_2, \lambda_1), \; V_2 = \diag(\lambda_1, \lambda_2)$
	\begin{equation}
	\begin{alignedat}{2}
	0 &< \langle \widehat \sigma(\log V_2) - \widehat \sigma(\log V_1), \log V_2 - \log V_1 \rangle \\
	&= \langle \begin{pmatrix} \widehat \sigma_2 & 0 \\ 0 & \widehat \sigma_1 \end{pmatrix} 
	- \begin{pmatrix} \widehat \sigma_1 & 0 \\ 0 & \widehat \sigma_2 \end{pmatrix}, 
	\begin{pmatrix} \log \lambda_2 & 0 \\ 0 & \log \lambda_1 \end{pmatrix}
	- \begin{pmatrix} \log \lambda_1 & 0 \\ 0 & \log \lambda_2 \end{pmatrix} \rangle \\
	&= (\widehat \sigma_2 - \widehat \sigma_1) \, (\log \lambda_2 - \log \lambda_1) + (\widehat \sigma_1 - \widehat \sigma_2) \, (\log \lambda_1 - \log \lambda_2) = 2 \, (\log \lambda_2 - \log \lambda_1) \, (\widehat \sigma_2 - \widehat \sigma_1) \\
	\implies \qquad 0 &< (\lambda_2 - \lambda_1) \, (\sigma_2 - \sigma_1),
	\end{alignedat}
	\end{equation}
where the last implication is due to the monotonicity of the (scalar) logarithm function (i.e.~$\lambda_2 > \lambda_1$ if and only if $\log \lambda_2 > \log \lambda_1$) and the fact that $\widehat \sigma_i(\log \lambda) = \sigma_i(\lambda)$ per definition. \\
\\
Generalizing this argument to three (or even $n$) dimensions we pick two distinct eigenvalues $\lambda_i \neq \lambda_j$ of the matrix $V$ and interchange their position in the diagonal matrix $\log \diag(\lambda_1, \ldots, \lambda_n)$, as we did for $\lambda_1$ and $\lambda_2$ in the two-dimensional case. This once again yields two matrices $V_1 = \diag(\ldots, \lambda_i, \ldots, \lambda_j, \ldots)$ and $V_2 = (\ldots, \lambda_j, \ldots, \lambda_i, \ldots)$ that can be inserted into the monotonicity inequality \eqref{eqmonotonlog}. When subtracting $\log V_2 - \log V_1$, every entry that is not $\log \lambda_i$ or $\log \lambda_j$ cancels out, resulting in
	\begin{align}
	0 &< \langle \widehat \sigma(\log V_2) - \widehat \sigma(\log V_1), \log V_2 - \log V_1 \rangle &&= (\widehat \sigma_i - \widehat \sigma_j) \, (\log \lambda_i - \log \lambda_j) + (\widehat \sigma_j - \widehat \sigma_i) \, (\log \lambda_j - \log \lambda_i) \notag \\
	& &&= 2 \, (\log \lambda_i - \log \lambda_j) \, (\widehat \sigma_i - \widehat \sigma_j) \notag \\
	\implies \qquad 0 &< (\lambda_i - \lambda_j) \, (\sigma_i - \sigma_j), \quad \text{for} \quad i \neq j, &&
	\end{align}
with the same reasoning as in the two-dimensional case.
\end{proof}
\begin{rem}
Note that monotonicity of $\widehat \sigma(\log V)$ in $\log V$ is equivalent to monotonicity of $\widehat \sigma(\log B)$ in $\log B$, since $\log B = 2 \, \log V$.
\end{rem}
\subsubsection{Summarizing considerations for the hyperelastic case}
Summarizing the results from the previous Sections, we have proven:
	\begin{center}
	\fbox{
	\begin{minipage}[h!]{0.85\linewidth}
		\centering
		\textbf{For an isotropic hyperelastic material, the corotational stability postulate (CSP) $\langle \frac{\DD^{\ZJ}}{\DD t}[\sigma] , D \rangle > 0$ is equivalent to the strong Hilbert-monotonicity of $\widehat \sigma$ in $\log V$, which is TSTS-M$^{++}$.}
	\end{minipage}}
	\end{center}
It will be proven in detail (see Appendix \ref{appendixhill}) with the methods presented in this chapter, that Hill's inequality $\langle \frac{\DD^{\ZJ}}{\DD t}[\tau] , D \rangle > 0$ (see \cite{hill1968constitutivea,hill1968constitutiveb}) for the Kirchhoff stress tensor $\tau$ is satisfied if and only if the tensor of components $\frac{\partial \tau_i}{\partial (\log \lambda_j)}$, which is (major) symmetric since $\tau_i = J \, \sigma_i = \pD_{\log \lambda_i} \widehat{\WW}(\log \lambda_i)$, is positive-definite (even for the more general Cauchy-elastic case). In the hyperelastic case this implies that $\log V \mapsto \widehat \WW(\log V)$ must be a convex function and for $\tau = \DD_{\log V} \widehat \WW(\log V)$ it holds
	\begin{align}
	\langle \tau(\log V_1) - \tau(\log V_2), \log V_1 - \log V_2 \rangle > 0, \qquad \forall \, V_1, V_2 \in \Sym^{++}(3), \quad V_1 \neq V_2.
	\end{align}
\textbf{This requirement, however, is not strong enough to exclude physically unacceptable responses}, as shows the constitutive Hencky energy \cite{Neff_Osterbrink_Martin_Hencky13}, for which $\tau = 2 \, \mu \, \log V + \lambda \, \tr(\log V) \cdot \id$ is strongly Hilbert-monotone in $\log V$ for $2 \, \mu + 3 \, \lambda > 0$.
%
%
%
%
\section{The corotational stability postulate in the Cauchy-elastic case} \label{sec:leblondcauchy}
In this section we generalize the results for hyperelasticity and the Zaremba-Jaumann rate from the previous section by showing, that the same conclusion still holds true in the Cauchy-elastic case. The crucial difference for Cauchy-elasticity consists in the fact that the second Piola-Kirchhoff stress tensor $S_2$ is now not necessarily derived from an elastic energy density $\WW(F)$. Thus, its derivative $\pD_E\,S_2(E)$ could lack the major symmetry property. \\
\\
More precisely, in the hyperelastic case it is $\pD_E \, S_2(E)=\pD^2_E \WW(E)$, so that $\pD_E \, S_2$ inherits both the major and minor symmetries by Schwarz's theorem. This means that 
\begin{equation}
	\pD_E \, S_2\in\Sym\partonb{\Sym(3),\Sym(3)\,}
	\coloneqq
	\grafb{T\in\textrm{Lin}\partonb{\Sym(3),\Sym(3)\,}\;\left.\right|\;\scal{T.A}{B}
		=
		\scal{A}{T.B}\;\forall \, A,B\in\Sym(3)}.
\end{equation}
Component-wise, $\pD_E \, S_2\in\Sym\partonb{\Sym(3),\Sym(3)\,}$ reads as
\begin{align}
    (\pD_E \, S_2)_{ijhk}
    =
    (\pD_E \, S_2)_{jihk}=(\pD_E \, S_2)_{ijkh}=(\pD_E \, S_2)_{hkij}
    \qquad
    \forall \, 
    i,j,h,k\in\{1,2,3\}.
\end{align}
However, in the Cauchy-elastic case, we only have $\pD_E \, S_2\in\textrm{Lin}\partonb{\Sym(3),\Sym(3)\,}$, i.e.~component-wise,
\begin{align}
    (\pD_E \, S_2)_{ijhk}
    =
    (\pD_E \, S_2)_{jihk}=(\pD_E \, S_2)_{ijkh}
    \qquad
    \forall
    i,j,h,k\in\{1,2,3\}.
\end{align}
Let us explicitly remark that $\pD_E \, S_2(E)$ is the linearization of the second Piola-Kirchhoff stress tensor $S_2(E)$ at $E$, so that $S_2:\Sym(3)\to\Sym(3)$ is in general a non-linear function
and ($T_E \Sym(3)$ denoting the tangential space of $\Sym(3)$)
\begin{align}
\pD_E \, S_2(E):\Sym(3)\simeq T_E\Sym(3)\xrightarrow{\phantom{aaaa}} T_E\Sym(3)\simeq\Sym(3)
\end{align}
because $\Sym(3)$ is an affine space. \\
\\
As before, our goal consists in showing that the bilinear form
\begin{equation}
	\begin{alignedat}{2}
		Q_\ela(\dot{E})
		&=
		\langle \dot S_2(E) , \dot{E} \rangle + 2 \, \,\textrm{tr}\,\big(C^{-1}\,\dot E\,S_2\,\dot E\big) - \langle C^{-1} , \dot{E} \rangle \, \langle S_2(E) , \dot{E} \rangle
	\end{alignedat}
\end{equation}
derived in Section \ref{sec3.01} is positive definite even in the Cauchy-elastic case.  

\subsection{Representation of $S_2$ in the eigenbasis of $E$}

Similar to Corollary \ref{cor2} we have
\begin{prop} \label{eigenbasis}
Consider a homogeneous and isotropic Cauchy-elastic material. Then there are functions $\{\breve s_j(e)\}_{j=1}^3$,
 such that $S_2(C)$ can be represented in the eigenvalues and eigenvectors of $E$ by setting
\begin{equation}
	S_2(C(e))
	=
    \breve
	S_2(e)
	=
	\sum_{j=1}^3 \breve s_j(e) \, U^j\otimes U^j.
\end{equation}
This is, for the non-hyperelastic case, the analogue of \eqref{eq:stress_Kirchh_dec}. 
\end{prop}
\begin{proof}
As before, $E$ can be decomposed w.r.t.\ its basis of eigenvectors as $E=\sum_{i=1}^3 e_i\,U^i\otimes U^i$. According to \cite[Thm. 3.6-2, p.116 and Ex. 3.9, p.135]{ciarlet2022}, we can express $S_2$ as a function of the invariants of $C$ when considering a homogeneous and isotropic material. Next, we introduce three auxiliary functions $\gamma_i:\bR^3\to\bR$, $i\in\{1,2,3\}$ 
of the principal invariants $\{\iota_i(\lambda_1^2,\lambda_2^2,\lambda_3^2)\}_{i=0}^2$ of $C$ (which are functions of the eigenvalues of $C$) such that \cite{BakerEri54, richter1948isotrope, richter1949hauptaufsatze, Richter50, Richter52}
\begin{align}\label{dec_C}
	S_2(C)
	&
	=
	\sum_{i=0}^2 \gamma_i(\iota_1,\iota_2,\iota_3) C^i.
\end{align} 
Therefore, considering the relation $C = 2 \, E + \id$, we can expand \eqref{dec_C} as follows
\begin{align}
	S_2(C)
	&
	=
	\sum_{i=0}^2 \gamma_i(\iota_1,\iota_2,\iota_3) C^i
	=
	\sum_{i=0}^2 \gamma_i(\iota_1,\iota_2,\iota_3) (2 \, E + \id)^i
	\nonumber
	\\
	&
	=
	\gamma_0(\iota_1,\iota_2,\iota_3) \, \id
	+
	\gamma_1(\iota_1,\iota_2,\iota_3)(2 \, E + \id)
	+
	\gamma_2(\iota_1,\iota_2,\iota_3) (2 \, E + \id)^2
	\nonumber
	\\
	&
	=
	\gamma_0(\iota_1,\iota_2,\iota_3) \, \id
	+
	\gamma_1(\iota_1,\iota_2,\iota_3)(2 \, E + \id)
	+
	\gamma_2(\iota_1,\iota_2,\iota_3) (4 \, E^2 + 4 \, E + \id)
	\nonumber
	\\
	&
	=
	\parqB{\gamma_0(\iota_1,\iota_2,\iota_3) + \gamma_1(\iota_1,\iota_2,\iota_3) + \gamma_2(\iota_1,\iota_2,\iota_3)} \id
	+
	\parqB{2 \, \gamma_1(\iota_1,\iota_2,\iota_3) + 4 \, \gamma_2(\iota_1,\iota_2,\iota_3)} E
	\nonumber
	\\
	&
	\qquad
	+ 
	4 \, \gamma_2(\iota_1,\iota_2,\iota_3) E^2
\end{align}
Hence, introducing the functions $c_i:\bR^3\to\bR$, $i\in\{0,1,2\}$ of the principal invariants $\{\iota_i(\lambda_1^2,\lambda_2^2,\lambda_3^2)\}_{i=1}^3$ of $C$ as
\begin{align}
    c_0(\iota_1,\iota_2,\iota_3)
    &
    =
    \parqB{\gamma_0(\iota_1,\iota_2,\iota_3) + \gamma_1(\iota_1,\iota_2,\iota_3) + \gamma_2(\iota_1,\iota_2,\iota_3)},
    &
    c_1(\iota_1,\iota_2,\iota_3)
    &
    =
    \parqB{2 \, \gamma_1(\iota_1,\iota_2,\iota_3) + 4 \, \gamma_2(\iota_1,\iota_2,\iota_3)},
    \nonumber
    \\
    c_2(\iota_1,\iota_2,\iota_3)
    &
    =
    4 \, \gamma_2(\iota_1,\iota_2,\iota_3)
\end{align}
we further develop the expression of $S_2(C)$ as follows
\begin{align}
	S_2(C)
	&
	=
	\sum_{i=0}^2 c_i(\iota_1,\iota_2,\iota_3) E^i
	=
	 c_0(\iota_1,\iota_2,\iota_3) \id + c_1(\iota_1,\iota_2,\iota_3) E + c_2(\iota_1,\iota_2,\iota_3)  E^2
	\nonumber
	\\
	&
	=
	c_0(\iota_1,\iota_2,\iota_3) \sum_{j=1}^3  U^j\otimes U^j 
	+ 
	c_1(\iota_1,\iota_2,\iota_3) \sum_{j=1}^3 \, e_j U^j\otimes U^j 
	+ 
	c_2(\iota_1,\iota_2,\iota_3)  \sum_{j=1}^3 (e_j)^2 \, U^j\otimes U^j
	\nonumber
	\\
	&
	=
	\sum_{j=1}^3  \partonbb{ \sum_{i=0}^2 \partonB{ c_i(\iota_1,\iota_2,\iota_3) (e_j)^i } U^j\otimes U^j}.
\end{align} 
Thus, from the relation $C = 2 \, E \, + \id$ and the fact that $\lambda_i^2 = 1 + 2 \, e_i$, we can introduce three new functions $\{\breve s_j(e)\}_{j=1}^3$ defined by
\begin{align}
     \breve s_j(e)
     \coloneqq
     \sum_{i=0}^2
     c_i\partonB{\iota_1\partonb{\lambda_1^2(e),\lambda_2^2(e),\lambda_3^2(e)},\iota_2\partonb{\lambda_1^2(e),\lambda_2^2(e),\lambda_3^2(e)},\iota_3\partonb{\lambda_1^2(e),\lambda_2^2(e),\lambda_3^2(e)}} (e_j)^i,
\end{align}
such that $S_2(C)$ can be represented in the eigenvalues of $E$ by setting
\begin{align}
	S_2(C(e))
	&=
    \breve
	S_2(e)
	=
	\sum_{j=1}^3 \breve s_j(e) \, U^j\otimes U^j. \qedhere
\end{align}
\end{proof}
\noindent By abuse of notation, we simply write $S_2(e)$ for $\breve S_2(e)$. 
\subsection{Positive definiteness of the quadratic form $Q_\ela$ in isotropic Cauchy-elasticity} 
Similar to Proposition \ref{prop5} we now obtain
\begin{prop} \label{propeqall.2}
In the Cauchy-elastic case we have the identities
\begin{equation}
\label{eqall.2}
\begin{alignedat}{2}
\scalb{\dot S_2}{\dot E}
	&
	=
	\sum_{i,j=1}^3
	\frac{\partial \breve s_i}{\partial e_j}
	\, \dot{E}_{jj} \, \dot{E}_{ii}
	+
	\sum_{i\neq j}
	\frac{\breve s_i-\breve s_j}{e_i-e_j}
	\dot{E}_{ij}^2, \\
	\scalb{C^{-1}}{\dot E}\scalb{S_2}{\dot E}
	&
	=
	\sum_{i=1}^3\frac{1}{\lambda_i^2} \,  \breve s_i \,  \dot E_{ii}^2
	+
	\frac{1}{2}
	\sum_{i\neq j}
	\partonbb{
		\frac{1}{\lambda_i^2}
		\breve s_j
		+\frac{1}{\lambda_j^2}
		\breve s_i
	}
	\dot E_{ii}
	\dot{E}_{jj}, \\
\tr\partonb{C^{-1}\dot E \, S_2\,\dot E}
	&
	=
	\sum_{i=1}^3\frac{1}{\lambda_i^2} \, \dot E_{ii}^2 \, \breve s_i
	+
	\frac{1}{2}
	\sum_{i\neq j}
	\partonbb{
		\frac{\breve s_j}{\lambda_i^2} 
		+
		\frac{\breve s_i }{\lambda_j^2}  
	} \dot E_{ij}^2.
\end{alignedat}
\end{equation}

\end{prop}
\begin{proof}
The derivative $\dot S_2(e)$ can be computed as done before, i.e.
\begin{align}
	\frac{\dif}{\dif t} S_2(e(x,t))
	&
	=
	\frac{\dif}{\dif t}
	\sum_{i=1}^3  \breve s_i(e(x,t)) \, U^i(x,t)\otimes U^i(x,t)
	=
	\sum_{i=1}^3 
	\parqbb{
		\sum_{j=1}^3  \frac{\partial \breve s_i}{\partial e_j}(e) \, \dot e_j \, U^j\otimes U^j
		+
		\breve s_i(e)
		\partonbb{\dot U^i \otimes U^i + U^i \otimes \dot U^i}
	}
	\nonumber
	\\
	\overset{\textrm{by }\eqref{eq:1}}&{=}
	\sum_{i,j=1}^3 
	\frac{\partial \breve s_i}{\partial e_j}(e) \, \dot e_{j} \, U^j\otimes U^j
	+
	\sum_{i\neq j} 
	+
	\frac{\breve s_i(e)-\breve s_j(e)}{e_i-e_j}
	\,
	\dot E_{ij}
	\,
	U^i \otimes U^j.
\end{align}
Then we obtain for $\eqref{eqall.2}_1$
\begin{align}
	\scalb{\dot S_2}{\dot E}
	\overset{\textrm{by} \eqref{eq:Spunto}}&{=}
	\scal{\sum_{i,j=1}^3 
		\frac{\partial \breve s_i}{\partial e_j} \, \dot e_{j} \, U^j\otimes U^j
		+
		\sum_{i\neq j} 
		+
		\frac{\breve s_i-\breve s_j}{e_i-e_j}
		\,
		\dot E_{ij}
		\,
		U^i \otimes U^j
	}{\sum_{h,k} \dot E_{hk} \, U^h\otimes U^k}
	\nonumber
	\\
	&
	=
	\scal{\sum_{i,j=1}^3
		\frac{\partial \breve s_i}{\partial e_j} \, \dot{E}_{jj}
		\;
		U^i\otimes U^i
	}{\sum_{h,k} \dot E_{hk} \, U^h\otimes U^k}
	+
	\scal{
		\sum_{i\neq j}
		\frac{\breve s_i-\breve s_j}{e_i-e_j}
		\dot{E}_{ij} \, U^i\otimes U^j
	}{\sum_{h\neq k} \dot E_{hk} \, U^h\otimes U^k}
	\nonumber
	\\
	&
	=
	\sum_{i,j=1}^3
	\frac{\partial \breve s_i}{\partial e_j}
	\, \dot{E}_{jj} \, \dot{E}_{ii}
	+
	\sum_{i\neq j}
	\frac{\breve s_i-\breve s_j}{e_i-e_j}
	\dot{E}_{ij}^2.
\end{align}
Moreover we have
\begin{align}
	\scalb{S_2}{\dot E}
	&
	=
	\scal{\sum_{j=1}^3 \breve s_j \, U^j\otimes U^j}{
		\sum_{j=1}^3\dot e_j U^j\otimes U^j
		+
		\sum_{j=1}^3 e_j \dot U^j\otimes U^j
		+
		\sum_{j=1}^3 e_j  U^j\otimes \dot U^j
	}
	\nonumber
	\\
	&
	=
	\sum_{j=1}^3 \breve s_j \, \underbrace{\dot e_j}_{=\,\dot E_{jj}}
	+
	\sum_{j=1}^3 \breve s_j \,  \underbrace{\normb{U^j}^2}_{=\,1}
	\underbrace{\scalb{\dot U^j}{U^j}}_{=\,0}
	+
	\sum_{j=1}^3 \breve s_j \, \underbrace{\normb{U^j}^2}_{=\,1}
	\underbrace{\scalb{\dot U^j}{U^j}}_{=\,0}
	=
	\sum_{j=1}^3 \breve s_j \, \dot{E}_{jj},
\end{align}
so that we get for $\scalb{C^{-1}}{\dot E}\scalb{S_2}{\dot E}$
\begin{align}
	\scalb{C^{-1}}{\dot E}\scalb{S_2}{\dot E}
	&
	=
	\parton{\sum_{i=1}^3\frac{1}{\lambda_i^2} \, \dot E_{ii}}
	\parton{\sum_{j=1}^3\breve s_j \, \dot{E}_{jj}}
	=
	\sum_{i=1}^3\frac{1}{\lambda_i^2} \,  \breve s_i \,  \dot E_{ii}^2
	+
	\frac{1}{2}
	\sum_{i\neq j}
	\partonbb{
		\frac{1}{\lambda_i^2}
		\breve s_j
		+\frac{1}{\lambda_j^2}
		\breve s_i
	}
	\dot E_{ii}
	\dot{E}_{jj}.
\end{align}
Finally, the trace $\tr\partonb{C^{-1}\dot E \, S_2\,\dot E}$ is given by
\begin{align}
	\tr\partonb{C^{-1}\dot E \, S_2\,\dot E}
	&
	=
	\tr\partonbb{
		\sum_{k,j}\frac{1}{\lambda_k^2}\partonbb{\sum_{i}\dot E_{ki}
			\, \breve s_i \,
			\dot E_{ij}}U^k\otimes U^j}
	=
	\sum_{i,j}\frac{1}{\lambda_j^2} \, \dot E_{ji} \, \breve s_i \, \dot E_{ij}
	\\
	&
	=
	\sum_{i}\frac{1}{\lambda_i^2} \, \dot E_{ii}^2 \, \breve s_i
	+
	\frac{1}{2}
	\sum_{i\neq j}
	\partonbb{
		\frac{\breve s_j}{\lambda_i^2} 
		+
		\frac{\breve s_i }{\lambda_j^2}  
	} \dot E_{ij}^2. \qedhere
\end{align}
\end{proof}
\noindent Combining these identities and recalling that the quadratic form $Q_{\ela}$ was given by the same expression as $Q_{\hyp}$, namely
	\begin{align}
	Q_{\ela} = \langle \dot S_2(E) , \dot{E} \rangle + 2 \, \,\textrm{tr}\,\big(C^{-1}\,\dot E\,S_2\,\dot E\big) - \langle C^{-1} , \dot{E} \rangle \, \langle S_2(E) , \dot{E} \rangle,
	\end{align}
we see that $Q_{\ela}$ can be rewritten as
\begin{align}
Q_\ela(\dot E)
=
Q_{\ela(1)}(\dot E_{11},\dot E_{22}, \dot E_{33})
+
Q_{\ela(2)}(\dot E_{12},\dot E_{23}, \dot E_{31}),
\end{align}
where this time (compare with \eqref{eqhypcompare1} and \eqref{eqhypcompare2} in the hyperelastic case)
\begin{align}
	Q_{\ela(1)}(\dot E_{11},\dot E_{22}, \dot E_{33})
	&
	=
	\sum_{i=1}^3\partonbb{
		\frac{\partial \breve s_i}{\partial e_i}
		+
		\frac{\breve s_i}{\lambda_i^2}
	}\dot E_{ii}^2
	+
	\sum_{i\neq j}
	\partonbb{
		\frac{\partial \breve s_j}{\partial e_i}
		-
		\frac{1}{2}
		\partonbb{
			\frac{\breve s_j}{\lambda_i^2}
			+
			\frac{\breve s_i}{\lambda_j^2}
		}
	}
	\dot E_{ii}
	\dot{E}_{jj}
\end{align}
and 
\begin{align}
	Q_{\ela(2)}(\dot E_{12},\dot E_{23}, \dot E_{31})
	= 
	\sum_{i\neq j}
	\partonbb{
		\frac{
			\breve s_i
			-
			\breve s_j
		}
		{e_i-e_j}
		+
		\frac{\breve s_j}{\lambda_i^2}
		+
		\frac{\breve s_i}{\lambda_j^2}
	}
	\dot{E}_{ij}^2 \, ,
\end{align}
i.e.
\begin{equation}
	\widehat Q_\ela(\dot E,\dot E)
	=
	\left\langle \begin{pmatrix}
			\dot E_{11} \\ \dot E_{22} \\ \dot E_{33} \\ \dot E_{23} \\ \dot E_{13} \\ \dot E_{12}
		\end{pmatrix}\!,\!
		\left(\begin{array}{@{}c@{}c|c@{}c@{}}
			&\phantom{0} &&
			\\
			& \widehat Q_{\ela(1)} &0 &
			\\
			&\phantom{0} &&
			\\
			\hline
			&\phantom{0} &&
			\\
			&0 & \widehat Q_{\ela(2)}&
			\\
			&\phantom{0} &&
		\end{array}\right)
		\!.\!
		\begin{pmatrix}
			\dot E_{11} \\ \dot E_{22} \\ \dot E_{33} \\ \dot E_{23} \\ \dot E_{13} \\ \dot E_{12}
		\end{pmatrix}\right\rangle
\end{equation}
with
\begin{equation}
	\widehat Q_\ela(\dot E,\dot E)
	=
	\scal{\begin{pmatrix}
			\dot E_{11} \\ \dot E_{22} \\ \dot E_{33} 
	\end{pmatrix}\!\!}{ \widehat Q_{\ela(1)} . \!\! \begin{pmatrix}
			\dot E_{11} \\ \dot E_{22} \\ \dot E_{33} 
	\end{pmatrix}}_{\mathclap{\bR^{3}}}
    \;\;
	+
	\;\;
	\scal{\begin{pmatrix}
			\dot E_{23} \\ \dot E_{13} \\ \dot E_{12}
	\end{pmatrix}\!\!}{ \widehat Q_{\ela(2)} . \!\! \begin{pmatrix}
			\dot E_{23} \\ \dot E_{13} \\ \dot E_{12}
	\end{pmatrix}}_{\mathclap{\bR^{3}}},
\end{equation}
where
\begin{equation*}
	\widehat Q_{\ela(1)}
	=
	\begin{pmatrix}
		\frac{\partial \breve s_1}{\partial e_1}
		+
		\frac{\breve s_1}{\lambda_1^2}
		&
		\frac{\partial \breve s_1}{\partial e_2}
		-
		\frac{1}{2}
		\partonB{
			\frac{\breve s_2}{\lambda_1^2}
			+\frac{\breve s_1}{\lambda_2^2}
		}
		&
		\frac{\partial \breve s_1}{\partial e_3}
		-
		\frac{1}{2}
		\partonB{
			\frac{\breve s_3}{\lambda_1^2}
			+\frac{\breve s_1}{\lambda_3^2}
		}
		\\[2mm]
		\frac{\partial \breve s_2}{\partial e_1}
		-
		\frac{1}{2}
		\partonB{
			\frac{\breve s_2}{\lambda_1^2}
			+\frac{\breve s_1}{\lambda_2^2}
		}
		&
		\frac{\partial \breve s_2}{\partial e_2}
		+
		\frac{\breve s_2}{\lambda_2^2}
		&
		\frac{\partial \breve s_2}{\partial e_3}
		-
		\frac{1}{2}
		\partonB{
			\frac{\breve s_3}{\lambda_2^2}
			+\frac{\breve s_2}{\lambda_3^2}
		}
		\\[2mm]
		\frac{\partial \breve s_3}{\partial e_1}
		-
		\frac{1}{2}
		\partonB{
			\frac{\breve s_3}{\lambda_1^2}
			+\frac{\breve s_1}{\lambda_3^2}
		}
		&
		\frac{\partial \breve s_3}{\partial e_2}
		-
		\frac{1}{2}
		\partonB{
			\frac{\breve s_3}{\lambda_2^2}
			+
			\frac{\breve s_2}{\lambda_3^2}
		}
		&
		\frac{\partial \breve s_3}{\partial e_3}
		+
		\frac{\breve s_3}{\lambda_3^2}
	\end{pmatrix}
\end{equation*}
and
\begin{equation}
	\widehat Q_{\ela(2)}
	=
	\begin{pmatrix}
		\frac{
			\breve s_2
			-
			\breve s_3
		}
		{e_2-e_3}
		+
		\frac{\breve s_3}{\lambda_2^{\mathstrut 2}}
		+
		\frac{\breve s_2}{\lambda_3^{\mathstrut 2}}
		&
		0
		&
		0
		\\[2mm]
		0
		&
		\frac{
			\breve s_1
			-
			\breve s_3
		}
		{e_1-e_3}
		+
		\frac{\breve s_3}{\lambda_1^{\mathstrut 2}}
		+
		\frac{\breve s_1}{\lambda_3^{\mathstrut 2}}
		&
		0
		\\[2mm]
		0
		&
		0
		&
		\frac{
			\breve s_1
			-
			\breve s_2
		}
		{e_1-e_2}
		+
		\frac{\breve s_2}{\lambda_1^{\mathstrut 2}}
		+
		\frac{\breve s_1}{\lambda_2^{\mathstrut 2}}
	\end{pmatrix}.
\end{equation}

%
\clearpage
\subsection{Positive definiteness of the bilinear forms}

The expressions of $Q_{\ela(1)}$ and $Q_{\ela(2)}$ can be simplified by noting that by \eqref{eq.sigma},
\begin{align}
	\sigma 
	= 
	\frac{1}{J} \, F \, S_2 \, F^T 
	= 
	\frac{1}{J} \, F \, \left( \sum_i \breve s_i(e) \, U^i \otimes U^i \right) \, F^T .
\end{align}
Similar to \eqref{eqeigenvaluessigma}, we can prove that the vectors $u^j\coloneqq F.U^j$ are eigenvectors for $\sigma$, since
\begin{align}
	\sigma.u^j
	&
	=
	\frac{1}{J} \partonBB{ F \, \partonbb{ \sum_i \breve s_i(e) \, U^i \otimes U^i } \, F^T }.u^j
	=
	\frac{1}{J} \; F \, \partonbb{ \sum_i \breve s_i(e) \, U^i \otimes U^i } \, (F^T \, F).U^j
	\nonumber
	\\
	&
	=
	\frac{1}{J} \; F . \partonbb{  \sum_i \breve s_i(e) \, U^i \scal{ U^i }{ (F^T \, F).U^j}}
	=
	\frac{1}{J} \; F . \partonbb{  \sum_i \breve s_i(e) \, U^i C_{ij}}
	\\
	&
	=
	\frac{1}{J} \; F . \partonbb{  \sum_i \breve s_i(e) \, U^i \lambda_{j}^2\delta_{ij}}
	=
	\frac{1}{J} \; F . \partonbb{  \breve s_j(e) \, U^j \lambda_{j}^2}
	=
	\frac{1}{J} \, \lambda_{j}^2 \,  \breve s_j(e) \, u^j, \nonumber
\end{align}
where we used that
$
	C
	=
	2E+\id
	=
	\sum_i\lambda_i^2 \, U^i\otimes U^i.
$
Thus, $\sigma$ decomposes as\footnote{Clearly, if the eigenvalues of $\sigma$ are distinct then, by the spectral theorem, the vectors $\{u^i\}_{i=1}^3$ are orthogonal. If there is an eigenvalue with multiplicity bigger than 1, the corresponding eigenvectors in the representation of $\sigma$ are chosen orthogonal.}
$
	\sigma
	=
	\sum_i
	\frac{1}{J} \, \lambda_{i}^2 \,  \breve s_i(e) \, u^i\otimes u^i
$
and the eigenvalues of $\sigma$ are given by 
\begin{equation}\label{eq:princ_stresses_sigma}
	\sigma_i
	=
	\frac{1}{J} \, \lambda_{i}^2 
	\,  
	\breve s_i(e)
	\qquad
	\textrm{for}
	\qquad 
	i\in\{1,2,3\} 
\end{equation}
Note that the vectors $U^i$ are linearly independent and so are the images $u^i$ by invertibility of $F$, which ensures that the $u^i$ are distinct.
\subsubsection{Quadratic form $Q_{\ela(1)}$}
First, let us remark that 
\begin{equation}
	\scal{\begin{pmatrix}
			\dot E_{11} \\ \dot E_{22} \\ \dot E_{33} 
	\end{pmatrix}\!\!}{ \widehat Q_{\ela(1)} . \!\! \begin{pmatrix}
			\dot E_{11} \\ \dot E_{22} \\ \dot E_{33} 
	\end{pmatrix}}_{\!\!\!\bR^{3}}
	=
	\scal{\begin{pmatrix}
			\dot E_{11} \\ \dot E_{22} \\ \dot E_{33} 
	\end{pmatrix}\!\!}{ \sym\widehat Q_{\ela(1)} . \!\! \begin{pmatrix}
			\dot E_{11} \\ \dot E_{22} \\ \dot E_{33} 
	\end{pmatrix}}_{\!\!\!\bR^{3}}
    \qquad
    \forall
    \begin{pmatrix}
    	\dot E_{11} \\ \dot E_{22} \\ \dot E_{33} 
    \end{pmatrix}
    \in\bR^3
\end{equation}
simply because, for any skew symmetric matrix $A\in\mathfrak{so}(3)$, it always holds
$
\scal{A.v}{v}_{\bR^{3}}
=0
$
for every $v\in\bR^3$.

Considering $e_i=\frac{1}{2}\partonb{\lambda_i^2-1}$, since $\lambda\mapsto\log\lambda$ is monotone increasing, we can introduce an auxiliary principal stress in the positive principle stretches $\{\lambda_i\}_{i=1}^3$

\begin{equation}\label{eq:lambda_new_2}
	\widehat s_i(\log \lambda)
	\coloneqq
	\breve s_i(e(\lambda))
	=
	\breve s_i\partonbb{\frac{1}{2}\partonb{\lambda^2-1}},
\end{equation}
where $\lambda$ and $\log \lambda$ stand respectively for the vectors $(\lambda_1,\lambda_2,\lambda_3)$ and $\partonb{\log \lambda_1,\log \lambda_2,\log \lambda_3}$.

Our goal is once again to establish a relation between the entries of $\sym \, \DD_{\log \lambda} \widehat \sigma(\log \lambda)$ and the components of $Q_{\hyp(1)}$. Therefore, we observe that with \eqref{eq:princ_stresses_sigma} and \eqref{eq:lambda_new_2} the identities
\begin{equation} \label{eqrelationfortau}
	J\,\widehat\sigma_i(\log \lambda)
	=
	\lambda_{i}^2 
	\,
	\widehat s_i(\log \lambda)
	\qquad
	\Longleftrightarrow
	\qquad
	\widehat\sigma_i(\log \lambda)
	=
	\frac{\lambda_{i}^2}{J}
	\widehat s_i(\log \lambda)
	\qquad
	\forall i\in\{1,2,3\}
\end{equation}
hold. This allows us to compute the derivative of each $\widehat\sigma_i(\log \lambda)$ w.r.t.\ $\log\lambda_i$. 
Reminding that
\begin{align}
J
=
\det F
=
\lambda_1 \, \lambda_2 \, \lambda_3
=
e^{ \, \sum_{h=1}^3\log\lambda_h},
\end{align}
for the diagonal components of $\pD_{\log\lambda}\widehat\sigma(\log \lambda)$ we obtain
\begin{align}
	\frac{\partial \widehat \sigma_i}{\partial \log\lambda_i}
	&
	=
	\frac{\partial }{\partial \log\lambda_i}
	\partonBB{
		\frac{\lambda_{i}^2}{J}
		\widehat s_i(\log \lambda)
	}
	=
	\frac{\partial }{\partial \log\lambda_i}
	\partonBB{
		\frac{\lambda_{i}}{\prod_{j\neq i}\lambda_j}
		\widehat s_i(\log \lambda)
	}
	=
	\frac{\partial }{\partial \log\lambda_i}
	\partonBB{
		e^{\log\lambda_i}
		\,
		e^{-\log(\prod_{j\neq i}\lambda_j)}
		\,
		\widehat s_i(\log \lambda)
	}\nonumber
	\\
	&
	=
	e^{-\log(\prod_{j\neq i}\lambda_j)}
	\frac{\partial }{\partial \log\lambda_i}
	\partonBB{
		e^{\log\lambda_i}
		\,
		\widehat s_i(\log \lambda)
	}
	=
	e^{-\log(\prod_{j\neq i}\lambda_j)}
	\partonBB{
		e^{\log\lambda_i}
		\,
		\widehat s_i(\log \lambda)
		+
		e^{\log\lambda_i}
		\,
		\frac{\partial \widehat s_i(\log \lambda)}{\partial \log\lambda_i}
	}
	\nonumber
	\\
	&
	=
	e^{-\log(\prod_{j\neq i}\lambda_j)}
	e^{\log\lambda_i}
	\partonBB{
		\widehat s_i(\log \lambda)
		+
		\,
		\frac{\partial \widehat s_i(\log \lambda)}{\partial \log\lambda_i}
	}
	=
	\frac{\lambda_i^2}{J}
	\partonBB{
		\widehat s_i(\log \lambda)
		+
		\,
		\frac{\partial \widehat s_i(\log \lambda)}{\partial \log\lambda_i}
	},
\end{align}
i.e.
\begin{equation}
	J \, \frac{\partial \widehat\sigma_i}{\partial \log\lambda_i}
	=
	\lambda_i^2
	\,
	\partonBB{
		\widehat s_i(\log \lambda)
		+
		\,
		\frac{\partial \widehat s_i(\log \lambda)}{\partial \log\lambda_i}
	}.
\end{equation}
For the off-diagonal derivatives ($i\neq j$) we obtain 
\begin{align}
	\frac{\partial \widehat \sigma_i}{\partial \log\lambda_j}
	&
	=
	\frac{\partial }{\partial \log\lambda_j}
	\partonBB{
		\frac{\lambda_{i}^2}{J}
		\widehat s_i(\log \lambda)
	}
	=
	\frac{\partial }{\partial \log\lambda_j}
	\partonBB{
		\frac{\lambda_{i}}{\prod_{h\neq i}\lambda_h}
		\widehat s_i(\log \lambda)
	}
	=
	\frac{\partial }{\partial \log\lambda_j}
	\partonBB{
		e^{\log\lambda_i}
		\,
		e^{-\log(\prod_{h\neq i}\lambda_h)}
		\,
		\widehat s_i(\log \lambda)
	}\nonumber
	\\
	&
	=
	e^{\log\lambda_i}
	\frac{\partial }{\partial \log\lambda_j}
	\partonBB{
		e^{-\log(\prod_{h\neq i}\lambda_h)}
		\,
		\widehat s_i(\log \lambda)
	}
	=
	e^{\log\lambda_i}
	\frac{\partial }{\partial \log\lambda_j}
	\partonBB{
		e^{-\sum_{h\neq i}\log(\lambda_h)}
		\,
		\widehat s_i(\log \lambda)
	}\nonumber
	\\
	&
	=
	e^{\log\lambda_i}
	\partonBB{
		-
		e^{-\sum_{h\neq i}\log(\lambda_h)}
		\,
		\widehat s_i(\log \lambda)
		+
		e^{-\sum_{h\neq i}\log(\lambda_h)}
		\,
		\frac{\partial \widehat s_i(\log \lambda)}{\partial \log\lambda_j}
	}
	\\
	&
	=
	e^{\log\lambda_i}
	e^{-\sum_{h\neq i}\log(\lambda_h)}
	\partonBB{
		-
		\widehat s_i(\log \lambda)
		+
		\frac{\partial \widehat s_i(\log \lambda)}{\partial \log\lambda_j}
	}
	\nonumber
	=
	\frac{\lambda_i^2}{J}
	\partonBB{
		-
		\widehat s_i(\log \lambda)
		+
		\,
		\frac{\partial \widehat s_i(\log \lambda)}{\partial \log\lambda_j}
	}. \nonumber
\end{align}
Evaluating the sum $\frac{\partial \widehat \sigma_i}{\partial \log\lambda_j}+\frac{\partial \widehat \sigma_j}{\partial \log\lambda_i}$, we have
\begin{align}\label{New_new_ident}
	\frac{J}{2 \, \lambda_i^2 \, \lambda_j^2}\partonBB{
		\frac{\partial \widehat \sigma_i}{\partial \log\lambda_j}
		+
		\frac{\partial \widehat \sigma_j}{\partial \log\lambda_i}
	}
	&
	=
	\frac{J}{2 \, \lambda_i^2 \, \lambda_j^2}\partonBB{
		- \frac{\lambda_i^2}{J} \, \widehat s_i(\log \lambda)
		+ \frac{\lambda_i^2}{J} \, \frac{\partial \widehat s_i(\log \lambda)}{\partial \log\lambda_j}
		- \frac{\lambda_j^2}{J} \, \widehat s_j(\log \lambda)
		+ \frac{\lambda_j^2}{J} \, \frac{\partial \widehat s_j(\log \lambda)}{\partial \log\lambda_i}
	}\nonumber
	\\
	&
	=
	\frac{1}{2}
	\partonBB{
		\frac{1}{\lambda_j^2}
		\,
		\frac{\partial \widehat s_i(\log \lambda)}{\partial \log\lambda_j}
		+
		\frac{1}{\lambda_i^2}
		\,
		\frac{\partial \widehat s_j(\log \lambda)}{\partial \log\lambda_i} 
	}
	-
	\frac{1}{2}
	\partonBB{
		\frac{\widehat s_j(\log \lambda)}{\lambda_i^2}
		+
		\frac{\widehat s_i(\log \lambda)}{\lambda_j^2}    
	}.
\end{align}
Now, let us explore the link between $\frac{\partial\breve s_i}{\partial e_j}$ and $\frac{\partial\widehat s_i}{\partial \log\lambda_j}$. From the identities $\widehat s_i(\log \lambda)=\breve s_i(e(\lambda))$, deriving w.r.t.\ $\lambda_j$, we obtain
\begin{equation}
	\begin{cases}
		\displaystyle 
		\frac{\partial \breve s_i(e(\lambda))}{\partial\lambda_j}
		=
		\scal{\pD_e \, \breve s_i(e(\lambda))}{ \frac{\partial e}{\partial \lambda_j}}
		=
		\sum_{k=1}^3 \frac{\partial\breve s_i}{\partial e_k}\partonb{e(\lambda)}\frac{\partial e_k}{\partial \lambda_j}
		=
		\sum_{k=1}^3 \frac{\partial\breve s_i}{\partial e_k}\partonb{e(\lambda)} \, \lambda_j \, \delta_{jk}
		=
		\frac{ \partial \breve s_i}{\partial e_j}\partonb{e(\lambda)} \; \lambda_j,
		\\[6mm]
		\displaystyle 
		\frac{\partial \widehat s_i(\log \lambda)}{\partial\lambda_j}
		=
		\scal{\pD_{\log\lambda} \, \widehat s_i(\log \lambda)}{ \frac{\partial\log \lambda}{\partial \lambda_j}}
		=
		\sum_{k=1}^3 \frac{\partial \widehat s_i(\log \lambda)}{\partial (\log \lambda_k)} \, \frac{1}{\lambda_j} \, \delta_{kj}
		=
		\frac{\partial \widehat s_i(\log \lambda)}{\partial (\log \lambda_j)} \, \frac{1}{\lambda_j}.
	\end{cases}
\end{equation}
Hence from the identity 
$
\frac{\partial}{\partial \lambda_j}\breve s_i\partonb{e(\lambda)}
=
\frac{\partial}{\partial \lambda_j}\widehat s_i(\log \lambda)
$ 
we see
\begin{equation}\label{new_ident}
	\frac{\partial \widehat s_i(\log \lambda)}{\partial (\log \lambda_j)}
	=
	\frac{ \partial \breve s_i}{\partial e_j}\partonb{e(\lambda)} \; \lambda_j^2
	\qquad
	\qquad
	\forall i,j\in\{1,2,3\}.
\end{equation}
Therefore, using in \eqref{New_new_ident} the identities established in \eqref{new_ident}, we obtain
\begin{align}
	\frac{J}{2 \, \lambda_i^2 \, \lambda_j^2}\partonBB{
		\frac{\partial \widehat \sigma_i}{\partial \log\lambda_j}
		+
		\frac{\partial \widehat \sigma_j}{\partial \log\lambda_i}
	}
	&
	=
	\frac{1}{2}
	\partonBB{
		\frac{1}{\lambda_j^2}
		\,
		\frac{\partial \widehat s_i(\log \lambda)}{\partial \log\lambda_j}
		+
		\frac{1}{\lambda_i^2}
		\,
		\frac{\partial \widehat s_j(\log \lambda)}{\partial \log\lambda_i} 
	}
	-
	\frac{1}{2}
	\partonBB{
		\frac{\widehat s_j(\log \lambda)}{\lambda_i^2}
		+
		\frac{\widehat s_i(\log \lambda)}{\lambda_j^2}    
	}\nonumber
	\\
	&
	=
	\frac{1}{2}
	\partonBB{ 
		\frac{ \partial \breve s_i}{\partial e_j}\partonb{e(\lambda)}
		+
		\frac{ \partial \breve s_j}{\partial e_i}\partonb{e(\lambda)}
	}
	-
	\frac{1}{2}
	\partonBB{ 
		\frac{\breve s_j\partonb{e(\lambda)}}{\lambda_i^2}
		+
		\frac{\breve s_i\partonb{e(\lambda)}}{\lambda_j^2}
	}
\end{align}
when $i\neq j$ and 
\begin{equation}
	\frac{J}{\lambda_i^4} \, \frac{\partial \widehat\sigma_i}{\partial \log\lambda_i}
	=
	\frac{1}{\lambda_i^2}
	\,
	\partonBB{
		\widehat s_i(\log \lambda)
		+
		\,
		\frac{\partial \widehat s_i(\log \lambda)}{\partial \log\lambda_i}
	}
	=
	\frac{1}{\lambda_i^2}
	\,
	\frac{\partial \widehat s_i(\log \lambda)}{\partial \log\lambda_i}
	+
	\frac{\widehat s_i(\log \lambda)}{\lambda_i^2}
	=
	\frac{ \partial \breve s_i}{\partial e_j}\partonb{e(\lambda)}
	+
	\frac{\breve s_i\partonb{e(\lambda)}}{\lambda_i^2}
\end{equation}
when $i=j$. In other words, we have established that
\begin{equation}
	J
	\scal{\begin{pmatrix}
			\frac{\dot E_{11}}{\lambda_1^2} \\[1mm] \frac{\dot E_{22}}{\lambda_2^2} \\[1mm] \frac{\dot E_{33}}{\lambda_3^2} 
	\end{pmatrix}\!\!}{ 
		\sym \pD_{\log\lambda}\widehat\sigma 
		. \!\!
		\begin{pmatrix}
			\frac{\dot E_{11}}{\lambda_1^2} \\[1mm] \frac{\dot E_{22}}{\lambda_2^2} \\[1mm] \frac{\dot E_{33}}{\lambda_3^2} 
	\end{pmatrix}}_{\mathclap{\bR^{3}}}
	=
	\scal{
		\begin{pmatrix}
			\dot E_{11} \\ \dot E_{22} \\ \dot E_{33} 
	\end{pmatrix}\!\!}{ \sym \widehat Q_{\ela(1)} .\!\! \begin{pmatrix}
			\dot E_{11} \\ \dot E_{22} \\ \dot E_{33} 
	\end{pmatrix}}_{\!\!\!\bR^{3}}
	=
	\scal{
		\begin{pmatrix}
			\dot E_{11} \\ \dot E_{22} \\ \dot E_{33} 
	\end{pmatrix}\!\!}{ \widehat Q_{\ela(1)} . \!\! \begin{pmatrix}
			\dot E_{11} \\ \dot E_{22} \\ \dot E_{33} 
	\end{pmatrix}}_{\!\!\!\bR^{3}},    
\end{equation}
from which follows that 
\[
    Q_{\ela(1)} > 0
    \qquad
    \Longleftrightarrow
    \qquad
    \sym \pD_{\log\lambda} \widehat \sigma(\log \lambda) \in \Sym^{++}(3).
\]
\subsubsection{Quadratic form $Q_{\ela(2)}$}

Via the established identities 
$
\widehat\sigma_i
=
\frac{\lambda_{i}^2}{J}
\widehat s_i(\log \lambda)
=
\frac{\lambda_{i}^2}{J}
\breve s_i(e(\lambda))
\;
\Longleftrightarrow
\;
\frac{J}{\lambda_{i}^2}
\,
\widehat\sigma_i
=
\widehat s_i(\log \lambda)
=
\breve s_i(e(\lambda))
$
in \eqref{eqrelationfortau} we obtain
\begin{align}
	Q_{\ela(2)}(\dot E_{12},\dot E_{23}, \dot E_{31})
	&
	= 
	\sum_{i\neq j}
	\partonbb{
		\frac{
			\breve s_i
			-
			\breve s_j
		}
		{e_i-e_j}
		+
		\frac{\breve s_j}{\lambda_i^2}
		+
		\frac{\breve s_i}{\lambda_j^2}
	}
	\dot{E}_{ij}^2
	=
	\sum_{i\neq j}
	\partonBB{
		\frac{
			\frac{J}{\lambda_i^2} \widehat\sigma_i^2-\frac{J}{\lambda_j^2}\widehat\sigma_j^2}{e_i-e_j}
		+
		\frac{J}{\lambda_i^2\,\lambda_j^2}\widehat\sigma_j
		+
		\frac{J}{\lambda_j^2\,\lambda_i^2}\widehat\sigma_i
	}\dot{E}_{ij}^2\nonumber
	\\
	&
	=
	J
	\sum_{i\neq j}
	\partonBB{
		\frac{1}{\underbrace{e_i-e_j}_{\mathclap{=\,\frac{1}{2}\lambda_i^2-\frac{1}{2}-\frac{1}{2}\lambda_j^2+\frac{1}{2}}}}
		\,
		\frac{\lambda_j^2\widehat\sigma_i-\lambda_i^2\widehat\sigma_j}{\lambda_i^2\lambda_j^2}
		+
		\frac{\widehat\sigma_i+\widehat\sigma_j}{\lambda_i^2\,\lambda_j^2}
	}\dot{E}_{ij}^2
	=
	J
	\sum_{i\neq j}
	\partonBB{
		\frac{2}{\lambda_i^2-\lambda_j^2}
		\,
		\frac{\lambda_j^2\widehat\sigma_i-\lambda_i^2\widehat\sigma_j}{\lambda_i^2\lambda_j^2}
		+
		\frac{\widehat\sigma_i+\widehat\sigma_j}{\lambda_i^2\,\lambda_j^2}
	}\dot{E}_{ij}^2\nonumber
	\\[-5mm]
	&
	\label{eqlastone1}
	=
	J
	\sum_{i\neq j}
	\partonbb{
		\frac{\cancel{2} \, \lambda_j^2 \, \widehat\sigma_i - \cancel{2} \, \lambda_i^2 \, \widehat\sigma_j 
			+ \overbrace{(\widehat\sigma_i+\widehat\sigma_j)(\lambda_i^2-\lambda_j^2)}^{\mathclap{\lambda_i^2 \, \widehat\sigma_i-\cancel{\lambda_j^2 \, \widehat\sigma_i}+\cancel{\lambda_i^2 \, \widehat\sigma_j}-\lambda_j^2 \, \widehat\sigma_j}}}{(\lambda_i^2-\lambda_j^2) \, \lambda_i^2 \, \lambda_j^2}
	}\dot{E}_{ij}^2
	\\
	&
	=
	J
	\sum_{i\neq j}
	\partonbb{
		\frac{ \lambda_j^2 \, \widehat\sigma_i -  \lambda_i^2 \, \widehat\sigma_j 
			+ 
			\lambda_i^2 \, \widehat\sigma_i-\lambda_j^2 \, \widehat\sigma_j}{(\lambda_i^2-\lambda_j^2) \, \lambda_i^2 \, \lambda_j^2}
	}\dot{E}_{ij}^2
	=
	J
	\sum_{i\neq j}
	\frac{  \widehat\sigma_i \, (\lambda_i^2 + \lambda_j^2 ) 
		-   
		\widehat\sigma_j \, (\lambda_i^2 + \lambda_j^2 )}{(\lambda_i^2-\lambda_j^2) \, \lambda_i^2 \, \lambda_j^2}
	\, \dot{E}_{ij}^2
	\nonumber
	\\
	&
	=
	J
	\sum_{i\neq j}
	\frac{  (\widehat\sigma_i - \widehat\sigma_j ) \, (\lambda_i^2 + \lambda_j^2 ) 
	}{(\lambda_i^2-\lambda_j^2) \, \lambda_i^2 \, \lambda_j^2}
	\, \dot{E}_{ij}^2
	=
	J
	\sum_{i\neq j}
	\frac{  (\widehat\sigma_i - \widehat\sigma_j ) 
	}{(\lambda_i^2-\lambda_j^2)}
	\,
	\frac{\lambda_i^2 + \lambda_j^2  
	}{ \lambda_i^2 \, \lambda_j^2}
	\, \dot{E}_{ij}^2
	\nonumber
	\\
	&
	=
	J
	\sum_{i\neq j}
	\frac{  \widehat\sigma_i - \widehat\sigma_j  
	}{(\lambda_i^2-\lambda_j^2)}
	\underbrace{\partonbb{
			\frac{1  
			}{ \lambda_i^2 }
			+
			\frac{1  
			}{ \lambda_j^2}
		}
		\, \dot{E}_{ij}^2}_{>0}.
	\nonumber
\end{align}
The analysis of the positive definiteness of the quadratic form $Q_{\ela(2)}$ follows by the same arguments as in Section \ref{par:QQ}, as it is implied by the Baker-Ericksen inequalities and as in the hyperelastic case we have
\begin{equation} 
	\hspace*{-10pt}
	Q_{\ela(1)}(\dot E)>0
	\iff
	\Lambda\in\Sym^{++}(3)
	\iff
	\log V
	\to
	\widehat \sigma(\log V)
	\;\textrm{is strongly Hilbert  monotone}
	\quad
	\Longrightarrow
	\quad
	\text{BE}^+,
\end{equation}
where the last equivalence is due to Theorem \ref{theorem:mainResult}.
\clearpage
\section{Conclusion and open questions}
We have shown that the corotational stability postulate (CSP) expressed in terms of the Zaremba-Jaumann rate $\langle \frac{\DD^{\ZJ}}{\DD t}[\sigma], D \rangle > 0, \; \forall \, D \in \Sym(3) \! \setminus \! \{0\}$ is equivalent to $\sym \DD_{\log B} \widehat \sigma(\log B) \in \Sym^{++}_4(6)$, which is essentially the monotonicity of the Cauchy stress $\sigma$ with respect to the logarithmic strain $\log B$. It is still surprising that the CSP condition ``favors'' the logarithmic strain. The hidden mechanism behind this feature needs to be clarified. Moreover, certain limitations of the result and the methodology must be acknowledged. 
	\begin{enumerate}
	\item Any constitutive condition should be independent of an assumed rate. This is presently not the case.
	\item The use of the Zaremba-Jaumann rate seems somewhat adhoc. Why should the result be restricted to this rate? 
	\item Hill uses the Kirchhoff stress $\tau$ while Leblond starts with the Cauchy stress $\sigma$. Thus the question arises: Which stress should be taken in the corotational stability postulate (CSP)?
	\item The proof needs a lot of technical machinery in Lagrangean axes and seems to hide the essential ideas, thus raising the question: can we expect a similar result if we would use a different corotational rate and more importantly: how could we transfer the methods of proof?
	\item The corotational stability postulate should be motivated in a larger context. 
	\item The exponentiated Hencky energy (cf.~\cite{nedjar2018, NeffGhibaLankeit}) 
	\begin{align}
	\WW(F) = \frac{\mu}{k}\,e^{k\,\norm{\log(\sqrt{F \, F^T})}^2}+\frac{\lambda}{2\widehat{k}}\,e^{\widehat{k}\,[{\rm tr}(\log(\sqrt{F \, F^T}))]^2}, \qquad \mu > 0, \; 2 \, \mu + 3 \, \lambda > 0
	\end{align}
satisfies the CSP condition throughout but it is not polyconvex (cf.~\cite{martin2017}). It would be interesting to see CSP and polyconvexity being satisfied simultaneously.
	\item Finally, the corotational stability postulate should be put to use in showing some kind of local existence for nonlinear isotropic Cauchy-elasticity. An attempt in this regard will be made in \cite{blesgen2024}.
	\end{enumerate}
We will address all these issues in an upcoming paper \cite{agostino2024}.

\footnotesize
\bibliographystyle{plain} 
\bibliography{Leblondrefs}
\normalsize
\begin{appendix}
\section{Appendix}
\subsection{Notation} \label{appendixnotation}
\textbf{Inner product} \\
\\
For $a,b\in\R^n$ we let $\langle {a},{b}\rangle_{\R^n}$  denote the scalar product on $\R^n$ with associated vector norm $\norm{a}_{\R^n}^2=\langle {a},{a}\rangle_{\R^n}$. We denote by $\R^{n\times n}$ the set of real $n\times n$ second order tensors, written with capital letters. The standard Euclidean scalar product on $\R^{n\times n}$ is given by
$\langle {X},{Y}\rangle_{\R^{n\times n}}=\tr{(X Y^T)}$, where the superscript $^T$ is used to denote transposition. Thus the Frobenius tensor norm is $\norm{X}^2=\langle {X},{X}\rangle_{\R^{n\times n}}$, where we usually omit the subscript $\R^{n\times n}$ in writing the Frobenius tensor norm. The identity tensor on $\R^{n\times n}$ will be denoted by $\id$, so that $\tr{(X)}=\langle {X},{\id}\rangle$. \\
\\
\noindent \textbf{Frequently used spaces} 
	\begin{itemize}
	\item $\Sym(n), \rm \Sym^+(n)$ and $\Sym^{++}(n)$ denote the symmetric, positive semi-definite symmetric and positive definite symmetric second order tensors respectively.
	\item ${\rm GL}(n):=\{X\in\R^{n\times n}\;|\det{X}\neq 0\}$ denotes the general linear group.
	\item ${\rm GL}^+(n):=\{X\in\R^{n\times n}\;|\det{X}>0\}$ is the group of invertible matrices with positive determinant.
	\item $\mathrm{O}(n):=\{X\in {\rm GL}(n)\;|\;X^TX=\id\}$.
	\item ${\rm SO}(n):=\{X\in {\rm GL}(n,\R)\;|\; X^T X=\id,\;\det{X}=1\}$.
	\item $\mathfrak{so}(3):=\{X\in\mathbb{R}^{3\times3}\;|\;X^T=-X\}$ is the Lie-algebra of skew symmetric tensors.
	\item The set of positive real numbers is denoted by $\R_+:=(0,\infty)$, while $\overline{\R}_+=\R_+\cup \{\infty\}$.
	\end{itemize}
\textbf{Frequently used tensors}
	\begin{itemize}
	\item $F = \DD \varphi(x,t)$ is the Fréchet derivative (Jacobean) of the deformation $\varphi(\cdot,t) : \Omega_x \to \Omega_{\xi} \subset \R^3$. $\varphi(x,t)$ is usually assumed to be a diffeomorphism at every time $t \ge 0$ so that the inverse mapping $\varphi^{-1}(\cdot,t) : \Omega_{\xi} \to \Omega_x$ exists.
	\item $C=F^T \, F$ is the right Cauchy-Green strain tensor.
	\item $B=F\, F^T$ is the left Cauchy-Green (or Finger) strain tensor.
	\item $U = \sqrt{F^T \, F} \in \Sym^{++}(3)$ is the right stretch tensor, i.e.~the unique element of ${\rm Sym}^{++}(3)$ with $U^2=C$.
	\item $V = \sqrt{F \, F^T} \in \Sym^{++}(3)$ is the left stretch tensor, i.e.~the unique element of ${\rm Sym}^{++}(3)$ with $V^2=B$.
	\item $\log V = \frac12 \, \log B$ is the spatial logarithmic strain tensor or Hencky strain.
	\item $L = \dot F \cdot F^{-1} = \DD_\xi v(\xi)$ is the spatial velocity gradient.
	\item $v = \frac{\dif}{\dif t} \varphi(x, t)$ denotes the Eulerian velocity.
	\item $D = \sym \, L$ is the spatial rate of deformation, the Eulerian strain rate tensor.
	\item $W = \sk \, L$ is the vorticity tensor.
	\item We also have the polar decomposition $F = R \, U = V R \in {\rm GL}^+(3)$ with an orthogonal matrix $R \in \OO(3)$ (cf. Neff et al.~\cite{Neffpolardecomp}), see also \cite{LankeitNeffNakatsukasa,Neff_Nagatsukasa_logpolar13}.
	\end{itemize}
\noindent \textbf{The strain energy function $\WW(F)$} \\
\\
We are only concerned with rotationally symmetric functions $\WW(F)$ (objective and isotropic), i.e.
	\begin{equation*}
	\WW(F)={\WW}(Q_1^T\, F\, Q_2), \qquad \forall \, F \in {\rm GL}^+(3), \qquad  Q_1 , Q_2 \in {\rm SO}(3).
	\end{equation*}
\textbf{List of additional definitions and useful identities}
	\begin{itemize}
	\item For two metric spaces $X, Y$ and a linear map $L: X \to Y$ with argument $v \in X$ we write $L.v:=L(v)$. This applies to a second order tensor $A$ and a vector $v$ as $A.v$ as well as a fourth order tensor $\C$ and a second order tensor $H$ as $\C.H$.
	\item We denote the space of minor and major symmetric, positive definite fourth order tensors $\C$ by $\Sym^{++}_4(6)$, i.e.~$\C \in \Sym^{++}_4(6)$ if and only if $\langle \C.D, D \rangle > 0$ for all $D \in \Sym(3) \! \setminus \! \{0\}$.
	\item We define $J = \det{F}$ and denote by $\Cof(X) = (\det X)X^{-T}$ the cofactor of a matrix in ${\rm GL}^{+}(3)$.
	\item We define $\sym X = \frac12 \, (X + X^T)$ and $\sk X = \frac12 \, (X - X^T)$ as well as $\dev X = X - \frac13 \, \tr(X) \cdot \id$.
	\item For all vectors $\xi,\eta\in\R^3$ we have the tensor product $(\xi\otimes\eta)_{ij}=\xi_i\,\eta_j$.
	\item $S_1=\DD_F \WW(F) = \sigma \cdot \Cof F$ is the non-symmetric first Piola-Kirchhoff stress tensor.
	\item $S_2=F^{-1}S_1=2\,\DD_C \widetilde{\WW}(C)$ is the symmetric second  Piola-Kirchhoff stress tensor.
	\item $\sigma=\frac{1}{J}\,  S_1\, F^T=\frac{1}{J}\,  F\,S_2\, F^T=\frac{2}{J}\DD_B \widetilde{\WW}(B)\, B=\frac{1}{J}\DD_V \widetilde{\WW}(V)\, V = \frac{1}{J} \, \DD_{\log V} \widehat \WW(\log V)$ is the symmetric Cauchy stress tensor.
	\item $\sigma = \frac{1}{J} \, F\, S_2 \, F^T = \frac{2}{J} \, F \, \DD_C \widetilde{\WW}(C) \, F^T$ is the ''\textbf{Doyle-Ericksen formula}'' \cite{doyle1956}.
	\item For $\sigma: \Sym(3) \to \Sym(3)$ we denote by $\DD_B \sigma(B)$ with $\sigma(B+H) = \sigma(B) + \DD_B \sigma(B).H + o(H)$ the Fréchet-derivative. For $\sigma: \Sym^+(3) \subset \Sym(3) \to \Sym(3)$ the same applies. Similarly, for $\WW : \R^{3 \times 3} \to \R$ we have $\WW(X + H) = \WW(X) + \langle \DD_X \WW(X), H \rangle + o(H)$.
	\item $\tau = J \, \sigma = 2\, \DD_B \widetilde{\WW}(B)\, B $ is the symmetric Kirchhoff stress tensor.
	\item $\tau = \DD_{\log V} \widehat{\WW}(\log V)$ is the ``\textbf{Richter-formula}'' \cite{richter1948isotrope, richter1949hauptaufsatze}.
	\item $\sigma_i =\dd\frac{1}{\lambda_1\lambda_2\lambda_3}\dd\lambda_i\frac{\partial g(\lambda_1,\lambda_2,\lambda_3)}{\partial \lambda_i}=\dd\frac{1}{\lambda_j\lambda_k}\dd\frac{\partial g(\lambda_1,\lambda_2,\lambda_3)}{\partial \lambda_i}, \ \ i\neq j\neq k \neq i$ are the principal Cauchy stresses (the eigenvalues of the Cauchy stress tensor $\sigma$), where $g:\mathbb{R}_+^3\to \mathbb{R}$ is the unique function  of the singular values of $U$ (the principal stretches) such that $\WW(F)=\widetilde{\WW}(U)=g(\lambda_1,\lambda_2,\lambda_3)$.
	\item $\sigma_i =\dd\frac{1}{\lambda_1\lambda_2\lambda_3}\frac{\partial \widehat{g}(\log \lambda_1,\log \lambda_2,\log \lambda_3)}{\partial \log \lambda_i}$, where $\widehat{g}:\mathbb{R}^3\to \mathbb{R}$ is the unique function such that \\ \hspace*{0.3cm} $\widehat{g}(\log \lambda_1,\log \lambda_2,\log \lambda_3):=g(\lambda_1,\lambda_2,\lambda_3)$.
	\item $\tau_i =J\, \sigma_i=\dd\lambda_i\frac{\partial g(\lambda_1,\lambda_2,\lambda_3)}{\partial \lambda_i}=\frac{\partial \widehat{g}(\log \lambda_1,\log \lambda_2,\log \lambda_3)}{\partial \log \lambda_i}$ \, . 
	\end{itemize}
\subsection{Hilbert-monotonicity }\label{ips2}
Regarding Hilbert-monotonicity, we recall the following properties from Ghiba et al.~\cite{ghiba2024}.
\begin{definition} \cite{NeffMartin14}
A tensor function $\Sigma_f:{\rm Sym}^{++}(3)\to\Sym^{++}\,$ is called  {\bf strictly Hilbert-monotone} if
	\begin{align}
	\label{eq:introductionMatrixMonotonicity}
	\iprod{\Sigma_f(U)-\Sigma_f(\overline U),\,U-\overline U}_{\R^{3\times 3}}>0\qquad\forall\,U\neq\overline U\in{\rm Sym}^{++}(3)\,.
	\end{align}
We refer to this inequality as strict {\bf{Hilbert-space matrix-monotonicity}} of the tensor function $\Sigma_f$.  
\end{definition}
\begin{definition}\cite{NeffMartin14}
A vector function  $f:\R^3_+\to\R^3$ is {\bf strictly vector monotone} if 
	\begin{align}
	\iprod{f(\lambda)-f(\overline\lambda),\,\lambda-\overline\lambda}_{\R^3}>0\qquad\forall\lambda\neq\overline\lambda\in\R^3_+.
	\end{align}
\end{definition}
\begin{definition}
A differentiable function $f$ between finite-dimensional Hilbert spaces is called {\bf strongly monotone} if $\DD f$ is positive definite everywhere. Thus ``stronlgy'' implies ``strictly'', see e.g.~Remark \ref{remarkmon}.
\end{definition}
\noindent Note that for an arbitrary vector function $f: \R^3_+ \to \R^3$, ${\rm D}f\,(\lambda_1,\lambda_2,\lambda_3)$ in itself might not be symmetric. One main goal of a forthcoming paper \cite{MartinVossGhibaNeff}  is to proof the following result, thereby elucidating on Ogden's work \cite[last page in Appendix]{Ogden83}, based on the seminal contributions of Hill \cite{hill1968constitutivea,hill1968constitutiveb,hill1970constitutive}: 
\begin{thm}
	\label{theorem:mainResult}
	A symmetric function $f:\R^3_+\to\R^3$ is strictly (strongly) vector-monotone if and only if $\Sigma_f$ is strictly (strongly) matrix-monotone.
\end{thm}
\begin{rem} \label{remA5}
Theorem \ref{theorem:mainResult} is decisive for the equivalence
	\begin{align}
	\sym \DD_{\log V} \widehat \sigma(\log V) \in \Sym^{++}_4(6) \qquad \iff \qquad \sym \frac{\partial \widehat \sigma_i}{\partial \log \lambda_j} \in \Sym^{++}(3),
	\end{align}
where the $\widehat \sigma_i$ are the principal Cauchy stresses expressed as function of the principle logarithmic strains $\log \lambda_i$.
\end{rem}
\subsection{True-Stress-True-Strain monotonicity (TSTS-M)} \label{appendixtstsm}
\begin{definition}
We define three notions of {\bf True-Stress-True-Strain monotonicity} as follows
	\begin{equation}
	\begin{alignedat}{2}
	&\text{TSTS-M:} \qquad &\langle \widehat \sigma(\log V_1)- \widehat \sigma(\log V_2),\log V_1-\log V_2\rangle&\ge 0, \qquad \forall\, V_1, V_2\in {\rm Sym}^{++}(3), \ V_1\neq V_2, \\
	&\text{TSTS-M$^+$:} \qquad &\langle \widehat \sigma(\log V_1)- \widehat \sigma(\log V_2),\log V_1-\log V_2\rangle&> 0, \qquad \forall\, V_1, V_2\in {\rm Sym}^{++}(3), \ V_1\neq V_2, \\
	&\text{TSTS-M$^{++}$:} \qquad & \sym \, \DD_{\log V} \widehat \sigma(\log V) \in \Sym^{++}_4(6). &
	\end{alignedat}
	\end{equation}
Note, that this is equivalent to monotonicity of $\widehat \sigma$ in $\log B$ since $\log B=2\, \log V$.
\end{definition}
\noindent Regarding TSTS-M$^+$, we have the following properties from \cite{NeffGhibaLankeit}.
\begin{rem}\label{remarkmon}
Sufficient for TSTS-M$^+$ is Jog and Patil's \cite{jog2013conditions} constitutive requirement that
	\begin{align}
	\text{TSTS-M$^{++}$:} \qquad \sym \, \DD_{\log B}\,\widehat \sigma(\log B) \in \Sym^{++}_4(6),
	\end{align}
\end{rem}
\begin{proof}
Let us remark that for all $B_1, B_2\in {\rm Sym}^{++}(3)$ and $0\leq t\leq 1$, we have $2\,\log V_1=\log B_1, \, 2\,\log V_2=\log B_2$ and $t\, (\log
V_1-\log V_2)+\log V_2\in{\rm Sym}(3)$, where $V_1^2=B_1,\, V_2^2=B_2$ . Moreover, we have
	 \begin{align}\label{Joginespr}
	  \langle \widehat \sigma(\log B_1)&- \widehat \sigma(\log B_2),\log B_1-\log B_2\rangle=2\,\langle \widehat \sigma(2\,\log V_1)- \widehat \sigma(2\,\log V_2),\log V_1-\log
	  V_2\rangle\notag\\&=2\,\left\langle\left[\int_0^1 \frac{\rm d}{\rm dt}\, \widehat \sigma \bigg(2\,t\, (\log V_1-\log V_2)+2\,\log V_2\bigg) \dif t\right],\log V_1-\log
	  V_2\right\rangle\\&=4\,\int_0^1 \left\langle\left[\DD_{\log V}\, \widehat \sigma \bigg(2\,t\, (\log V_1-\log V_2)+2\,\log V_2\bigg).\,(\log V_1-\log V_2)\right],\log
	  V_1-\log V_2\right\rangle \dif t\,.\notag \\
	  &= 4\,\int_0^1 \left\langle\left[\sym \left( \DD_{\log V}\, \widehat \sigma \bigg(2\,t\, (\log V_1-\log V_2)+2\,\log V_2\bigg) \right) .\,(\log V_1-\log V_2)\right],\log
	  V_1-\log V_2\right\rangle \dif t\,.\notag
	 \end{align}
Where the last equation of \eqref{Joginespr} is due to the fact that 
for any skew symmetric matrix $A\in\mathfrak{so}(3)$, it always holds
$
\scal{A.v}{v}_{\bR^{3}}
=0
$
for every $v\in\bR^3$. Using that the integrand is non-negative, due to the assumption that $\Lambda = \sym \, \DD_{\log V} \widehat \sigma(\log V)$ is positive definite, the TSTS-M$^+$ condition
 follows.
\end{proof}
\begin{rem}
As an easy consequence of the previous remark we obtain the implications
	\begin{align}
	\text{TSTS-M$^{++}$} \qquad \implies \qquad \text{TSTS-M$^{+}$} \qquad \implies \qquad \text{TSTS-M}
	\end{align}
as well as the equivalence \qquad
	\fbox{
	\begin{minipage}[h!]{0.65\linewidth}
		\centering
		TSTS-M$^{++}$ \qquad $\iff$ \qquad corotational stability postulate (CSP).
	\end{minipage}}
\end{rem}
\subsection{Hill's inequality} \label{appendixhill}
In this Appendix we rederive\footnote
{
Many collegues agree that Hill's papers are difficult to read (to say the least). He is voluntarily skipping details of the mathematical development, but apparently Hill's final results have always been correct.
}
the result by Hill \cite{hill1968constitutivea} following the detailed method of Lagrangian axis for both the hyperelastic and the Cauchy-elastic case. Note that the result, which in our notation reads
	\begin{align}
	\langle \frac{\DD^{\ZJ}}{\DD t}[\tau], D \rangle > 0 \quad &\iff \quad \DD_{\log V} \widehat \tau(\log V) \in \Sym^{++}_4(6) \\
	& \implies \quad \langle \widehat \tau(\log V_1) - \widehat \tau (\log V_2), \log V_1 - \log V_2 \rangle > 0 \qquad \forall \, V_1, V_2 \in \Sym^{++}(3), \quad V_1 \neq V_2, \notag
	\end{align}
has already been stated by Hill in \cite{hill1968constitutivea}. However, it seems that his development is not rigorous so that we saw the need to provide a comprehensive and self-consistent proof. 
\subsubsection{The hyperelastic case}
Let us begin by repeating the calculations from Section \ref{sec:leblond} for the Kirchhoff stress $\tau$. First, we derive the corresponding quadratic form $Q_{\tau, \hyp}$. We obtain from \eqref{eqfirsteq01} with $\tau = J \, \sigma = F \, S_2 \, F^T$ instead of $\sigma$
	\begin{equation}
	\begin{alignedat}{2}
	\frac{\DD^{\ZJ}}{\DD t}[\tau] &= \frac{\dif}{\dif t}[\tau] + [\tau \ W - W \, \tau] = \frac{\dif}{\dif t}[F \, S_2 \, F^T] + [\tau \, W - W \, \tau] \\
	&= \dot{F} \, S_2 \, F^T + F \, \dot{S}_2 \, F^T + F \, S_2 \, \dot{F}^T + [\tau \, W - W \, \tau] \\
	&= F \, \dot{S}_2 \, F^T + \underbrace{\dot{F} \, F^{-1}}_{L} \, \underbrace{F \, S_2 \, F^T}_{\tau} + \underbrace{F \, S_2 \, F^T}_{\tau} \, \underbrace{F^{-T} \, \dot{F}^T}_{L^T} + [\tau \, W - W \, \tau] \\
	&= F \, \dot{S}_2 \, F^T + L \, \tau + \tau \, L^T + \tau \, W - W \, \tau = F \, \dot{S}_2 \, F^T + D \, \tau + \tau \, D \\
	&= F \, \dot{S}_2 \, F^T + D \, F \, S_2 \, F^T + F \, S_2 \, F^T \, D.
	\end{alignedat}
	\end{equation}
Next, we calculate $\left\langle \frac{\DD^{\ZJ}}{\DD t}[\tau], D \right\rangle$, using the identities
	\begin{equation}
	\dot{E} = \frac12 \, (\dot{F}^T \, F + F^T \, \dot{F}), \qquad \quad D = F^{-T} \, \dot{E} \, F^{-1},
	\end{equation}
which yields
	\begin{align}
	J \, \left\langle \frac{\DD^{\ZJ}}{\DD t}[\tau] , D \right\rangle &= \langle F \, \dot{S_2} \, F^T + D \, F \, S_2 \, F^T + F \, S_2 \, F^T \, D, D \rangle = \langle F \, \dot{S_2} \, F^T + F^{-T} \, \dot{E} \, S_2 \, F^T + F \, S_2 \, \dot{E} \, F^{-1} , F^{-T} \, \dot{E} \, F^{-1} \rangle \notag \\
	&= \langle F \, \dot{S_2} \, F^T , F^{-T} \, \dot{E} \, F^{-1} \rangle + \langle F^{-T} \, \dot{E} \, S_2 \, F^T , F^{-T} \, \dot{E} \, F^{-1} \rangle + \langle F \, S_2 \, \dot{E} \, F^{-1} , F^{-T} \, \dot{E} \, F^{-1} \rangle \notag \\
	\label{eqappendix01}
	&= \langle \dot{S_2} , \dot{E} \rangle + \langle \underbrace{F^{-1} \, F^{-T}}_{= \; (F^T F)^{-1}} \, \dot{E} \, S_2 , \dot{E} \, \underbrace{F^{-1} \, F}_{\id} \rangle + \langle \underbrace{F^{-1} \, F}_{\id} \, S_2 \, \dot{E} , \dot{E} \, \underbrace{F^{-1} \, F^{-T}}_{= \; (F^T F)^{-1}} \rangle \\
	&= \langle \dot{S_2} , \dot{E} \rangle + \langle C^{-1} \, \dot{E} \, S_2 \, \dot{E}, \id \rangle + \langle C^{-1} \, \dot{E} \, S_2 \, \dot{E} , \id \rangle = \langle \dot{S_2} , \dot{E} \rangle + 2 \, \langle C^{-1} \, \dot{E} \, S_2 \, \dot{E} , \id \rangle \notag \\
	&= \langle \dot{S_2} , \dot{E} \rangle + 2 \, \tr(C^{-1} \, \dot{E} \, S_2 \, \dot{E}) \, . \notag
	\end{align}
This leads to the definition of the quadratic form
	\begin{align}
	\label{eqtau1}
	Q_{\tau, \hyp}(\dot{E}):=  \langle \dot{S_2} , \dot{E} \rangle + 2 \, \tr(C^{-1} \, \dot{E} \, S_2 \, \dot{E}).
	\end{align}
Then we may use the orthonormal frame $\{U^i(x,t)\}_{i=1}^3$ of eigenvectors of $E(x,t) = \frac12 \, (C(x,t) - \id)$ with eigenvalues $\{e_i(x,t)\}_{i=1}^3$ and the function $\breve{\WW}$, to express all quantities from \eqref{eqappendix01} in terms of the orthonormal frame. This once again leads to the identities
	\begin{equation}
	\begin{alignedat}{2}
	\langle \dot{S}_2, \dot{E} \rangle
	&
	=
	\sum_{i,j=1}^3
	\frac{\partial^2\breve{\WW}}{\partial e_i\partial e_j} \, \dot{E}_{jj}\, \dot{E}_{ii}
	+
	\sum_{i\neq j}
	\frac{\frac{\partial\breve{\WW}}{\partial e_i}-\frac{\partial\breve{\WW}}{\partial e_j}}{e_i-e_j}\dot{E}_{ij}^2, \\
	\tr(C^{-1} \, \dot{E} \, S_2 \, \dot{E})
	&
	=
	\sum_{i=1}^3\frac{1}{\lambda_i^2}\dot E_{ii}^2\frac{\partial \breve{\WW}}{\partial e_i}
	+
	\frac{1}{2}
	\sum_{i\neq j}\partonbb{\frac{1}{\lambda_i^2}\frac{\partial \breve{\WW}}{\partial e_j}+\frac{1}{\lambda_j^2}\frac{\partial \breve{\WW}}{\partial e_i}} \dot E_{ij}^2.
	\end{alignedat}
	\end{equation}
Thus, the quadratic form defined by \eqref{eqtau1} is the sum of the two \textit{independent} quadratic forms, each of which must therefore be positive definite:
\begin{align}
    Q_{\tau, \hyp}(\dot E)
    =
    \underbrace{Q_{\tau, \hyp(1)}(\dot E_{11},\dot E_{22}, \dot E_{33})}_{Q_{\tau, \hyp(1)}(\dot e_{1},\dot e_{2}, \dot e_{3})}
    +
    Q_{\tau, \hyp(2)}(\dot E_{12},\dot E_{23}, \dot E_{31}),
\end{align}
where 
\begin{align}
\label{eqappendix02}
	Q_{\tau, \hyp(1)}(\dot E_{11},\dot E_{22}, \dot E_{33})
	&
	=
	\sum_{i,j=1}^3 \frac{\partial^2 \breve{\WW}}{\partial e_i\partial e_j} \dot E_{ii} \dot E_{jj}
	+
	\sum_{i=1}^3\frac{2}{\lambda_i^2}\frac{\partial \breve{\WW}}{\partial e_i}\dot E_{ii}^2 
	=
	\sum_{i=1}^3\partonbb{
		\frac{\partial^2 \breve{\WW}}{\partial e_i^2}+\frac{2}{\lambda_i^2}\frac{\partial \breve{\WW}}{\partial e_i}
	}\dot E_{ii}^2
	+
	\sum_{i\neq j}
		\frac{\partial^2 \breve{\WW}}{\partial e_i\partial e_j}
	\dot E_{ii}
	\dot{E}_{jj}
\end{align}
and 
\begin{align}
	Q_{\tau, \hyp(2)}(\dot E_{12},\dot E_{23}, \dot E_{31})
	= 
	\sum_{i\neq j}
	\partonbb{
		\frac{\frac{\partial\breve{\WW}}{\partial e_i}-\frac{\partial\breve{\WW}}{\partial e_j}}{e_i-e_j}
		+
		\frac{1}{\lambda_i^2}\frac{\partial \breve{\WW}}{\partial e_j}+\frac{1}{\lambda_j^2}\frac{\partial \breve{\WW}}{\partial e_i}
	}
	\dot{E}_{ij}^2.
\end{align}
With the same calculation as in \eqref{eqeigenvaluessigma} one can prove, that $u^j = F.U^j$ are eigenvectors of $\tau$ with eigenvalues
	\begin{align}
	\label{eqtau2}
	\tau_i(\lambda_1, \lambda_2, \lambda_3) = \widehat \tau_i (\log \lambda_1, \log \lambda_2, \log \lambda_3) = \lambda_i^2 \, \frac{\partial \breve{\WW}(e)}{\partial e_i} \qquad \text{for} \qquad i \in \{1,2,3\}.
	\end{align}
Additionally, we can use the relations derived from the equality $\widehat{\WW}(\log \lambda) = \breve{\WW}(e(\lambda))$, i.e.
	\begin{equation}
	\label{eqappendix03}
	\begin{alignedat}{2}
	\frac{\partial \breve{\WW}}{\partial e_i}(e(\lambda)) \, \lambda_i &= \frac{\partial \widehat{\WW}(\log \lambda)}{\partial (\log \lambda_i)} \cdot \frac{1}{\lambda_i} ,\\
	\frac{\partial^2 \breve \WW}{\partial e_i^2}(e(\lambda)) + \frac{1}{\lambda_i^2} \, \frac{\partial \breve \WW}{\partial e_i}(e(\lambda)) &= \frac{1}{\lambda_i^4} \left( \frac{\partial^2 \widehat \WW (\log \lambda)}{\partial (\log \lambda_i)^2} - \frac{\partial \widehat \WW (\log \lambda)}{\partial (\log \lambda_i)} \right), \\
	\frac{\partial^2 \breve \WW}{\partial e_i \, \partial e_j} &= \frac{1}{\lambda_i^2 \, \lambda_j^2} \, \frac{\partial^2 \widehat \WW (\log \lambda)}{\partial (\log \lambda_i) \, \partial (\log \lambda_j)},
	\end{alignedat}
	\end{equation}
to obtain with \eqref{eqtau2}
	\begin{equation}
	\label{eqappendix04}
	\begin{alignedat}{4}
	\frac{\partial \breve \WW (e(\lambda))}{\partial e_i} &= \frac{1}{\lambda_i^2} \, \widehat \tau_i, \qquad \qquad
	 &\frac{\partial \widehat \WW (\log \lambda)}{\partial (\log \lambda_i)} &= \widehat \tau_i \, ,\\
	\frac{\partial^2 \widehat \WW(\log \lambda)}{\partial (\log \lambda_i)^2} &= \frac{\partial \widehat \tau_i}{\partial (\log \lambda_i)}, \qquad \qquad
	&\frac{\partial^2 \widehat \WW (\log \lambda)}{\partial (\log \lambda_i) \, \partial(\log \lambda_j)} &= \frac{\partial \widehat \tau_i}{\partial (\log \lambda_j)} = \frac{\partial \widehat \tau_j}{\partial (\log \lambda_i)}.
	\end{alignedat}
	\end{equation}
In the last equation of \eqref{eqappendix04} we used that $\tau = \DD_{\log V}\widehat \WW(\log V)$ (hyperelasticity) and thus the partial derivatives commute. This property is lost, when discussing the Cauchy-elastic case which is why we need to consider $\sym \DD_{\log \lambda} \widehat \tau$ instead in this case. Pairing \eqref{eqappendix03} with \eqref{eqappendix04} we obtain for \eqref{eqappendix02}
	\begin{equation}
	\begin{alignedat}{2}
	\frac{\partial^2 \breve{\WW}}{\partial e_i^2}+\frac{2}{\lambda_i^2}\frac{\partial \breve{\WW}}{\partial e_i} &= \frac{\partial^2 \breve{\WW}}{\partial e_i^2}+\frac{1}{\lambda_i^2}\frac{\partial \breve{\WW}}{\partial e_i} + \frac{1}{\lambda_i^2}\frac{\partial \breve{\WW}}{\partial e_i} = \frac{1}{\lambda_i^4} \left( \frac{\partial^2 \widehat \WW (\log \lambda)}{\partial (\log \lambda_i)^2} - \frac{\partial \widehat \WW (\log \lambda)}{\partial (\log \lambda_i)} \right) + \frac{1}{\lambda_i^2}\frac{\partial \breve{\WW}}{\partial e_i} \\
	&= \frac{1}{\lambda_i^4} \left( \frac{\partial \widehat \tau_i}{\partial (\log \lambda_i)} - \widehat \tau_i \right) + \frac{1}{\lambda_i^4} \, \widehat \tau_i = \frac{1}{\lambda_i^4} \, \frac{\partial \widehat \tau_i}{\partial (\log \lambda_i)}, \\
	\frac{\partial^2 \breve{\WW}}{\partial e_i\partial e_j} &= \frac{1}{\lambda_i^2 \, \lambda_j^2} \, \frac{\partial^2 \widehat \WW (\log \lambda)}{\partial (\log \lambda_i) \, \partial (\log \lambda_j)} = \frac{1}{2 \, \lambda_i^2 \, \lambda_j^2} \, \left( \frac{\partial \widehat \tau_i}{\partial (\log \lambda_j)} + \frac{\partial \widehat \tau_j}{\partial (\log \lambda_i)}\right).
	\end{alignedat}
	\end{equation}
Thus we may rewrite \eqref{eqappendix02} as
	\begin{align}
	Q_{\tau, \hyp(1)}(\dot{E}_{11},\dot{E}_{22},\dot{E}_{33}) 
	&= 
	\frac{1}{2} \, \sum_{i,j} \left(\frac{\partial \widehat \tau_i}{\partial (\log \lambda_j)} 
	+ 
	\frac{\partial \widehat \tau_j}{\partial (\log \lambda_i)}\right) \, \frac{\dot{E}_{ii}}{\lambda_i^2} \, \frac{\dot{E}_{jj}}{\lambda_j^2} \\
	&\hspace*{-100pt}=
	\left \langle \begin{pmatrix} \dot E_{11} \\ \dot E_{22} \\ \dot E_{33} \end{pmatrix} \!,\!
	\begin{pmatrix}
	\frac{1}{\lambda_1^4} \, \frac{\partial \widehat \tau_1}{\partial (\log \lambda_1)} 
	& \frac{1}{2 \, \lambda_1^2 \, \lambda_2^2} \, \left( \frac{\partial \widehat \tau_1}{\partial (\log \lambda_2)} + \frac{\partial \widehat \tau_2}{\partial (\log \lambda_1)} \right) 
	& \frac{1}{2 \, \lambda_1^2 \, \lambda_3^2} \, \left( \frac{\partial \widehat \tau_1}{\partial (\log \lambda_3)} + \frac{\partial \widehat \tau_3}{\partial (\log \lambda_1)} \right) \\
	\frac{1}{2 \, \lambda_1^2 \, \lambda_2^2} \, \left( \frac{\partial \widehat \tau_1}{\partial (\log \lambda_2)} + \frac{\partial \widehat \tau_2}{\partial (\log \lambda_1)} \right) 
	& \frac{1}{\lambda_2^4} \, \frac{\partial \widehat \tau_2}{\partial (\log \lambda_2)} 
	& \frac{1}{2 \, \lambda_2^2 \, \lambda_3^2} \, \left( \frac{\partial \widehat \tau_2}{\partial (\log \lambda_3)} + \frac{\partial \widehat \tau_3}{\partial (\log \lambda_2)} \right) \\
	\frac{1}{2 \, \lambda_1^2 \, \lambda_3^2} \, \left( \frac{\partial \widehat \tau_1}{\partial (\log \lambda_3)} + \frac{\partial \widehat \tau_3}{\partial (\log \lambda_1)} \right) 
	& \frac{1}{2 \, \lambda_2^2 \, \lambda_3^2} \, \left( \frac{\partial \widehat \tau_2}{\partial (\log \lambda_3)} + \frac{\partial \widehat \tau_3}{\partial (\log \lambda_2)} \right)
	& \frac{1}{\lambda_3^4} \, \frac{\partial \widehat \tau_3}{\partial (\log \lambda_3)} 
	\end{pmatrix} \!.\!
	\begin{pmatrix} \dot E_{11} \\ \dot E_{22} \\ \dot E_{33} \end{pmatrix} \right\rangle. \notag
	\end{align}
Furthermore, we note $Q_{\tau, \hyp(2)} = Q_{\hyp(2)}$, so that we may use the identity $J \widehat \sigma_i = \widehat \tau_i$ and the calculation from \eqref{eqQ2forsigma}, to conclude
\begin{align}
	\hspace*{-8pt}
	Q_{\tau, \hyp(2)}(\dot E_{12},\dot E_{23}, \dot E_{31})
	=
	Q_{\hyp(2)}
	=
	J \,
	\sum_{i\neq j}
	\frac{  \widehat \sigma_i - \widehat \sigma_j  
	}{(\lambda_i^2-\lambda_j^2)}
	\partonbb{
	\frac{1  
	}{ \lambda_i^2 }
	+
	\frac{1  
	}{ \lambda_j^2}
    }
	\, \dot{E}_{ij}^2
	=
	\sum_{i\neq j}
	\frac{  \widehat \tau_i - \widehat \tau_j  
	}{(\lambda_i^2-\lambda_j^2)}
	\underbrace{\partonbb{
	\frac{1  
	}{ \lambda_i^2 }
	+
	\frac{1  
	}{ \lambda_j^2}
    }
	\, \dot{E}_{ij}^2}_{>0}.
\end{align}
This proves, that the results for the Cauchy stress $\sigma$ also hold true for the Kirchhoff stress $\tau$ in the hyperelastic case.
\subsubsection{The Cauchy-elastic case}
Proceeding with the more general Cauchy-elastic case, we observe that by the construction in Proposition \ref{eigenbasis}, we can once again express all relevant quantities in the quadratic from
	 \begin{align}
	\label{eqtau3}
	Q_{\tau, \ela}(\dot{E}):=  \langle \dot{S_2} , \dot{E} \rangle + 2 \, \tr(C^{-1} \, \dot{E} \, S_2 \, \dot{E}),
	\end{align}
with respect to the eigenvalues and eigenvectors of $E$, leading to the relations (cf. Prop \ref{propeqall.2})
\begin{equation}
\begin{alignedat}{2}
\scalb{\dot S_2}{\dot E}
	&
	=
	\sum_{i,j=1}^3
	\frac{\partial \breve s_i}{\partial e_j}
	\, \dot{E}_{jj} \, \dot{E}_{ii}
	+
	\sum_{i\neq j}
	\frac{\breve s_i-\breve s_j}{e_i-e_j}
	\dot{E}_{ij}^2, \\
\tr\partonb{C^{-1}\dot E \, S_2\,\dot E}
	&
	=
	\sum_{i=1}^3\frac{1}{\lambda_i^2} \, \dot E_{ii}^2 \, \breve s_i
	+
	\frac{1}{2}
	\sum_{i\neq j}
	\partonbb{
		\frac{\breve s_j}{\lambda_i^2} 
		+
		\frac{\breve s_i }{\lambda_j^2}  
	} \dot E_{ij}^2
\end{alignedat}
\end{equation}
and thus yielding the representation
\begin{align}
    Q_{\tau, \ela}(\dot E)
    =
    Q_{\tau, \ela(1)}(\dot E_{11},\dot E_{22}, \dot E_{33})
    +
    Q_{\tau, \ela(2)}(\dot E_{12},\dot E_{23}, \dot E_{31}),
\end{align}
with
\begin{align}
	Q_{\tau, \ela(1)} = \sum_{i=1}^3 \left(\frac{\partial \breve s_i}{\partial e_i} + 2 \, \frac{\breve s_i}{\lambda_i^2}\right) \, \dot E_{ii}^2 + \sum_{i \neq j} \frac{\partial \breve s_i}{\partial e_j} \, \dot E_{ii} \, \dot E_{jj}
\end{align}
and
\begin{align}
	Q_{\tau, \ela(2)} = \sum_{i \neq j} \partonbb{\frac{\breve s_i - \breve s_j}{e_i - e_j} +
		\frac{\breve s_j}{\lambda_i^2} 
		+
		\frac{\breve s_i }{\lambda_j^2}  
	}  \, \dot E^2_{ij}.
\end{align}
Now, we define
\begin{equation}
\begin{alignedat}{2}
	\widehat Q_{\tau, \ela(1)}
	&=
	\begin{pmatrix}
		\frac{\partial \breve s_1}{\partial e_1}
		+
		2 \, \frac{\breve s_1}{\lambda_1^2}
		&
		\frac{\partial \breve s_1}{\partial e_2}
		&
		\frac{\partial \breve s_1}{\partial e_3}
		\\[2mm]
		\frac{\partial \breve s_2}{\partial e_1}
		&
		\frac{\partial \breve s_2}{\partial e_2}
		+
		2 \, \frac{\breve s_2}{\lambda_2^2}
		&
		\frac{\partial \breve s_2}{\partial e_3}
		\\[2mm]
		\frac{\partial \breve s_3}{\partial e_1}
		&
		\frac{\partial \breve s_3}{\partial e_2}
		&
		\frac{\partial \breve s_3}{\partial e_3}
		+
		2 \, \frac{\breve s_3}{\lambda_3^2}
	\end{pmatrix}, \qquad \text{with} \\
	\sym \, \widehat Q_{\tau, \ela(1)}
	&=
	\begin{pmatrix}
		\frac{\partial \breve s_1}{\partial e_1}
		+
		2 \, \frac{\breve s_1}{\lambda_1^2}
		&
		\frac12 \left(\frac{\partial \breve s_1}{\partial e_2} + \frac{\partial \breve s_2}{\partial e_1}\right)
		&
		\frac12 \left(\frac{\partial \breve s_1}{\partial e_3} + \frac{\partial \breve s_3}{\partial e_1}\right)
		\\[2mm]
		\frac12 \left(\frac{\partial \breve s_2}{\partial e_1} + \frac{\partial \breve s_1}{\partial e_2}\right)
		&
		\frac{\partial \breve s_2}{\partial e_2}
		+
		2 \, \frac{\breve s_2}{\lambda_2^2}
		&
		\frac12 \left(\frac{\partial \breve s_2}{\partial e_3} + \frac{\partial \breve s_3}{\partial e_2}\right)
		\\[2mm]
		\frac12 \left(\frac{\partial \breve s_3}{\partial e_1} + \frac{\partial \breve s_1}{\partial e_3}\right)
		&
		\frac12 \left(\frac{\partial \breve s_3}{\partial e_2} + \frac{\partial \breve s_2}{\partial e_3}\right)
		&
		\frac{\partial \breve s_3}{\partial e_3}
		+
		2 \, \frac{\breve s_3}{\lambda_3^2}
	\end{pmatrix}
\end{alignedat}
\end{equation}
and show that
\begin{equation} \label{eqtausym01}
	\scal{\begin{pmatrix}
			\frac{\dot E_{11}}{\lambda_1^2} \\[1mm] \frac{\dot E_{22}}{\lambda_2^2} \\[1mm] \frac{\dot E_{33}}{\lambda_3^2} 
	\end{pmatrix}\!\!}{ 
		\sym \pD_{\log\lambda}\widehat\tau
		. \!\!
		\begin{pmatrix}
			\frac{\dot E_{11}}{\lambda_1^2} \\[1mm] \frac{\dot E_{22}}{\lambda_2^2} \\[1mm] \frac{\dot E_{33}}{\lambda_3^2} 
	\end{pmatrix}}
	=
	\scal{
		\begin{pmatrix}
			\dot E_{11} \\ \dot E_{22} \\ \dot E_{33} 
	\end{pmatrix}\!\!}{ \sym \widehat Q_{\tau, \ela(1)} . \!\!\begin{pmatrix}
			\dot E_{11} \\ \dot E_{22} \\ \dot E_{33} 
	\end{pmatrix}}_{\! \! \! \bR^{3}}
	=
	\scal{
		\begin{pmatrix}
			\dot E_{11} \\ \dot E_{22} \\ \dot E_{33} 
	\end{pmatrix}\!\!}{ \widehat Q_{\tau, \ela(1)} . \!\!\begin{pmatrix}
			\dot E_{11} \\ \dot E_{22} \\ \dot E_{33} 
	\end{pmatrix}}_{\! \! \! \bR^{3}},    
\end{equation}
where the last equality holds because for any skew symmetric matrix $A \in \mathfrak{so}(3)$, we have $\langle A.v, v \rangle_{\R^3} = 0$ for every $v \in \R^3$. \\
\\
To show \eqref{eqtausym01}, recall the identities \eqref{new_ident}, given by
\begin{equation}
	\frac{\partial \widehat s_i(\log \lambda)}{\partial (\log \lambda_j)}
	=
	\frac{ \partial \breve s_i}{\partial e_j}\partonb{e(\lambda)} \; \lambda_j^2
	\qquad
	\qquad
	\forall i,j\in\{1,2,3\},
\end{equation}
as well as \eqref{eqrelationfortau}, which was given by
\begin{equation}
	J\,\widehat\sigma_i(\log \lambda)
	=
	\lambda_{i}^2 
	\,
	\widehat s_i(\log \lambda)
	\qquad
	\Longleftrightarrow
	\qquad
	\widehat\sigma_i(\log \lambda)
	=
	\frac{\lambda_{i}^2}{J}
	\widehat s_i(\log \lambda)
	\qquad
	\forall i\in\{1,2,3\}.
\end{equation}
Now, when deriving the components of $\sym \, \DD_{\log \lambda} \widehat \tau(\log \lambda)$, we obtain for the derivatives $\frac{\partial \widehat \tau_i}{\partial \log \lambda_i}$ with \newline $\widehat \tau(\log \lambda) = \tau(\lambda) = J \, \sigma(\lambda) = J \, \widehat \sigma(\log \lambda)$ and $\widehat s_i(\log \lambda) = \breve s_i(e(\lambda))$
	\begin{align}
	\label{eqsym01}
	\frac{\partial \widehat \tau_i}{\partial \log \lambda_i} = \frac{\partial}{\partial \log \lambda_i} \left(\mathrm{e}^{2 \, \log \lambda_i} \cdot \widehat s_i(\log \lambda) \right) = 2 \, \lambda_i^2 \, \widehat s_i + \lambda_i^2 \, \frac{\partial \widehat s_i}{\partial \log \lambda_i} = 2 \, \lambda_i^2 \, \breve s_i + \lambda_i^4 \, \frac{\partial \breve s_i}{\partial e_i}
	\end{align}
and similarly for $\frac{\partial \widehat \tau_i}{\partial \log \lambda_j}$
	\begin{align}
	\frac{\partial \widehat \tau_i}{\partial \log \lambda_j} = \lambda_i^2 \, \frac{\partial \widehat s_i(\log \lambda)}{\partial \log \lambda_j} = \lambda_i^2 \, \lambda_j^2 \, \frac{\partial \breve s_i}{\partial e_j}.
	\end{align}
While the diagonal entries of $\sym \, \DD_{\log \lambda} \widehat \tau(\log \lambda)$ are simply given by \eqref{eqsym01}, we obtain for the off-diagonal elements
	\begin{align}
	\label{eqsym02}
	\frac12 \left( \frac{\partial \widehat \tau_i}{\partial \log \lambda_j} + \frac{\partial \widehat \tau_j}{\partial \log \lambda_i} \right) = \frac12 \cdot \lambda_i^2 \, \lambda_j^2 \cdot \left(\frac{\partial \breve s_i}{\partial e_j} + \frac{\partial \breve s_j}{\partial e_i}\right).
	\end{align}
Thus, by combining \eqref{eqsym01} with \eqref{eqsym02}, we have proven that \eqref{eqtausym01} holds. \\
\\
Finally, we observe that, as for the hyperelastic case, we have $Q_{\tau, \ela(2)} = Q_{\ela(2)}$ and hence with $\widehat \tau = J \, \widehat \sigma$ we conclude for the calculation in \eqref{eqlastone1}
\begin{equation}
\begin{alignedat}{2}
	Q_{\tau, \ela(2)}(\dot E_{12},\dot E_{23}, \dot E_{31}) = Q_{\ela(2)}
	&
	= 
	\sum_{i\neq j}
	\partonbb{
		\frac{
			\breve s_i
			-
			\breve s_j
		}
		{e_i-e_j}
		+
		\frac{\breve s_j}{\lambda_i^2}
		+
		\frac{\breve s_i}{\lambda_j^2}
	}
	\dot{E}_{ij}^2
	=
	J
	\sum_{i\neq j}
	\frac{  \widehat\sigma_i - \widehat\sigma_j  
	}{(\lambda_i^2-\lambda_j^2)}
	\partonbb{
			\frac{1  
			}{ \lambda_i^2 }
			+
			\frac{1  
			}{ \lambda_j^2}
		}
		\, \dot{E}_{ij}^2
	\\
	&
	=
	\sum_{i\neq j}
	\frac{  \widehat\tau_i - \widehat\tau_j  
	}{(\lambda_i^2-\lambda_j^2)}
	\underbrace{\partonbb{
			\frac{1  
			}{ \lambda_i^2 }
			+
			\frac{1  
			}{ \lambda_j^2}
		}
		\, \dot{E}_{ij}^2}_{>0},
\end{alignedat}
\end{equation}
concluding the analysis for the Kirchhoff stress $\tau$ with the same results as those for the Cauchy stress $\sigma$.

\end{appendix}
\end{document}